\documentclass[12pt]{article}
\usepackage[english]{babel}
\usepackage{amsmath}
\usepackage{amssymb}
\usepackage{amsthm}
\usepackage{pb-diagram}
\usepackage{mathrsfs}
\usepackage{latexsym}
\usepackage{amsfonts}
\usepackage{latexsym}
\usepackage{dsfont}
\usepackage{algorithm}
\numberwithin{equation}{section}

\theoremstyle{definition}
\newtheorem{teo}{Theorem}[subsection]
\newtheorem{df}[teo]{Definition}
\newtheorem{cor}[teo]{Corollary}
\newtheorem{oss}[teo]{Remark}
\newtheorem{prop}[teo]{Proposition}
\newtheorem{lem}[teo]{Lemma}
\newtheorem{es}[teo]{Example}

\newcommand{\rcE}{{{\mathbb D}_E}}
\newcommand{\rcCone}{{\mathbb D}_{\mathbb C}}
\newcommand{\rcM}{{{\mathbb D}_M}}

\newcommand{\rcDE}{{{\mathbb D}_{E^*}}}
\newcommand{\rcDM}{{\mathbb D}_{M^*}}

\newcommand{\lind}[1]{\underset{#1}{\underrightarrow{\lim}}}
\newcommand{\hexpa}{{\mathscr{A}}^{\operatorname{exp}}_{\rcM}}
\newcommand{\hexpv}{{\mathscr{V}}^{\operatorname{exp}}_{\rcM}}
\newcommand{\hexpo}{{\mathscr{O}}^{\operatorname{exp}}_{\rcE}}
\newcommand{\hexppo}[1]{{\mathscr{O}}^{\operatorname{exp},({#1})}_{\rcE}}

\newcommand{\bexpo}{{\mathscr{B}}^{\operatorname{exp}}_{\rcM}}
\newcommand{\bexppo}{{\mathscr{B}}^{\operatorname{exp},(p)}_{\rcM}}

\newcommand{\oexp}{{\mathscr{O}}^{\operatorname{exp}}_{\rcE}}

\newcommand{\hinfo}{{\mathscr{O}}^{\operatorname{inf}}_{E^*_\infty}}
\newcommand{\hhinfo}[1]{{\mathscr{O}}^{\operatorname{inf}\!-\!{#1}}_{E^*_\infty}}

\newcommand{\bvth}{\overline{\vartheta}}
\newcommand{\QDC}[3]{C^{#2}({#3},\,{#3}')(\mathscr{Q}_{\rcE}^{(#1,\bullet)})}
\newcommand{\QQDC}[2]{C({#2},\,{#2}')(\mathscr{Q}_{\rcE}^{(#1,\bullet)})}
\newcommand{\QDRC}[2]{C^{#1}({#2},\,{#2}')(\mathscr{Q}^{(\bullet)}_{\rcE})}
\newcommand{\QQDRC}[1]{C({#1},\,{#1}')(\mathscr{Q}^{(\bullet)}_{\rcE})}
\newcommand{\HQDC}[3]{{{\textrm H}^{#2}(\QQDC{#1}{#3})}}
\newcommand{\HQDRC}[2]{{{\textrm H}^{#1}(\QQDRC{#2})}}

\newcommand{\ILHU}{\hat{\textrm H}^n(\hexpo(\mathcal{W}({U})))}
\newcommand{\ILHKU}{\hat{\textrm H}_K^n(\hexpo(\mathcal{W}({U})))}
\newcommand{\ilc}{\left(\dfrac{1}{2\pi\sqrt{-1}}\right)^n}
\newcommand{\IL}{\mathcal{I}\!\mathcal{L}}
\newcommand{\LL}{\mathcal{L}}
\newcommand{\LQ}{\mathscr{L}\!\mathscr{Q}}
\newcommand{\EQ}{\mathscr{E}\!\mathscr{Q}}
\newcommand{\LO}{\mathscr{L}\!\mathscr{O}^{\textrm{exp}}}
\newcommand{\EO}{\mathscr{E}\!\mathscr{O}^{\textrm{exp}}}
\newcommand{\pdiff}[1]{\dfrac{\partial}{\partial {#1}}}
\newcommand{\HHPC}[1]{\textrm{N}^*_{pc}({#1})}
\newcommand{\CatM}[1]{\mathrm{Mod}(\mathbb{Z}_{#1})}
\newcommand{\CatC}[1]{{\mathbf{\mathbf{C}}^+}(\mathrm{Mod}(\mathbb{Z}_{#1}))}
\newcommand{\CatK}[1]{{\mathbf{\mathbf{K}}^+}(\mathrm{Mod}(\mathbb{Z}_{#1}))}
\newcommand{\CatD}[1]{{\mathbf{\mathbf{D}}^+}(\mathrm{Mod}(\mathbb{Z}_{#1}))}

\newcommand{\CatMD}{{\mathbf{\mathbf{D}}^+}(\mathrm{Mod}(\mathbb{Z}))}

\title{Laplace hyperfunctions via \v{C}ech Dolbeault cohomology}
\author{Naofumi Honda\thanks{Department of Mathematics, Hokkaido University, Sapporo 060-0810, Japan}
	\, and \, Kohei Umeta\thanks{Department of General Education, College of
Science and Technology, Nihon University, Chiba 274-8501, Japan}}

\begin{document}
\maketitle
\begin{abstract}
	The paper studies several properties of Laplace hyperfunctions introduced by
	H.~Komatsu in the one dimensional case and by the authors in the 
	higher dimensional cases
	from the viewpoint of the relative \v{C}ech Dolbeault cohomology theory, which
	enables us, for example, to construct the Laplace transformation and its inverse
in a simple and systematic way. We also give some applications to a system of
PDEs with constant coefficients.
\end{abstract}

\tableofcontents

\section{Introduction}

A Laplace hyperfunction on the one dimensional space was first introduced
by H.~Komatsu (\cite{h1}, \cite{h4}, \cite{h5}) to justify the operational calculus for
arbitrary functions without any growth condition at infinity.
After his great success of the one dimensional Laplace hyperfunctions, 
the authors of this paper
established an Oka type vanishing theorem (Theorem 3.7 \cite{hu1}) and an edge of the wedge type
theorem (Theorem 3.12 \cite{hu2})
for the sheaf of holomorphic functions of several variables with exponential growth at infinity.
Thanks to these fundamental theorems, we were succeeded in defining 
the sheaf $\mathscr{B}^{\operatorname{exp}}$ of Laplace hyperfunction of several variables
as a local cohomology groups along the radial compactification 
${\mathbb D}_{\mathbb{R}^n} = \mathbb{R}^n \sqcup S^{n-1}$
of $\mathbb{R}^n$ with coefficients 
in the sheaf $\mathscr{O}^{\operatorname{exp}}$ 
of holomorphic functions with exponential growth,
and also showing that $\mathscr{B}^{\operatorname{exp}}$ is a soft sheaf (Corollary 5.9 \cite{hu2}).

Since a Laplace hyperfunction is defined as an element of the local cohomology group,
to understand its concrete expression we need some interpretation of the local cohomology group, which is done by usually considering its \v{C}ech representation through the relative
\v{C}ech cohomology group or more generally its ``intuitive representation'' 
introduced in \cite{KU} Section 4 (see Subsection \ref{subsec:intiuitve} also).

Recently T.~Suwa in \cite{Suwa01} and \cite{Suwa02} proposed another method to compute a local cohomology group
by using a soft resolution of a coefficient sheaf, 
which is called the relative \v{C}ech Dolbeault cohomology.
This implies, in particular, 
the sheaf of Sato's hyperfunction can be computed with the famous Dolbeault resolution
of holomorphic functions by using the relative \v{C}ech Dolbeault cohomology theory. 
In fact, N.~Honda, T.~Izawa and T.~Suwa \cite{HIS} studies Sato's hyperfunctions
from the viewpoint of \v{C}ech Dolbeault cohomology theory and finds that
several operations to a hyperfunction such as the integration of a hyperfunction 
along fibers, etc.~have very simple and easily understandable descriptions
in this framework because a hyperfunction is represented by a pair 
$(\mu_1,\, \mu_{01})$ of $C^\infty$-differential forms.

The purpose of this paper is to study Laplace hyperfunctions
from the viewpoint of \v{C}ech Dolbeault cohomology theory, which gives us
several advantages to their treatments like the case of Sato's hyperfunctions.
To make this point more clear,
we briefly explain, as such an example, an inverse Laplace transformation $\IL$
in the framework of \v{C}ech Dolbeault cohomology: It is given by
a quite simple form (see Definition \ref{def:IL} for details)
$$
\IL_\omega(f)
=\left[\ilc
\int_{\gamma^*} \rho(\omega)(\mathrm{Im}\,\zeta/|\mathrm{Im}\,\zeta|,\,z)\, e^{\zeta z} f(\zeta) d\zeta\right],
$$
where $\gamma^*$ is an appropriate real $n$-dimensional chain asymptotic to $\sqrt{-1}\mathbb{R}^n$ and
a pair $\rho(\omega)(\theta,z)$ of $C^\infty$-differential forms represents, roughly speaking, the constant function $1$ in the relative \v{C}ech Dolbeault cohomology group on $S^{n-1} \times {\mathbb D}_{\mathbb{C}^n}$ which also satisfies the support condition
$$
{\mathrm{supp}}(\rho(\omega))
\subset 
\widehat{\,\,\,\,}
\left\{(\theta,z) \in S^{n-1} \times \mathbb{C}^n;\, \langle \theta,\,\operatorname{Im} z \rangle > 0
\right\} \subset 
S^{n-1} \times {\mathbb D}_{\mathbb{C}^n}.
$$
Here ${\mathbb D}_{\mathbb{C}^n} = \mathbb{C}^n \sqcup S^{2n-1}$ 
is the radial compactification of $\mathbb{C}^n$,
and see Definition \ref{def:def_hat} for the symbol ${}^\wedge(\bullet)$.
Note that the above support condition for $\rho(\omega)$ guarantees the convergence 
of the integral.
The existence of such a kernel $\rho(\omega)e^{\zeta z}$ 
with the desired support condition 
is crucial in the definition of the inverse Laplace transformation, 
which comes from the fact that in \v{C}ech Dolbeault cohomology group
the support of a representative can be cut off in a desired way.

Furthermore, as was seen in the proof of Lemma \ref{lem:IL-path-independent},
we can estimate the support of a Laplace hyperfunction $\IL_\omega(f)$
by using the fact that
any derivative of $\rho(\omega)$ becomes zero as a cohomology class
because it is cohomologically constant.
In this way \v{C}ech Dolbeault cohomology theory gives us several new ideas and methods
in a study of Laplace hyperfunctions.

\

The paper is organized as follows: In Section 2, after a short review of \v{C}ech Dolbeault cohomology theory, we introduce several geometrical notations which are used
through the whole paper. Then we establish the fundamental de-Rham and Dolbeault theorems
in Section 3 and give the definition of the sheaf of Laplace hyperfunctions in Section 4.
We also give several expressions of Laplace hyperfunctions via \v{C}ech cohomology and
\v{C}ech Dolbeault cohomology in the same section.
The one of important facts in hyperfunction theory is the notion of boundary values
of holomorphic functions. We construct a boundary value morphism for Laplace hyperfunctions
in Section 5.  
The Laplace transformation and its inverse
in the framework of \v{C}ech Dolbeault cohomology are defined in Sections 6 and 7, and
the fact that they are inverse to each other is shown in Section 8.
The last section gives some applications to a system of PDEs with constant coefficients.

\section{Preparations}
Through the paper, we use the language of the derived categories:
Notations $\mathrm{Mod}(\mathbb{Z})$,
$\CatM{X}$, 
$\CatC{X}$,
$\CatK{X}$,
$\CatD{X}$, etc.~are the same as those in the book \cite{KS}, 
for example, 
$\mathrm{Mod}(\mathbb{Z})$ denotes the category of abelian groups,
$\CatM{X}$ the category of sheaves on $X$
of abelian groups,
$\CatC{X}$ the category of complexes bounded below of
sheaves on $X$ of abelian groups, and $\CatD{X}$ is 
the subcategory consisting of complexes bounded below
of the derived category of $\CatM{X}$. We sometimes write $\mathscr{F} \in \CatM{X}$
 instead of $\mathscr{F} \in \mathrm{Ob}(\CatM{X})$.

\subsection{A relative \v{C}ech derived complex}{\label{subsec:cheh-dolbeault-complex}}
{\label{subsec:relative_chec}}
In this subsection, we briefly recall the definition of a relative 
\v{C}ech derived complex.  For details, refer the readers to \cite{HIS}.
Let $X$ be a locally compact and $\sigma$-compact Hausdorff space and $K$ its closed subset, and
let $\mathcal{S}=\{U_i\}_{i \in \Lambda}$ be a finite open covering of $X$ and
$\Lambda'$ a subset of $\Lambda$ such that
$\mathcal{S}'=\{U_i\}_{i \in \Lambda'}$ ($\Lambda' \subset \Lambda$) becomes
an open covering of $X \setminus K$.
For $\alpha = (\alpha_0,\alpha_1,\cdots,\alpha_k) \in \Lambda^{k+1}$, we set
$$
U_\alpha = U_{\alpha_0} \cap U_{\alpha_1} \cap \cdots \cap U_{\alpha_k}.
$$

Let $\mathscr{F} \in \CatM{X}$. We denote by $C(\mathcal{S},\mathcal{S}';\,\mathscr{F})$
the relative \v{C}ech complex of $\mathscr{F}$ with respect to the pair $(\mathcal{S}, \mathcal{S}')$ of coverings, that is, 
$C(\mathcal{S}, \mathcal{S}';\,\mathscr{F})$ is the complex
$$
\cdots
\xrightarrow{\delta^{k-1}} C^k(\mathcal{S}, \mathcal{S}';\,\mathscr{F})
\xrightarrow{\delta^{k}} C^{k+1}(\mathcal{S}, \mathcal{S}';\,\mathscr{F})
\xrightarrow{\delta^{k+1}} C^{k+2}(\mathcal{S}, \mathcal{S}';\,\mathscr{F})
\xrightarrow{\delta^{k+2}} \cdots.
$$
Here $C^k(\mathcal{S}, \mathcal{S}';\,\mathscr{F})$ consists of
alternating sections $\{s_\alpha\}_{\alpha \in \Lambda^{k+1}}$ 
with $s_\alpha \in \mathscr{F}(U_\alpha)$ and $s_\alpha = 0$ if $\alpha \in (\Lambda')^{k+1}$, and the differential $\delta^k$ is defined by
$$
\delta^k(\{s_\alpha\}_{\alpha \in \Lambda^{k+1}})_\beta
= \sum_{i=1}^{k+2}(-1)^{i+1} s_{\beta^{\vee_i}}|_{U_\beta} \qquad (\beta \in \Lambda^{k+2}),
$$
where $\beta^{\vee_i}$ denotes the sequence such that the $i$-th element of $\beta$ is removed.

Let $\mathscr{F}^\bullet \in \CatC{X}$ be a complex with bounded below of sheaves of $\mathbb{Z}$-modules
$$
\cdots
\xrightarrow{d^{k-1}} \mathscr{F}^k
\xrightarrow{d^{k}} \mathscr{F}^{k+1}
\xrightarrow{d^{k+1}}  \mathscr{F}^{k+2}
\xrightarrow{d^{k+2}} \cdots.
$$
Then we denote by $C(\mathcal{S},\mathcal{S}')(\mathscr{F}^\bullet)$ 
the single complex associated with the double complex 
$$
\begin{array}{ccccccc}
&\uparrow\,\,\,\text{ }
&
&\uparrow\,\,\,\text{ }
&
&\uparrow\,\,\,\text{ }
& \\
\xrightarrow{d^{q-1}}
&C^{p+1}(\mathcal{S},\mathcal{S}';\,\mathscr{F}^q)
&\xrightarrow{d^{q}}
&C^{p+1}(\mathcal{S},\mathcal{S}';\,\mathscr{F}^{q+1})
&\xrightarrow{d^{q+1}}
&C^{p+1}(\mathcal{S},\mathcal{S}';\,\mathscr{F}^{q+2})
&\xrightarrow{d^{q+2}}
\\
&\uparrow\delta^p
&
&\uparrow\delta^p
&
&\uparrow\delta^p
& \\
\xrightarrow{d^{q-1}}
&C^{p}(\mathcal{S},\mathcal{S}';\,\mathscr{F}^q)
&\xrightarrow{d^{q}}
&C^{p}(\mathcal{S},\mathcal{S}';\,\mathscr{F}^{q+1})
&\xrightarrow{d^{q+1}}
&C^{p}(\mathcal{S},\mathcal{S}';\,\mathscr{F}^{q+2})
&\xrightarrow{d^{q+2}} \\
&\uparrow\,\,\,\text{ }
&
&\uparrow\,\,\,\text{ }
&
&\uparrow\,\,\,\text{ }
& \\
\end{array},
$$
that is, the complex is given by
$$
C^k(\mathcal{S},\mathcal{S}')(\mathscr{F}^\bullet) 
= \bigoplus_{p+q = k} C^{p}(\mathcal{S},\mathcal{S}';\,\mathscr{F}^q)
$$
and, for $\omega = \underset{p+q=k}{\oplus} \omega^{p,q} \in C^k(\mathcal{S},\mathcal{S}')(\mathscr{F}^\bullet)$,
$$
d^k_{C(\mathcal{S},\mathcal{S}')(\mathscr{F}^\bullet)}(\omega)
= \underset{p+q=k+1}{\oplus}(\delta^{p-1}(\omega^{p-1,q}) + (-1)^p d^{q-1}(\omega^{p,q-1})).
$$
Let $\mathscr{F} \in \CatM{X}$ and let $i: \mathscr{F} \to \mathscr{F}^\bullet$ be a resolution of $\mathscr{F}$ by soft sheaves, that is, 
$\mathscr{F}^\bullet \in \CatC{X}$ consists of soft sheaves and
the morphism $i$ of complexes is quasi-isomorphic.
Then we sometimes call the complex $C(\mathcal{S},\mathcal{S}')(\mathscr{F}^\bullet)$ the
 relative \v{C}ech derived complex of $\mathscr{F}$ 
(with respect to the pair $(\mathcal{S},\mathcal{S}')$ of coverings).
In particular, if $X$ is a complex manifold and $\mathscr{C}^{(0,\bullet)}_X$ is 
the Dolbeault complex which is a soft resolution of the sheaf $\mathscr{O}_X$ of holomorphic functions on $X$, then we say
$C(\mathcal{S},\mathcal{S}')(\mathscr{C}^{(0,\bullet)}_X)$ to be the relative
\v{C}ech Dolbeault complex.
\begin{teo}[\cite{HIS}]\label{teo:C-D-complex-iso}
Under the above situation, there exists the canonical isomorphism in $\CatMD$:
$$
\mathbf{R}\Gamma_K(X;\, \mathscr{F}) \simeq C(\mathcal{S},\mathcal{S}')(\mathscr{F}^\bullet).
$$
\end{teo}
\begin{es}
If we take 
$$
\mathcal{V}=\{V_0 = X \setminus K,\, V_1=X\},\qquad \mathcal{V}'=\{V_0\}
$$
as coverings of $X$ and $X \setminus K$,
then the complex $C(\mathcal{V},\mathcal{V}')(\mathscr{F}^\bullet)$ becomes quite simple
as follows:
$$
C^k(\mathcal{V},\mathcal{V}')(\mathscr{F}^\bullet)
= \mathscr{F}^k(V_1) \oplus \mathscr{F}^{k-1}(V_{01}),
$$
where $V_{01}=V_0 \cap V_1$, and $d^k_{C(\mathcal{V},\mathcal{V}')(\mathscr{F}^\bullet)}$
is given by
$$
\mathscr{F}^k(V_1) \oplus \mathscr{F}^{k-1}(V_{01}) \ni
(\omega_1,\, \omega_{01}) \mapsto (d^k\omega_1,\, \omega_1|_{V_{01}} - d^{k-1}\omega_{01})
\in \mathscr{F}^{k+1}(V_1) \oplus \mathscr{F}^{k}(V_{01}).
$$
This complex is often denoted by $C(X,X \setminus K)(\mathscr{F}^\bullet)$,
and its $k$-th cohomology group is also written by $H^k(X,X \setminus K;\,\mathscr{F})$
if $\mathscr{F}^\bullet$ is a soft resolution of $\mathscr{F}$, which is
isomorphic to $H^k_K(X;\,\mathscr{F})$ by the above theorem.
\end{es}
\subsection{Radial compactification}
Let $M$ be an $n$-dimensional real vector space with the norm
$| \bullet |$ 
and $E = M \otimes_{\mathbb{R}} \mathbb{C}$. 
We denote by $\rcE$ (resp. $\rcM$) the radial compactification 
$E \sqcup S^{2n-1}$ (resp. $M \sqcup S^{n-1}$) of 
$E$ (resp. $M$) as usual (see Definition 2.1 \cite{hu2}). 
Note that $\rcM = \overline{M}$ holds, where $\overline{M}$ is the closure
of $M$ in $\rcE$.
We also set $M_\infty = \rcM \setminus M$ and $E_\infty = \rcE \setminus E$.
Through the paper, we use the following identification
$$
E_\infty = S^{2n-1} = (E \setminus \{0\})/\mathbb{R}_+,\qquad
M_\infty = S^{n-1} = (M \setminus \{0\})/\mathbb{R}_+.
$$
In particular, $\zeta \in E_\infty$ is sometimes identified with a unit vector in $E$.

\

We define an $\mathbb{R}_+$-action on $\rcE$ 
by, for $\lambda \in \mathbb{R}_+$ and $x \in \rcE$,
$$
\lambda x =
\begin{cases}
\lambda x & \text{if $x \in E$}, \\
x & \text{if $x \in E_\infty$}.
\end{cases}
$$
The $\mathbb{R}_+$-action on $\rcM$ is defined to be the restriction of the one in $\rcE$ to $\rcM$.
And we also define an addition for $a \in M$ (resp. $a \in E$) and
$x \in \rcM$ (resp. $x \in \rcE$) by
$$
a + x =
\begin{cases}
	a + x & \text{if $x \in M$ (resp. $x \in E$)}, \\
	x & \text{if $x \in M_\infty$ (resp. $x \in E_\infty$)}.
\end{cases}
$$
\begin{df}{\label{def:cone}}
A subset $K$ in $\rcM$ is said to be a cone with vertex $a \in M$ in $\rcM$ if there exists
an $\mathbb{R}_+$-conic set $L \subset \rcM$ such that
$$
K = a + L = \{a + x \in \rcM;\, x \in L\}.
$$
Here, if $L$ is an empty set, we set $a+L = \emptyset$ for convenience.
\end{df}
The notion of a cone in $\rcE$ is similarly defined.
%
%
We often need to extend an open subset in $E$ to the one in $\rcE$. 
\begin{df}{\label{def:def_hat}}
Let $V$ be an open subset in $E$, we define the open subset $\widehat{V}$ 
in $\rcE$ by
$$
\widehat{V} = \rcE \setminus \overline{(E \setminus V)}.
$$
\end{df}
Note that we sometimes write $\widehat{\,\,\,\,}V$ instead of $\widehat{V}$.
For an open subset $U$ in $M$, we can define 
an open subset $\widehat{U}$ in $\rcM$ in the same way as that in $\rcE$.  
\begin{lem}
Let $V$ be an open subset in $E$. Then
$\widehat{V}$ is the largest open subset $W$ in $\rcE$ with
$
	V = W \cap E.
$
\end{lem}
\begin{proof}
Let $W$ be the largest open subset with $W \cap E = V$.
Clearly we have $W \supset \widehat{V}$. Let us show the converse inclusion:
$W \cap E =V$ implies $E \setminus V \subset \rcE \setminus W$. Since $\rcE \setminus W$
is closed, we have $\overline{E \setminus V} \subset \rcE \setminus W$, which shows
$W \subset \widehat{V}$.
\end{proof}

In Definition 3.4 of \cite{hu1}, we introduced the notion that an open subset $U$ in
$\rcE$ is regular at $\infty$. In this paper, we call such an open subset
``$1$-regular at $\infty$'' to distinguish it from the similar notion for a closed subset
defined below.

\begin{df}{\label{def:regular-set}}
	A closed subset $F \subset \rcE$ is said to be regular (at $\infty$) if 
$
\overline{F \cap E} = F
$
holds.
\end{df}

\begin{lem}
Let $K \subset \rcE$ be a closed cone with vertex $a$.
Then $K$ becomes regular if and only if the equivalence
$$
\pi_{E_\infty}(x) \in K \cap E_\infty \iff a + x \in K
$$
holds for any $x \in E \setminus \{0\}$.
Here $\pi_{E_\infty}: E\setminus \{0\} \to E_\infty = (E \setminus \{0\})/\mathbb{R}_+$ is the canonical projection.
\end{lem}
\begin{proof}
Assume $\overline{K \cap E} = K$.
Let $x_\infty =\pi_{E_\infty}(x) \in K \cap E_\infty$ with $x \ne 0$. 
It follows from the assumption that
there exists a sequence $\{x_k\} \subset K \cap E$ such that 
$x_k \to  x_\infty$ $(k \to \infty)$, which implies $(x_k -a)/|x_k - a|$ converges to $x/|x|$. Since $K$ is a closed cone with the vertex $a$, we have $(x_k -a)/|x_k -a | + a \in K$, and hence, we get
$x/|x| + a \in K$ by taking $k \to \infty$. This implies $x + a \in K$. Conversely, if $x + a \in K$ ($x \ne 0$).
Then we have $\mathbb{R}_{\ge 0} x + a \subset K$, which implies 
$\pi_{E_\infty}(x) \in \overline{K \cap E}$. Hence we get $\pi_{E_\infty}(x) \in K$.

\

We show $\overline{K \cap E} = K$ under the equivalence condition of the lemma.
Since $K$ is closed, we have $\overline{K \cap E} \subset K$.
It suffices to show that $x_\infty \in \overline{K \cap E}$ holds
for $x_\infty \in K \cap E_\infty$. If we identify $x_\infty$ as a unit vector in $E$,
then we have $x_\infty + a \in K$ by the equivalence condition of the lemma. 
Therefore, we have $\mathbb{R}_{\ge 0} x_\infty + a \subset K \cap E$, which implies $x_\infty \in \overline{K \cap E}$.
\end{proof}
Note that, for example, the set consisting of the only one point in $E_\infty$ 
is a closed cone in our definition, however, which is not regular.

\begin{lem}[Lemma 3.5 \cite{hu1}]
Let $K \subset \rcE$ be a closed cone with vertex $a$. The conditions below are equivalent:
\begin{enumerate}
\item $K$ is regular.
\item $\widehat{\,\,\,\,}(E \setminus K) = \rcE \setminus K$ holds.
\item $\rcE \setminus K$ is a $1$-regular at $\infty$ (for the definition
	of $1$-regularity, see Definition 3.4 \cite{hu1}).
\end{enumerate}
\end{lem}
\begin{proof}
It follows from the definition of $\widehat{\,\,\,\,}(\bullet)$ that we have
$$
\widehat{\,\,\,\,}(E \setminus K) = \rcE \setminus \overline{(K \cap E)}.
$$
Hence the conditions 1.~and 2.~are equivalent. 

Let us show 3.~implies 1.
By the definition, ``$\rcE \setminus K$ being $1$-regular'' is equivalently saying that
$$
K \cap E_\infty = \mathrm{clos}^1_{\infty}(K)
$$
holds. 
Since we have $\mathrm{clos}^1_{\infty}(K) \subset \overline{(K \cap E)} \cap E_\infty$,
we get
$$
K \cap E_\infty = \mathrm{clos}^1_{\infty}(K) \subset 
\overline{(K \cap E)} \cap E_\infty \subset K \cap E_\infty,
$$
which shows $\overline{K \cap E} = K$.  
Since $\rcE \setminus K$ is an open cone with vertex $a$,
the implication 2.~to 3.~immediately follows from the Lemma 3.5 \cite{hu1}.
\end{proof}

The following definition are often used through the paper:
For open subsets  $U \subset \rcM$  and $\Gamma \subset M$,
define an open subset $U \widehat{\times} \sqrt{-1}\Gamma$ in $\rcE$ by
\begin{equation}{\label{eq:hat_times}}
U \widehat{\times} \sqrt{-1}\Gamma = \widehat{\,\,\,\,}
((U \cap E) \times \sqrt{-1}\Gamma) \subset \rcE.
\end{equation}

\

Let $M^*$ and $E^*$ be dual vector spaces of $M$ and $E$, respectively.
Then we can define the radial compactification $\rcDM$ and $M^*_\infty$ 
(resp. $\rcDE$ and $E^*_\infty$)
for a vector space $M^*$ (resp. $E^*$) in the same way as those of $\rcM$ and
$M_\infty$ (resp. $\rcE$ and $E_\infty$).

We also define the open subset $\widehat{V}$ in $\rcDE$ for
an open subset $V$ in $E^*$ in the same way as that in $\rcE$, that is,
\begin{equation}
\widehat{V} = \rcDE \setminus \overline{(E^* \setminus V)}.
\end{equation}
Now we introduce the subset $\HHPC{Z}$ in $E^*_\infty$ and 
the canonical projection $\varpi_{M^*_\infty}$ as follows: The canonical projection 
$\varpi_{M^*_\infty}: E^*_\infty \setminus \sqrt{-1}M^*_\infty \to M^*_\infty$  is 
defined by
\begin{equation}{\label{eq:inf_times_star}}
E^*_\infty \setminus \sqrt{-1}M^*_{\infty}
= ((M^*\setminus\{0\}) \oplus \sqrt{-1}M^*)/\mathbb{R}_+
\,\xrightarrow{\,\varpi_{M^*_\infty}\,}\, (M^*\setminus\{0\})/\mathbb{R}_+ = M^*_\infty,
\end{equation}
which is induced from the canonical projection
$E^* = M^* \oplus \sqrt{-1}M^* \to M^*$, that is,
$\varpi_{M^*_\infty}$ is given by
$$
E^*_\infty \setminus \sqrt{-1}M^*_\infty 
\ni \xi + \sqrt{-1}\eta\,\, ((\xi,\eta) \in S^{2n-1},\, \xi \ne 0)
\,\,\mapsto\,\, \xi/|\xi| \in M^*_\infty.
$$
Let $Z$ be a subset in $\rcE$. 
\begin{df}
The subset $\HHPC{Z}$ in $E^*_\infty$ is defined by
$$
\{\zeta \in E^*_\infty;\, \operatorname{Re}\,\langle z,\, \zeta \rangle > 0\,\,(\forall z
\in \overline{Z} \cap E_\infty)\}.
$$
\end{df}
Note that $\HHPC{Z}$ is an open subset in $E^*_\infty$. Further, 
for $Z \subset \rcM \subset \rcE$, we see that $\HHPC{Z} \cap \sqrt{-1}M^*_\infty \ne \emptyset$ holds if and only if $\overline{Z} \cap M_\infty = \emptyset$ (i.e., $Z$ is a compact set in $M$).
\begin{df}{\label{df:proper_contained}}
We say that $Z$ is properly contained in a half space  
of $\rcE$ with direction $\zeta \in E^*_\infty$
if  there exists $r \in \mathbb{R}$ such that
\begin{equation}
	\overline{Z} \subset \widehat{\,\,\,\,}{\{z \in E;\, 
	\operatorname{Re}\langle z,\,\zeta \rangle > r\}},
\end{equation}
where $\zeta$ is regarded as a unit vector in $E^*$.
If a subset $Z$ is properly contained in a half space of $\rcE$ with some
direction, then $Z$ is often said to be a proper subset in $\rcE$.
\end{df}
Then it is easy to see:
\begin{lem}
	Let $\zeta \in E^*_\infty$ and $Z \subset \rcE$.
The $Z$ is properly contained in a half space of $\rcE$ with direction $\zeta$ 
if and only if
$\zeta \in \HHPC{Z}$.
\end{lem}
\begin{proof}
The implication $\Longrightarrow$ is clear. We will show the converse implication.
Take an open subset $\Omega$ in $E_\infty$ such that
$$
\overline{Z} \cap E_\infty \subset  \Omega \subset \overline{\Omega} 
\subset {\{z \in E_\infty;\, 
\operatorname{Re}\langle z,\,\zeta \rangle > 0\}}.
$$
Then, since $(\overline{Z} \cap E_\infty) \cap (E_\infty \setminus \Omega) =\emptyset$ 
holds, there exists $R > 0$ such that
$$
\overline{\{t\omega \in E;\, t \ge R, \omega \in E_\infty \setminus \Omega\}}
\cap \overline{Z} = \emptyset,
$$
from which we have
$$
\overline{Z} \subset \widehat{\,\,\,\,}{\{z \in E;\, 
\operatorname{Re}\langle z,\,\zeta \rangle > -R|\zeta|\}}.
$$
\end{proof}
\begin{es}
	Let $G$ be an $\mathbb{R}_+$-conic closed subset in $E$ and $a \in E$. 
Set $K = \overline{a + G} \subset \rcE$.  Then we have
$$
\HHPC{K} = \HHPC{G} = \widehat{\,\,\,\,}(\mathrm{int}\,G^\circ) \cap E^*_\infty,
$$
where $G^\circ$ is the dual cone of $G$ in $E^*$, that is,
$$
G^\circ = \{\zeta \in E^*;\, \textrm{Re}\,\langle z,\,\zeta \rangle \ge 0\,\,
(\forall z \in G)\}.
$$
\end{es}

\section{Several variants of de-Rham and Dolbeault complexes of exponential type on $\rcE$}{\label{sec:dolbeault-complex}}

Let $V$ be an open subset in $\rcE$ and $f$ a measurable function on $V \cap E$. 
We fix a coordinate system $z = x + \sqrt{-1}y$ of $E$ in what follows.

We say
that $f$ is of exponential type (at $\infty$) on $V$ if, for any compact subset $K$ in $V$,
there exists $H_K > 0$ such that $|\exp (-H_K|z|)\, f(z)|$ 
is essentially bounded on $K \cap E$, i.e., 
\begin{equation}
||\exp (-H_K|z|)\, f(z) ||_{L^\infty(K \cap E)} < +\infty.
\end{equation}
Set
$$
\mathscr{Q}_{\rcE}(V) := \left\{f \in C^\infty(V \cap E);\,
\begin{aligned}
&\text{Any higher derivative of $f$ with respect to variables $z$ and $\bar{z}$}\\
&\text{is of exponential type on $V$}
\end{aligned}
\right\}.
$$
Then it is easy to see that 
$\{\mathscr{Q}_{\rcE}(V)\}_V$ forms the sheaf $\mathscr{Q}_{\rcE}$ on $\rcE$. The following easy lemma is crucial in our theory:
\begin{lem}
The sheaf $\mathscr{Q}_{\rcE}$ is fine, in particular, it is a soft sheaf.
\end{lem}
Let $\mathscr{Q}_{\rcE}^{(p,q)}$ denote the sheaf on $\rcE$ of $(p,\,q)$-forms with
coefficients in $\mathscr{Q}_{\rcE}$, that is, $f \in \mathscr{Q}_{\rcE}^{(p,q)}(V)$ is written by
$$
\sum_{|I|=p, |J|=q} f_{I,J}(z) dz_I \wedge d\overline{z}_J
$$
with $f_{I,J}(z) \in\mathscr{Q}_{\rcE}(V)$, and set
$$
\mathscr{Q}_{\rcE}^{(k)} = \bigoplus_{p + q = k}\, \mathscr{Q}_{\rcE}^{(p,q)}.
$$
Now we define the de-Rham complex $\mathscr{Q}_{\rcE}^{(\bullet)}$ on $\rcE$ with coefficients in $\mathscr{Q}_{\rcE}$ by
$$
0
\longrightarrow
\mathscr{Q}_{\rcE}^{(0)} 
\overset{d}{\longrightarrow}
\mathscr{Q}_{\rcE}^{(1)} 
\overset{d}{\longrightarrow}
\dots
\overset{d}{\longrightarrow}
\mathscr{Q}_{\rcE}^{(2n)}
\longrightarrow
0,
$$
and the Dolbeault complex $\mathscr{Q}_{\rcE}^{(p, \bullet)}$ on $\rcE$ by
$$
0
\longrightarrow
\mathscr{Q}_{\rcE}^{(p,0)} 
\overset{\bar{\partial}}{\longrightarrow}
\mathscr{Q}_{\rcE}^{(p,1)} 
\overset{\bar{\partial}}{\longrightarrow}
\dots
\overset{\bar{\partial}}{\longrightarrow}
\mathscr{Q}_{\rcE}^{(p,n)}
\longrightarrow
0.
$$
Let $\hexpo$ (resp. $\hexppo{p}$) denote the sheaf of 
holomorphic functions (resp. $p$-forms) of exponential type (at $\infty$) on $\rcE$,
that is,
$$
\hexppo{p}(V) = \{f \in \mathscr{Q}_{\rcE}^{(p,0)}(V);\, \bar{\partial} f = 0\}.
$$
The following proposition can be shown by the similar arguments as those in the proof of
the usual de-Rham and Dolbeault theorems with bounds.
\begin{prop}
Both the canonical morphisms of complexes below are quasi-isomorphic:
$$
\mathbb{C}_{\rcE} \longrightarrow \mathscr{Q}_{\rcE}^{(\bullet)},\qquad
\hexppo{p} \longrightarrow \mathscr{Q}_{\rcE}^{(p,\,\bullet)}.
$$
\end{prop}
We show, in \cite{hu1}, the Oka type vanishing 
theorem of holomorphic functions of exponential type
on a Stein domain. Hence the above proposition immediately concludes:
\begin{cor}[\cite{hu1}, Theorem 3.7]
Assume that $V\cap E$ is Stein and that $V$ is $1$-regular at $\infty$. Then
we have the quasi-isomorphism
$$
\hexppo{p}(V) \longrightarrow \mathscr{Q}_{\rcE}^{(p,\,\bullet)}(V).
$$
\end{cor}
Furthermore, the edge of the wedge type theorem of exponential type 
has been also established in our previous papers:
\begin{teo}[\cite{hu2}, Theorem 3.12, Proposition 4.1]{\label{thm:edge_of_wedge}}
	The complexes ${\mathbf R}\Gamma_{\rcM}(\hexppo{p})$ and
${\mathbf R}\Gamma_{\rcM}(\mathbb{Z}_{\rcE})$ are concentrated in degree $n$.
Furthermore, ${\mathscr H}^n_{\rcM}(\mathbb{Z}_{\rcE})$ is isomorphic to ${\mathbb Z}_{\rcM}$.
\end{teo}
%

\

In subsequent sections, we need to extend our de-Rham theorem to the one with a parameter.
Let $T$ be a real analytic manifold and set $Y := T \times \rcE$ and 
$Y_\infty = T \times (\rcE \setminus E)$.
We denote by $p_T: Y \to T$ (resp. $p_{\rcE}: Y \to \rcE$)
the canonical projection to $T$ (resp. $\rcE$). 

Let $W$ be an
open subset in $Y$ and $f(t,z)$ a measurable function on $W \setminus Y_\infty$.
We say that $f(t,z)$ is of exponential type on $W$ if, for any compact subset
$K$ in $W$, there exists $H_K > 0$ such that $|\exp(-H_K|z|)\, f(t,z)|$
is essentially bounded on $K \setminus Y_\infty$.

Now we introduce the set $\LQ_{Y}(W)$ consisting of 
a locally integrable function $f(t,z)$ on $W \setminus Y_\infty$ satisfying
the condition that
any higher derivative (in the sense of distributions, for example) of $f(t,z)$ with respect to the variables $z$ and $\bar{z}$
is a locally integrable function of exponential type on $W$.
Then, in the same way as in $\mathscr{Q}_{\rcE}$, 
the family $\{\LQ_Y(W)\}_W$
forms the sheaf $\LQ_Y$ on $Y$ which is fine.
Let $\LQ_Y^{(k)}$ denotes the sheaf on $Y$ of $k$-forms with respect to 
the variables in $E$, and let us define the de-Rham complex $\LQ_Y^{(\bullet)}$
by
$$
0
\longrightarrow
\LQ_{Y}^{(0)} 
\overset{d_{\rcE}}{\longrightarrow}
\LQ_{Y}^{(1)} 
\overset{d_{\rcE}}{\longrightarrow}
\dots
\overset{d_{\rcE}}{\longrightarrow}
\LQ_{Y}^{(2n)}
\longrightarrow
0,
$$
where $d_{\rcE}$ is the differential on $\rcE$.

Let $\EQ_Y$ be the subsheaf of $\LQ_Y$ consisting of 
a $C^\infty$-function (with respect to all the variables $t$, $z$ and $\bar{z}$) whose any higher derivative also belongs to $\LQ_Y$.  
Then we have also
the de-Rham complex $\EQ_Y^{(\bullet)}$:
$$
0
\longrightarrow
\EQ_{Y}^{(0)} 
\overset{d_{\rcE}}{\longrightarrow}
\EQ_{Y}^{(1)} 
\overset{d_{\rcE}}{\longrightarrow}
\dots
\overset{d_{\rcE}}{\longrightarrow}
\EQ_{Y}^{(2n)}
\longrightarrow
0.
$$

We denote by $\mathscr{L}^\infty_{loc,T}$ (resp. $\mathscr{E}_T$) the sheaf of $L_{loc}^\infty$-functions (resp. $C^\infty$-functions) on $T$.
Then the following proposition follows from the same arguments as those of a usual de-Rham complex.
\begin{prop}
We have the quasi-isomorphisms
$$
{p}_{T}^{-1} \mathscr{L}^\infty_{loc,T} \longrightarrow
\LQ_Y^{(\bullet)}\qquad\text{and}\qquad
{p}_{T}^{-1} \mathscr{E}_{T} \longrightarrow \EQ_Y^{(\bullet)}.
$$
\end{prop}
We also have
\begin{prop}{\label{prop:const-with-param}}
	Let $\mathscr{F}$ be a sheaf of $\mathbb{Z}$-modules on $T$.
The complexes ${\mathbf R}\Gamma_{p^{-1}_{\rcE}(\rcM)}(p_T^{-1} \mathscr{F})$ 
is concentrated in degree $n$, and we have the canonical isomorphism
$$
\tilde{p}_T^{-1} \mathscr{F}\,\, \otimes_{\mathbb{Z}_{p^{-1}_{\rcE}(\rcM)}} \,\,
or_{p_{\rcE}^{-1}(\rcM)/ Y}
\longrightarrow 
{\textrm H}^n_{p_{\rcE}^{-1}(\rcM)}(p_T^{-1} \mathscr{F}),
$$
where $\tilde{p}_{T}: p^{-1}_{\rcE}(\rcM) = T \times \rcM \to T$ is the canonical projection.
\end{prop}
\begin{proof}
Since $\rcM$ has an open neighborhood $U$ in $\rcE$ which is topologically
isomorphic to $\rcM \times \mathbb{R}^n$, we may replace
$\rcE$ with $U = \rcM \times \mathbb{R}^n$, and we have the commutative diagram 
of topological spaces
$$
\begin{diagram}
\node{T} 
\node{Y = T \times \rcM \times \mathbb{R}^n}
\arrow{w,t}{p_T} 
\arrow{s,l}{\pi}
\node{p^{-1}_{\rcE}(\rcM) = T \times \rcM}
\arrow{sw,l}{{\textrm id}}
\arrow{w,t}{i}  \\
\node[2]{p^{-1}_{\rcE}(\rcM) = T \times \rcM}
\arrow{nw,l}{\tilde{p}_T}
\end{diagram},
$$
where $i(t,x) = (t,x,0)$ and $\pi(t,x,y) = (t,x)$.
Then, for a sheaf $\mathscr{F}$ on $T$, we have a chain of isomorphisms
$$
\begin{aligned}
{\mathbf R}\Gamma_{p^{-1}_{\rcE}(\rcM)}(p_T^{-1} \mathscr{F})
&\simeq i^! p^{-1}_T \mathscr{F}
\simeq i^! \pi^{-1} \tilde{p}^{-1}_T \mathscr{F} \\
&\simeq i^! \pi^! \tilde{p}_T^{-1} \mathscr{F} 
\otimes i^{-1} or_{Y/p^{-1}_{\rcE}(\rcM)}[-n] \\
&\simeq \tilde{p}_T^{-1} \mathscr{F} 
\otimes i^{-1} or_{Y/p^{-1}_{\rcE}(\rcM)}[-n].
\end{aligned}
$$
The last isomorphism comes from the fact $\pi \circ i = \textrm{id}$, which
also implies 
$$
or_{p^{-1}_{\rcE}(\rcM)/Y} \otimes i^{-1} or_{Y/p^{-1}_{\rcE}(\rcM)} \simeq {\mathbb Z}_{p^{-1}_{\rcE}(\rcM)}.
$$
This completes the proof.
\end{proof}

\begin{cor}{\label{cor:mes-zero}}
Let $W$ be an open subset in $Y$ and $s \in {\textrm H}^n_{p^{-1}_{\rcE}(\rcM)}(W;\,p_T^{-1} \mathscr{L}^\infty_{loc,T})$, and
let $\Delta$ be a subset in $\tilde{W} := W \cap p^{-1}_{\rcE}(\rcM)$. Assume 
the conditions below:
\begin{enumerate}
\item $\tilde{p}_T(\tilde{W}) \setminus \tilde{p}_T(\Delta)$ 
is a set of measure zero in $T$.
\item For any $q \in \Delta$, the stalk $s_q \in 
{\textrm H}^n_{p^{-1}_{\rcE}(\rcM)}(\,p_T^{-1} \mathscr{L}^\infty_{loc,T})_q$ 
of $s$ is zero.
\item The set $\tilde{p}_T^{-1}\tilde{p}_T(q) \cap \tilde{W}$ 
is connected for any $q \in \tilde{W}$.
\end{enumerate}
Then $s$ is zero.
\end{cor}
\begin{proof}
We have the commutative diagram, for any point $q \in \tilde{W}$,
$$
\begin{matrix}
{\textrm H}^n_{p^{-1}_{\rcE}(\rcM)}(W;\,p_T^{-1} \mathscr{L}^\infty_{loc,T}) 
&\simeq
&\Gamma(\tilde{W};\, \tilde{p}^{-1} \mathscr{L}^\infty_{loc,T}) 
&\simeq
&\Gamma(\tilde{p}_T(\tilde{W});\, \mathscr{L}^\infty_{loc,T}) \\
\downarrow & & \downarrow & & \downarrow \\
{\textrm H}^n_{p^{-1}_{\rcE}(\rcM)}(p_T^{-1} \mathscr{L}^\infty_{loc,T})_q 
&\simeq
&(\tilde{p}_T^{-1} \mathscr{L}^\infty_{loc,T})_q 
&\simeq
&(\mathscr{L}^\infty_{loc,T})_{\tilde{p}_T(q)}.
\end{matrix}
$$
Hence $s$ can be regarded as an $L_{loc}^\infty$-function on $\tilde{p}_T(\tilde{W})$.
Then, by the assumption, $s$ is zero on $\tilde{p}_T(\Delta)$. Hence
$s$ is almost everywhere zero, and thus, $s$ is zero
as an $L_{loc}^\infty$-function. This completes the proof.
\end{proof}

We can also define the Dolbeault complex with a parameter in the same way
as $\mathscr{Q}_{\rcE}^{(p,\,\bullet)}$.
Let $\LQ_Y^{(p,q)}$ and $\EQ_Y^{(p,q)}$ be the sheaves of
$(p,q)$-forms of $z$ and $\bar{z}$ with coefficients in $\LQ_Y$ and $\EQ_Y$, respectively.
Then we define the Dolbeault complex $\LQ^{(p, \bullet)}$ with a parameter on $Y$ by
$$
0
\longrightarrow
\LQ_{Y}^{(p,0)} 
\overset{\bar{\partial}}{\longrightarrow}
\LQ_{Y}^{(p,1)} 
\overset{\bar{\partial}}{\longrightarrow}
\dots
\overset{\bar{\partial}}{\longrightarrow}
\LQ_Y^{(p,n)}
\longrightarrow
0,
$$
and $\EQ^{(p, \bullet)}$ on $Y$ by
$$
0
\longrightarrow
\EQ_{Y}^{(p,0)} 
\overset{\bar{\partial}}{\longrightarrow}
\EQ_{Y}^{(p,1)} 
\overset{\bar{\partial}}{\longrightarrow}
\dots
\overset{\bar{\partial}}{\longrightarrow}
\EQ_Y^{(p,n)}
\longrightarrow
0.
$$
Then by standard arguments we have
\begin{prop}
Both the canonical morphisms of complexes below are quasi-isomorphic:
$$
\LO_Y \longrightarrow \LQ^{(0,\bullet)}_Y,\qquad
\EO_Y \longrightarrow \EQ^{(0,\,\bullet)}_Y.
$$
\end{prop}
Here $\LO_Y$ and $\EO_Y$ are the subsheaves of $\LQ_Y$ and $\EQ_Y$
consisting of sections which are holomorphic with respect to the variables $z$,
respectively.

\section{Various expressions of Laplace hyperfunctions}
Let $M$ be an $n$-dimensional real vector space with the norm
$|\bullet|$
and $E = M \otimes_{\mathbb{R}} \mathbb{C}$. Recall that $\rcE$ (resp. $\rcM$) denotes the radial compactification $E \sqcup S^{2n-1}$ (resp. $M \sqcup S^{n-1}$) of $E$ (resp. $M$). 
Let $U$ be an open subset in $\rcM$, 
and $V$ an open subset in $\rcE$ with $V \cap \rcM = U$.

\begin{df}
The sheaf on $\rcM$ of $p$-forms of Laplace hyperfunctions is defined by
$$
\bexppo := {\mathscr H}^n_{\rcM}(\hexppo{p}) \otimes_{\mathbb{Z}_\rcM} or_{\rcM/\rcE},
$$
where $or_{\rcM/\rcE}$ is the relative orientation sheaf over $\rcM$, that is, it is given by ${\mathscr H}^n_{\rcM}(\mathbb{Z}_{\rcE})$.
\end{df}
It follows from Theorem {\ref{thm:edge_of_wedge}} that
we have
$$
\bexppo(U) = 
{\textrm{H}}^n_{U}(V;\,\hexppo{p})\,\, \otimes_{{\mathbb{Z}_{\rcM}}(U)} \,\,
or_{\rcM/\rcE}(U).
$$
In particular, 
$\{{\textrm{H}}^n_{U}(V;\,\hexppo{p})\}_U$ forms a sheaf on $\rcM$.

The above cohomology groups have several equivalent expressions. We briefly recall
those definitions which will be used in this paper.

\subsection{Representation by relative \v{C}ech Dolbeault cohomology groups}
We first give a representation of a Laplace hyperfunction by the relative \v{C}ech Dolbeault cohomology groups.
Set $V_0 = V \setminus {U}$, $V_1 = V$ and $V_{01} = V_0 \cap V_1$ as usual.
Then define the coverings
$$
\mathcal{V}_{U} = \{V_0,\, V_1\}, \qquad \mathcal{V}_{U}'=\{V_0\}.
$$
The complex $\QQDC{p}{\mathcal{V}_{U}}$ is called the relative
\v{C}ech Dolbeault complex of exponential type
(see Subsection \ref{subsec:cheh-dolbeault-complex} for the definition of the functor 
$C(\mathcal{V}_{U},{\mathcal{V}_{U}}')(\bullet)$):
$$
\begin{aligned}
0
\longrightarrow
\QDC{p}{0}{\mathcal{V}_{U}}
&\overset{\bvth}{\longrightarrow}
\QDC{p}{1}{\mathcal{V}_{U}}
\overset{\bvth}{\longrightarrow}
\dots \\
&\qquad\overset{\bvth}{\longrightarrow}
\QDC{p}{n}{\mathcal{V}_{U}}
\longrightarrow
0,
\end{aligned}
$$
where $\bvth$ is used to denote the differential of this complex.
In the same way, the complex $\QQDRC{\mathcal{V}_{U}}$
$$
\begin{aligned}
0
\longrightarrow
\QDRC{0}{\mathcal{V}_{U}}
&\overset{D}{\longrightarrow}
\QDRC{1}{\mathcal{V}_{U}}
\overset{D}{\longrightarrow}
\dots \\
&\qquad\overset{D}{\longrightarrow}
\QDRC{2n}{\mathcal{V}_{U}}
\longrightarrow
0,
\end{aligned}
$$
where $D$ is used to denote the differential of this complex,
is called the relative \v{C}ech de-Rham complex of exponential type.

As was seen in the previous section, $\mathscr{Q}_{\rcE}^{(\bullet)}$ and
$\mathscr{Q}_{\rcE}^{(p,\bullet)}$ are soft resolutions of 
$\mathbb{C}_{\rcE}$ and $\hexppo{p}$, respectively.
Hence the following theorem immediately follows from Theorem \ref{teo:C-D-complex-iso}:
\begin{teo}
	There exist the canonical isomorphisms in $\CatMD$:
$$
{\textbf R}\Gamma_{U}(V;\,{\mathbb C}_{\rcE}) 
\simeq \QQDRC{\mathcal{V}_{U}},\qquad
{\textbf R}\Gamma_{U}(V;\,\hexppo{p}) 
\simeq \QQDC{p}{\mathcal{V}_{U}}.
$$
\end{teo}

\

It follows from the theorem that we have
\begin{equation}
\bexppo(U) \simeq
\HQDC{p}{n}{\mathcal{V}_{U}}\,
\otimes_{{\mathbb Z}_{\rcM(U)}}\, or_{\rcM/\rcE}(U).
\end{equation}
This implies that any Laplace hyperfunction $u \in \bexppo(U)$ is represented
by a pair $(\omega_1, \omega_{01})$ of $C^\infty$-forms which satisfies the following conditions 1.~and 2.
\begin{enumerate}
	\item $\omega_1 \in  \mathscr{Q}_{\rcE}^{(p,n)}(V)$ and 
	$\omega_{01} \in  \mathscr{Q}_{\rcE}^{(p,n-1)}(V \setminus U)$
	\item $\overline{\partial} \omega_{01} = \omega_1$ on $V \setminus {U}$.
\end{enumerate}

\begin{oss}
Let $\mathcal{S} = \{S_i\}_{i \in \Lambda}$ be a finite open covering of $V$,
and let $\Lambda' \subset \Lambda$.
Assume $\mathcal{S}'= \{S_i\}_{i \in \Lambda'}$ is an open covering of $V \setminus U$.
Then, as did in Subsection \ref{subsec:relative_chec},
$\QQDC{p}{\mathcal{S}}$ 
(resp. $\QQDRC{\mathcal{S}}$)
denotes
the relative \v{C}ech Dolbeault complex 
(resp. the relative \v{C}ech de-Rham complex) of exponential type
with respect to the pair $(\mathcal{S}, \mathcal{S}')$ of coverings.
For these complexes, we also have the isomorphisms
$$
{\textbf R}\Gamma_{U}(V;\,{\mathbb C}_{\rcE}) 
\simeq \QQDRC{\mathcal{S}},\qquad
{\textbf R}\Gamma_{U}(V;\,\hexppo{p}) 
\simeq \QQDC{p}{\mathcal{S}}.
$$
\end{oss}

\subsection{Representation by relative \v{C}ech cohomology groups}
Next we give a representation of a Laplace hyperfunction by the relative \v{C}ech cohomology groups. We first recall the Grauert type theorem for existence of a Stein open neighborhood of an open subset in $\rcM$.
\begin{teo}[Theorem 4.10 \cite{KU}]
Let $\Omega$ be an open cone in $M$, and
let $V \subset \rcE$ be an open neighborhood of $\widehat{\Omega}$. Then we can find an open set
$W \subset \rcE$ such that
\begin{enumerate}
\item $\widehat{\Omega} \subset W \subset V$.
\item $W \cap E$ is a Stein open subset.
\item $W$ is $1$-regular at $\infty$.
\end{enumerate}
\end{teo}

By taking the above theorem into account, we assume that, in this subsection, $V$ is $1$-regular at $\infty$ and $V \cap E$ is a Stein open subset.  We also set $U = V \cap \rcM$.
Let $\eta_0,\dots, \eta_{n-1}$ be linearly independent vectors in $M^*$ so that
$\{\eta_0, \dots, \eta_{n-1}\}$ forms a positive frame of $M^*$.
Set $\eta_{n} := - (\eta_0 + \dots + \eta_{n-1}) \in M^*$ and
$$
S_k := \widehat{\,\,\,\,}
\{z = x+\sqrt{-1}y \in E;\,z \in V,\, \langle y,\, \eta_k\rangle > 0\} \qquad (k=0,1,\cdots,n).
$$
For convenience, we set $S_{n+1}=V$.
Let $\Lambda = \{0,1,2,\dots,n+1\}$ and set,
for any $\alpha = (\alpha_0,\dots,\alpha_k) \in \Lambda^{k+1}$,
$$
S_\alpha := S_{\alpha_0} \cap S_{\alpha_1} \cap \dots \cap S_{\alpha_k}.
$$
We define coverings of $V$ and $V \setminus U$ by
$$
\mathcal{S} := \{S_0,S_1,\dots, S_{n+1}\},\qquad
\mathcal{S}' := \{S_0,\dots, S_{n}\}.
$$
Since $S_\alpha \cap E$ is a Stein open subset and $S_\alpha$ is $1$-regular at $\infty$ for any $\alpha \in \Lambda^{k+1}$, by the theory of the relative \v{C}ech cohomology,
we have the isomorphism
$$
{\textrm H}^n_{U}(V;\,\hexppo{p})
\simeq
{\textrm H}^n(C(\mathcal{S},\,\mathcal{S}')(\hexppo{p})).
$$
Let $\Lambda^{k+1}_*$ be the subset in $\Lambda^{k+1}$ consisting of $\alpha
= (\alpha_0,\dots, \alpha_k)$ with
$$
\alpha_0 < \alpha_1 < \dots < \alpha_k = n+1.
$$
Then we obtain
$$
{\textrm H}^n(C(\mathcal{S},\,\mathcal{S}')(\hexppo{p}))
\simeq \dfrac{\bigoplus_{\alpha \in \Lambda^{n+1}_*}
\hexppo{p}(S_\alpha)}{\bigoplus_{\beta \in \Lambda^{n}_*} \hexppo{p}(S_\beta)}.
$$
Hence, any hyperfunction $u$ has a representative $\underset{\alpha \in \Lambda^{n+1}_*}{\oplus} f_\alpha$ which is a formal sum of
$(n+1)$-holomorphic functions of exponential type defined on each $S_\alpha$ ($\alpha \in \Lambda^{n+1}_*$).

Note that the \v{C}ech representation and the \v{C}ech Dolbeault representation of Laplace hyperfunctions are linked by the following diagram
whose morphisms are all quasi-isomorphisms.
\begin{equation}
C(\mathcal{S},\,\mathcal{S}')(\hexppo{p})
\xrightarrow{\,\,\alpha_1\,\,}
\QQDC{p}{\mathcal{S}}
\xleftarrow{\,\,\alpha_2\,\,}
\QQDC{p}{\mathcal{V}_U},
\end{equation}
where the middle complex is the \v{C}ech Dolbeault one associated with the 
covering $(\mathcal{S}, \mathcal{S}')$, $\alpha_1$ is induced from
the canonical morphism $\hexppo{p} \to \mathscr{Q}_{\rcE}^{(p,\bullet)}$ of Dolbeault complexes and $\alpha_2$ follows from the fact that
$\mathcal{S}$ is a finer covering of $\mathcal{V}_{U}$.

Furthermore, let $i:\mathscr{Q}_{\rcE}^{(p,\bullet)} \to \mathscr{L}^\bullet$
be a flabby resolution of the complex $\mathscr{Q}_{\rcE}^{(p,\bullet)}$, i.e.,
$i$ is a morphism of complexes which is quasi-isomorphic and $\mathscr{L}^\bullet$
is the complex of flabby sheaves on $\rcE$. Then we have the commutative diagram
$$
\begin{diagram}
\node[3]{\Gamma_U(V;\,\mathscr{L}^\bullet)} 
\arrow{s,l}{}
\arrow{sw,l}{}
\\
\node[2]{C(\mathcal{S},\mathcal{S}')(\mathscr{L}^\bullet)}
\node{C(\mathcal{V}_U,\mathcal{V}'_U)(\mathscr{L}^\bullet)}
\arrow{w,t}{\alpha_2}\\
\node{C(\mathcal{S},\,\mathcal{S}')(\hexppo{p})}
\arrow{e,t}{\alpha_1}
\arrow{ne,l}{}
\node{\QQDC{p}{\mathcal{S}}}
\arrow{n,l}{i}
\node{\QQDC{p}{\mathcal{V}_U}}
\arrow{n,l}{i}
\arrow{w,t}{\alpha_2}
\end{diagram},
$$
where all the morphism are quasi-isomorphisms. Hence we have obtained the (canonical) isomorphisms between cohomology groups:
$$
\begin{diagram}
	\node[2]{\bexppo(U)=\mathrm{H}^n_U(V;\,\hexppo{p})} \\
	\node{\mathrm{H}^n(C(\mathcal{S},\,\mathcal{S}')(\hexppo{p}))}
\arrow{e,t}{\alpha_1}
\arrow{ne,l}{}
\node{\mathrm{H}^n(\QQDC{p}{\mathcal{S}})}
\arrow{n,l}{}
\node{\mathrm{H}^n(\QQDC{p}{\mathcal{V}_U})}
\arrow{nw,l}{}
\arrow{w,t}{\alpha_2}
\end{diagram}.
$$
In what follows, all the cohomology groups are identified through these canonical isomorphisms.
\subsection{Generalization of \v{C}ech representations}{\label{subsec:intiuitve}}

Representation by \v{C}ech cohomology groups can be generalized to the much more convenient one,
that is ``intuitive representation'' of Laplace hyperfunctions
introduced in \cite{KU}. Let us briefly recall this representation. 
Let $U$ be an open subset in $\rcM$, and
let $\Gamma$ be an $\mathbb{R}_+$-conic connected open subset in $M$. 

\begin{df}[\cite{KU} Definition 4.8]{\label{def:wedge}}
An open subset $W \subset \rcE$ is said to be an infinitesimal wedge
of type $U \widehat{\times} \sqrt{-1}\Gamma$ if and only if for any $\mathbb{R}_+$-conic open subset $\Gamma'$ properly contained in $\Gamma$ there exists an open neighborhood $O \subset \rcE$ of $U$ such that
$$
(U \widehat{\times} \sqrt{-1}\Gamma') \,\cap\, O \,\subset\, W.
$$
holds (see \eqref{eq:hat_times} for the symbol $\widehat{\times}$).
\end{df}
\begin{oss}
The definition of an infinitesimal  wedge itself does not assume the inclusion 
$W \subset U \widehat{\times} \sqrt{-1}\Gamma$.
\end{oss}
We denote by $\mathcal{W}(U \widehat{\times} \sqrt{-1}\Gamma)$ 
the set of all the infinitesimal wedges of type $U \widehat{\times} \sqrt{-1}\Gamma$
which are contained in $U\widehat{\times} \sqrt{-1}\Gamma$.
Furthermore, we set
$$
\mathcal{W}(U) := \bigcup_\Gamma \mathcal{W}(U \widehat{\times} \sqrt{-1}\Gamma),
$$
where $\Gamma$ runs through all the $\mathbb{R}_+$-conic connected open subsets in $M$
(in particular, $\Gamma$ is non-empty).

Define the quotient vector space
\begin{equation}
\hat{\mathrm{H}}^n(\oexp(\mathcal{W}(U)))
:= \left(\bigoplus_{W \in \mathcal{W}(U)}\,
\oexp(W)\right)/\mathcal{R},
\end{equation}
where $\mathcal{R}$ is a $\mathbb{C}$-vector space generated by elements
$$
f \oplus (-f|_{W_2})\quad  \in \oexp(W_1) \oplus \oexp(W_2)
$$
for any $W_2 \subset W_1$ in $\mathcal{W}(U)$ and any $f \in \oexp(W_1)$.
\begin{teo}[\cite{KU} Theorem 4.9]{\label{thm:intuitive}}
	Let $U$ be an open cone in $\rcM$ such that $\widehat{\,\,\,\,}(U \cap M) = U$.
Then there exists a family $b_{\mathcal{W}} = \{b_W\}_{W \in \mathcal{W}(U)}$
of morphisms $b_W: \oexp(W) \to \bexpo(U)$ $(W \in \mathcal{W}(U))$
which satisfies
$$
b_{W_1}(f) = b_{W_2}(f|_{W_2}) \qquad\text{in}\,\,\bexpo(U)
$$
for any $W_2 \subset W_1$ in $\mathcal{W}(U)$ and any $f \in \oexp(W_1)$.
Furthermore the induced morphism
$$
b_\mathcal{W}: \hat{\mathrm{H}}^n(\oexp(\mathcal{W}(U))) \to \bexpo(U)
$$
becomes an isomorphism.
\end{teo}
\begin{oss}{\label{oss:remark_bb_map_and_ic}}
If $W \in \mathcal{W}(U)$ is cohomologically trivial, that is, it satisfies the condition A2.~given in Subsection {\ref{sec:functrial-contstruction}}, 
then $b_W$ coincide with the boundary value map functorially constructed
in Subsection {\ref{sec:functrial-contstruction}}
(see also Subsection 3.1 in \cite{KU}, where the boundary value map was
constructed in a functorial way).
\end{oss}

Now let us consider the problem under the situation 
given in the previous subsection, that is,
open subsets $U$ and $V$, coverings $\mathcal{S}$ and $\mathcal{S}'$,
and vectors $\{\eta_k\}$ are those ones already given in the previous subsection. Additionally
we also assume $\widehat{\,\,\,\,}(U \cap M) = U$.

Then there exists the canonical isomorphism 
\begin{equation}
\iota_{IC}:{\textrm H}^n(C(\mathcal{S},\,\mathcal{S}')(\hexpo))
\longrightarrow \hat{\mathrm{H}}^n(\oexp(\mathcal{W}({U})))
\end{equation}
which is defined by
\begin{equation}
\dfrac{\bigoplus_{\alpha \in \Lambda^{n+1}_*}
\hexpo(S_\alpha)}{\bigoplus_{\beta \in \Lambda^{n}_*} \hexpo(S_\beta)}
\ni [(f_\alpha)_\alpha] \mapsto (-1)^n\bigoplus_\alpha\, \mathrm{sgn}(\alpha) f_\alpha \in \hat{\mathrm{H}}^n(\oexp(\mathcal{W}({U}))),
\end{equation}
where, for $\alpha = (\alpha_0,\cdots,\alpha_n)$, the $\mathrm{sgn}(\alpha)$ is $1$ if the vectors $\eta_{\alpha_0}$, $\eta_{\alpha_1}$,
$\cdots$, $\eta_{\alpha_{n-1}}$ form a positive frame in $M$ and it is $-1$ otherwise.

The following lemma follows from the construction of $b_{\mathcal{W}}$
(see Definition 3.14 and Theorem 3.15 \cite{KU}) .
\begin{lem}{\label{lem:b_w_and_I_C_commute}}
The morphism $\iota_{IC}$ makes the following diagram commutative:
$$
\begin{diagram}
\node{{\textrm H}^n(C(\mathcal{S},\,\mathcal{S}')(\hexpo))}
\arrow{e,t}{\iota_{IC}}\arrow{se,l}{\sim}
\node{\hat{\mathrm{H}}^n(\oexp(\mathcal{W}({U})))}\arrow{s,l}{b_\mathcal{W}} \\
\node[2]{\bexpo(U)}
\end{diagram},
$$
where all the morphism are isomorphic.
\end{lem}


\section{Boundary value map in $\rcE$}{\label{sec:boundary-value}}

One of the important features in hyperfunction theory is a boundary value map,
by which we can regard a holomorphic function of exponential type on an wedge along $\rcM$ as
a Laplace hyperfunction. We construct, in this section, the boundary value map
in the framework of the relative \v{C}ech Dolbeault cohomology.

Let $U$ be an open subset in $\rcM$ and $V$ an open subset in $\rcE$ such that
$V \cap \rcM = U$. Let $\Omega$ be an open subset in $\rcE$.

\subsection{Functorial construction}{\label{sec:functrial-contstruction}}

We first construct the boundary value map in a functorial way.
For an open subset $W$ in $\rcE$ and a complex $F$ of sheaves on $W$, we define its dual on $W$ by
$$
\textrm{D}_{W}(F) := {\mathbf R}{\mathcal{H}}om_{\mathbb{C}_W}(F,\, \mathbb{C}_W).
$$
Note that, for a complex $F$ of sheaves on $\rcE$, we have 
$\textrm{D}_{\rcE}(F)|_W = \textrm{D}_W(F|_W)$.
We assume:
\begin{enumerate}
	\item[A1.]  $\,\,$  $U \subset \overline{\Omega}$.
	\item[A2.]  $\,\,$  $\Omega$ is cohomologically trivial in $V$, that is, 
$$
\textrm{D}_{\rcE}(\mathbb{C}_\Omega)|_V \simeq \mathbb{C}_{\overline{\Omega}}|_V,\qquad
\textrm{D}_{\rcE}(\mathbb{C}_{\overline{\Omega}})|_V \simeq \mathbb{C}_{\Omega}|_V.
$$
\end{enumerate}
Through this subsection, we always assume conditions A1.~and A2.
Following Schapira's construction (see Section 11.5 in \cite{KS}) of a boundary value morphism, we can construct the corresponding one for a Laplace hyperfunction as follows:
Let $j_V: V \to \rcE$ be the canonical inclusion. 
By the assumption, we have the canonical morphism on $V$
$$
j_V^{-1}\mathbb{C}_{\overline{\Omega}} \to j_V^{-1}\mathbb{C}_{\rcM}.
$$
It follows from the assumption that we have
$$
\textrm{D}_{V}(j_V^{-1}\mathbb{C}_{\overline{\Omega}}) \simeq 
j_V^{-1} \mathbb{C}_{\Omega},\qquad
\textrm{D}_{V}(j_V^{-1}\mathbb{C}_{\rcM}) \simeq 
j_V^{-1}(\mathbb{C}_{\rcM} \otimes or_{\rcM/\rcE})[-n].
$$
Hence, applying the functor $\textrm{D}_{V}(\bullet)$ to the above morphism, we obtain the
canonical morphism
$$
j_V^{-1}(\mathbb{C}_{\rcM} \otimes or_{\rcM/\rcE})[-n] \to j_V^{-1}\mathbb{C}_{\Omega}.
$$
Now applying the functor 
${\mathbf R}{\textrm Hom}_{\mathbb{C}_V}(\bullet,\, j_V^{-1}\hexpo)$
to the above morphism and taking the $0$-th cohomology groups, we have obtained
the boundary value map
$$
b_{\Omega}: \hexpo(\Omega \cap V) \to \bexpo(U).
$$

\subsection{\v{C}ech Dolbeault construction of a boundary value map}{\label{subsec:c-d-boundary-value-map}}

The construction of a boundary value map for Laplace hyperfunctions
in the framework of the relative \v{C}ech Dolbeault cohomology
is the almost same as that for hyperfunctions done in the paper \cite{HIS}.
First recall the coverings
$$
\mathcal{V}_U = \{V_0,\, V_1\}, \qquad {\mathcal{V}_U}'=\{V_0\}
$$
of $V$ and $V \setminus U$,
where $V_0 = V \setminus U$, $V_1 = V$ and $V_{01} = V_0 \cap V_1$.
We now construct the boundary value morphism 
$$
b_\Omega: \hexpo(\Omega) \longrightarrow 
\HQDC{0}{n}{\mathcal{V}_U}\, \otimes_{{\mathbb Z}_{\rcM(U)}}\, or_{\rcM/\rcE}(U)
$$
in the framework of the relative \v{C}ech Dolbeault cohomology.

Let us first recall the morphism of complexes 
$$
\rho: \QQDRC{\mathcal{V}_U} \to \QQDC{0}{\mathcal{V}_U}, 
$$
which is 
defined by the projection to the space of anti-holomorphic forms, that is,
$$
\QDRC{k}{\mathcal{V}_U} \ni
\sum_{|I|+ |J|=k} f_{I,J}dz^I \wedge d\bar{z}^J
\qquad \mapsto \qquad
\sum_{|J|=k} f_{\emptyset, J} d\bar{z}^J
\in \QDC{0}{k}{\mathcal{V}_U}.
$$
Then we have
\begin{lem}
The following diagram commutes:
$$
\begin{diagram}
\node{{\mathbf R}\Gamma_{U}(V;\,{\mathbb C}_{\rcE})}
\arrow{e}\arrow{s}
\node{{\mathbf R}\Gamma_{U}(V;\,\hexpo)} \arrow{s} \\
\node{\QQDRC{\mathcal{V}_U}} 
\arrow{e,t}{\rho}
\node{\QQDC{0}{\mathcal{V}_U}}
\end{diagram},
$$
where the top horizontal arrow is the morphism associated with the canonical sheaf
morphism $\mathbb{C}_{\rcE} \to \hexpo$.
\end{lem}

Let us take a section $\mathds{1} \in {\textrm H}^n_{U}(V;\,\mathbb{Z}_{\rcE})$
such that, for each $x \in U$, the stalk $\mathds{1}_x$ of $\mathds{1}$ at $x$ generates ${\textrm H}^n_{\rcM}(\mathbb{Z}_{\rcE})_x$ as a $\mathbb{Z}$-module. 
Note that we have, in each connected component of $U$, 
two choices of such a $\mathds{1}$, i.e., 
either $\mathds{1}$ or $-\mathds{1}$. Then the canonical sheaf morphism
$\mathbb{Z}_{\rcE} \to \mathbb{C}_{\rcE}$ induces the injective morphism
$$
{\textrm H}^n_{U}(V;\,\mathbb{Z}_{\rcE}) \to 
{\textrm H}^n_{U}(V;\,\mathbb{C}_{\rcE}).
$$
Note that we still denote by $\mathds{1}$ the image in ${\textrm H}^n_{U}(V;\,\mathbb{C}_{\rcE})$ of $\mathds{1}$ by this morphism.

\

Now we assume the following conditions to $\Omega$.
\begin{enumerate}
	\item[B1.] $\,\,$The canonical inclusion
$(V \setminus \Omega) \setminus \rcM \hookrightarrow (V \setminus \Omega)$
gives a homotopical equivalence.
\end{enumerate}
The following lemma can be proved in the same way as that in Lemma 7.10 in \cite{HIS}.
\begin{lem}{\label{lem:support_one}}
Assume the conditions A1 and B1. Then
there exists $\tau = (\tau_1,\,\tau_{01}) \in \QDRC{n}{\mathcal{V}_U}$ which satisfies
the following conditions:
\begin{enumerate}
\item $D \tau = 0$ and $[\tau] = \mathds{1}$ in $\HQDRC{n}{\mathcal{V}_U}$.
\item $\operatorname{supp}_{V_{01}}(\tau_{01}) \subset \Omega$ and
$\operatorname{supp}_{V_{1}}(\tau_1) \subset \Omega$.
\end{enumerate}
\end{lem}

Now we assume the conditions A1 and B1, and let $\tau = (\tau_1,\tau_{01})$ be the one
given in the above Lemma.  Then we can define the morphism
\begin{equation}
b_\Omega: \hexpo(\Omega) \longrightarrow 
\HQDC{0}{n}{\mathcal{V}_U}\, 
\otimes_{{\mathbb Z}_{\rcM(U)}}\, or_{\rcM/\rcE}(U).
\end{equation}
by 
\begin{equation}{\label{eq:def_boundary_map}}
b_\Omega(f) = [f \rho(\tau)] \otimes \mathds{1} \qquad (f \in \hexpo(\Omega)).
\end{equation}
\begin{lem} 
The above $b_\Omega$ is well-defined.
\end{lem}

To avoid a higher jet as an $\Omega$, we also introduce the following condition
\begin{enumerate}
	\item[B2.] $\,\,$For any point $x \in \rcM$, there exist an open neighborhood 
$W \subset \rcE$ of $x$ and a non-empty open cone $\Gamma \subset M$
such that
$$
((W \cap M) \widehat{\times} \sqrt{-1}\Gamma) \,\cap\, W \,\subset\, \Omega.
$$
\end{enumerate}
Note that the condition B2 implies A1.
We also introduced the localized version of the condition B1.
\begin{enumerate}
	\item[B1'.] $\,\,$For any point $x \in \rcM$, there exist a family 
$\{V_\lambda\}_{\lambda \in \Lambda}$ of fundamental open neighborhoods
of $x$ in $V$, for which the canonical inclusion
$(V_\lambda \setminus \Omega) \setminus \rcM \hookrightarrow (V_\lambda \setminus \Omega)$
gives a homotopical equivalence.
\end{enumerate}

The following theorem can be shown in the same way as that in Appendix A.~in \cite{HIS}.
\begin{teo}[Theorem A.2 \cite{HIS}]{\label{teo:b_omega_commute}}
	Assume the conditions A2, B1, B1' and B2. Then the boundary value morphism
constructed in the functorial way and the one in this subsection coincide.
To be more precise, the following diagram commutes:
$$
\begin{diagram}
	\node{\hexpo(\Omega)}\arrow{e,t}{b_\Omega}\arrow{se,r}{b_\Omega}\node{\bexpo(U)} \\
	\node[2]
	{\mathrm{H}^n(C(\mathcal{V}_U,\mathcal{V}'_U)(\mathscr{Q}^\bullet_{\rcE}))
	\otimes_{{\mathbb Z}_{\rcM(U)}}\, or_{\rcM/\rcE}(U)}
\arrow{n,l}{}
\end{diagram}.
$$
\end{teo}

Furthermore, for any $W \in \mathcal{W}(U)$, we can find 
$\widetilde{W} \in \mathcal{W}(U)$ such that $\widetilde{W} \subset W$ and it satisfies
the conditions given in the above theorem. Hence, 
by Remark \ref{oss:remark_bb_map_and_ic}, we have the following corollary.
\begin{cor}{\label{cor:commute_b_omega_xx}}
For any $\Omega \in \mathcal{W}(U)$ satisfying the condition A2, 
we have the commutative diagram below:
$$
\begin{diagram}
\node{\hexpo(\Omega)}\arrow{e,t}{b_\Omega}\arrow{se,r}{}\node{\bexpo(U)} \\
\node[2]{\hat{\mathrm{H}}^n(\oexp(\mathcal{W}(U)))}
\arrow{n,l}{b_\mathcal{W}}
\end{diagram},
$$
where $b_{\mathcal{W}}$ is given by Theorem {\ref{thm:intuitive}} and
the down right arrow is just the embedding
$$
\hexpo(\Omega) \ni f \mapsto f \in \hat{\mathrm{H}}^n(\oexp(\mathcal{W}(U))).
$$
\end{cor}

Now we give a concrete construction of $\tau$ in a specific case.
\begin{es}\label{es:fundamental_example}
Let $U = \rcM$ and $V \subset \rcE$ be an open neighborhood of $U$. 
Let $\eta_0$, $\dots$, $\eta_\ell$ be a family of unit vectors in $M^*$ such that
for any $k+1$ choices $\eta_{\alpha_0}$, $\cdots$, $\eta_{\alpha_k}$
$(\alpha = (\alpha_0,\cdots, \alpha_k) \subset \{0,1,\cdots, \ell\})$
of vectors which are linearly dependent, the cone
$$
\mathbb{R}_{+} \eta_{\alpha_0} + \mathbb{R}_{+} \eta_{\alpha_1} + \cdots +
\mathbb{R}_{+} \eta_{\alpha_k}
$$
should contain a line (i.e., it becomes a non-proper cone). Furthermore, we also assume
that
$$
\sum_{k=0}^\ell \mathbb{R}_{+} \eta_k = M^*.
$$
Note that a typical family of such vectors is, for example,
\begin{itemize}
\item $n$ linearly independent unit vectors $\xi_0,\cdots,\xi_{n-1}$ with
	$\xi_{n} = - (\xi_0 + \cdots + \xi_{n-1})/|\xi_0 + \cdots + \xi_{n-1}|$.
\item a family of $2n$ unit vectors
$$
(\pm 1,0,\cdots,0), \cdots, (0,\cdots,0, \pm 1).
$$
\end{itemize}
Let $\Gamma_k$ ($k= 0, \cdots, \ell$) be open subsets in $M$ defined by
$$
\Gamma_k = \{y \in M;\, \langle y, \eta_k\rangle > 0\}.
$$
Then, by the conditions for $\{\eta_k\}$, for any 
$\alpha = (\alpha_0,\cdots, \alpha_k)$ with
$\eta_{\alpha_0}$, $\cdots$, $\eta_{\alpha_k}$ being linearly dependent
(which is always satisfied if $k \ge n$), we have
$$
\Gamma_{\alpha_0} \cap \Gamma_{\alpha_2} \cap \cdots \cap \Gamma_{\alpha_k} = \emptyset.
$$
Further we also have
$$
\bigcup_{k=0,\cdots,\ell}\,\Gamma_{k} = M \setminus \{0\}.
$$
Let $H_k$ ($k=0,\cdots,\ell$) be an $\mathbb{R}_+$-conic open convex subset in $M$ 
satisfying the conditions below:
\begin{enumerate}
	\item $H_k \subset \Gamma_k$ $(k=0,1,\cdots,\ell)$.
\item $H_0 \cup H_1 \cup \dots \cup H_\ell = M \setminus \{0\}$.
\end{enumerate}
Then we choose $\ell+1$ sections $\varphi_0$, $\dots$, $\varphi_{\ell}$ in
$\mathscr{Q}_{\rcE}(V \setminus U)$ which satisfies
\begin{enumerate}
\item Set $S_k := (U \widehat{\times} \sqrt{-1} H_k) \cap V$. Then
$\operatorname{supp}(\varphi_k) \subset S_k$ holds for $k=0,\dots,\ell$. 
\item $\varphi_0 + \varphi_1 + \dots + \varphi_{\ell} = 1$ on $V \setminus U$.

\end{enumerate}
Set
$$
\Lambda^{k+1}_* = \{\alpha = (\alpha_0,\cdots, \alpha_k) \subset \{0,1,\cdots,\ell+1\};\,
\alpha_0 < \alpha_1 < \cdots < \alpha_k = \ell+1\}
$$
and set $S_{\ell+1} = V$.
For $\alpha = (\alpha_0,\cdots, \alpha_k) \in \Lambda^{k+1}_*$, we define
$$
S_{\alpha} = S_{\alpha_0} \cap S_{\alpha_1} \cap \cdots \cap S_{\alpha_k}
$$
as usual and for $\alpha = (\alpha_0,\cdots, \alpha_n) \in \Lambda^{n+1}_*$, we also define
$$
\mathrm{sgn}(\alpha) = 
\begin{cases}
	+1 \qquad &(\text{$\xi_{\alpha_0}$, $\cdots$, $\xi_{\alpha_{n-1}}$ is a positive frame in $M^*$ }), \\
	-1\qquad &(\text{otherwise}).
\end{cases}
$$
Now for $\alpha = (\alpha_0,\cdots, \alpha_n) \in \Lambda^{n+1}_*$, we define
\begin{equation}
\tau^\alpha_{01} = (-1)^{n}(n-1)!\,\mathrm{sgn}(\alpha) \chi_{S_{\alpha}}\,
d \varphi_{\alpha_0} \wedge \cdots \wedge d\varphi_{\alpha_{n-2}},
\end{equation}
where $\chi_Z$ is the characteristic function of the set $Z$. 
We can see the following facts by the same reasoning as that of Example 7.14 in 
\cite{HIS}.
\begin{enumerate}
	\item $\tau^\alpha := (0,\,\tau^\alpha_{01})$ belongs to $\QDRC{n}{\mathcal{V}_{U}}$.
\item $D \tau^\alpha = 0$ and $[\tau^\alpha] = \mathds{1}$ in 
$\HQDRC{n}{\mathcal{V}_U}$. Here we choose $\mathds{1}$ so that it gives the standard positive orientation of $M$.
\item $\operatorname{supp}_{V\setminus U}(\tau^\alpha_{01}) \subset S_\alpha$.
This follows from the fact that, on $S_\alpha$, we have
$\varphi_{\alpha_0} + \cdots + \varphi_{\alpha_{n-1}} = 1$.
\end{enumerate}
Hence this $\tau^\alpha$ satisfies all the desired properties described in
Lemma \ref{lem:support_one} when $\Omega = S_\alpha$. Note that we have
\begin{equation}
\rho(\tau^\alpha) =
\left(0,\,\, (-1)^{n}(n-1)!\, \mathrm{sgn}(\alpha)\chi_{S_\alpha}\,
\bar{\partial} \varphi_{\alpha_0} \wedge \cdots \wedge \bar{\partial} \varphi_{\alpha_{n-2}}\right).
\end{equation}
In particular, for $f \in \hexpo(\Omega)$, we have
$$
b_\Omega(f) = [f \rho(\tau^\alpha)]
= [\left(0,\,\, (-1)^{n}(n-1)!\, \mathrm{sgn}(\alpha) f(z) \chi_{S_\alpha}\,
\bar{\partial} \varphi_{\alpha_0} \wedge \cdots \wedge \bar{\partial} \varphi_{\alpha_{n-2}}\right)].
$$

\

In what follows, we additionally assume that $V$ is $1$-regular and $V \cap E$ is Stein. 
We define the coverings $\mathcal{S}$ and
$\mathcal{S}'$ of $(V, V \setminus U)$ by
$$
\mathcal{S} = \{S_0,\cdots, S_{\ell+1}\}, \qquad
\mathcal{S}' = \{S_0,\cdots, S_{\ell}\},
$$
where 
$$
S_k = (U \widehat{\times} \sqrt{-1}H_{k+1}) \cap V \qquad
(k=0,1,\cdots,\ell)
$$
and $S_{\ell+1} = V$.  We also define the coverings 
$$
\mathcal{V}_U = \{V \setminus U,\,V\}, \qquad 
\mathcal{V}'_U = \{V \setminus U\}.
$$
Let us consider the diagram whose morphisms
are all isomorphic:
\begin{equation}
\mathrm{H}^n(C(\mathcal{S},\,\mathcal{S}')(\hexpo))
\xrightarrow{\,\,\alpha^n_1\,\,}
\HQDC{0}{n}{\mathcal{S}}
\xleftarrow{\,\,\alpha^n_2\,\,}
\HQDC{0}{n}{\mathcal{V}_U}.
\end{equation}
Let $f \in \hexpo(S_{\alpha})$.
Then it follows from Lemma A.5 \cite{HIS} that,
by repeated applications of the well-known Wiel procedure, we can find
$$
((\alpha_2^n)^{-1}\circ \alpha_1^n)([f]) =
[\left(0,\,\, (n-1)!\, f(z) \chi_{S_\alpha}\,
\bar{\partial} \varphi_{\alpha_0} \wedge \cdots \wedge \bar{\partial} 
\varphi_{\alpha_{n-2}}\right)]
=(-1)^n\mathrm{sgn}(\alpha)b_{S_{\alpha}}(f).
$$
Hence we have obtained the following commutative diagram:
\begin{equation}{\label{eq:b_w_and_diag_equiv}}
\begin{diagram}
\node{ {\textrm H}^n(C(\mathcal{S},\,\mathcal{S}')(\hexpo))} 
\arrow{s,l}{\iota_{IC}}\arrow{se,l}{(\alpha^n_2)^{-1} \circ \alpha_1^n} \\
\node{\hat{\textrm H}^n(\hexpo(\mathcal{W}(U)))}
\arrow{e,t}{b}
\node{\HQDC{0}{n}{\mathcal{V}_{U}}}
\end{diagram},
\end{equation}
where all morphisms are isomorphic and 
$
b: {\hat{\textrm H}^n(\hexpo(\mathcal{W}(U)))}
\to \HQDC{0}{n}{\mathcal{V}_{U}}
$
is defined by the boundary value morphisms
$$
\hexpo(W) \ni f \, \mapsto\, b_W(f) \in \HQDC{0}{n}{\mathcal{V}_U} \qquad
(W \in \mathcal{W}(U)).
$$
Furthermore, $\iota_{IC}$ is given by
$$
\dfrac{\bigoplus_{\alpha \in \Lambda^{n+1}_*}
\hexpo(S_\alpha)}{\bigoplus_{\beta \in \Lambda^{n}_*} \hexpo(S_\beta)}
\ni [(f_\alpha)_\alpha] \mapsto (-1)^n\bigoplus_\alpha\, \mathrm{sgn}(\alpha) f_\alpha \in \hat{\mathrm{H}}^n(\oexp(\mathcal{W}({U})))
$$
as in Lemma \ref{lem:b_w_and_I_C_commute}.
\end{es}

\section{Laplace transformation $\LL$ for hyperfunctions}

\subsection{Preparation}

Let $(z_1 = x_1 + \sqrt{-1}y_1,\,\cdots,\,z_n = x_n + \sqrt{-1}y_n)$
be a coordinate system of $E$.
Hereafter, we fix the orientation of $M$ and $E$ so that
$\left\{ \pdiff{x_1},\,\pdiff{x_2},\,\dots,\,\pdiff{x_n}\right\}$  gives the positive orientation on $M$, and
$\left\{\pdiff{y_1},\,\dots,\,\pdiff{y_n},\,\pdiff{x_1},\,\dots,\,\pdiff{x_n}\right\}$ give the one on $E$.
\begin{oss}
The above orientation of $E$ is different from the usual standard orientation
of $\mathbb{C}^n$, where $\left\{\pdiff{x_1},\,\pdiff{y_1},\,\pdiff{x_2},\,
\pdiff{y_2},\,\dots,\,\pdiff{x_n},\,\pdiff{y_n}\right\}$
is taken to be a positive frame. 
\end{oss}
We say that the boundary $\partial D$ of a subset $D$ in $\rcE$ is
(partially) smooth if $\partial D \cap E$ is (partially) smooth. Note that, when the boundary $\partial D$ 
is smooth, the orientation of $\partial D$ is determined
so that the outward-pointing normal vector of $\partial D$ followed by a positive frame
of $\partial D$ determines the positive orientation of $E$.

\

Let $h:E^*_\infty \to \{-\infty\} \cup \mathbb{R}$ be an upper semi-continuous function, and let $W$ be an open subset in $\rcDE$ and $f$ a holomorphic function on $W \cap E^*$. 

\begin{df}
We say that $f$ is of infra-$h$-exponential type (at $\infty$)
on $W$ if, for any compact set $K \subset W$ and any $\epsilon > 0$,
there exists $C > 0$ such that
$$
e^{|\zeta|\,h(\pi_{E^*_\infty}(\zeta))}|f(\zeta)| \le C e^{\epsilon |\zeta|} 
\qquad (\zeta \in K \cap (E^* \setminus \{0\})),
$$
where $\pi_{E^*_\infty}: E^*\setminus \{0\} \to (E^*\setminus \{0\})/\mathbb{R}_+ = E^*_\infty$ is the canonical projection, i.e., 
$\pi_{E^*_\infty}(\zeta) = \zeta/|\zeta|$, and we set $e^{-\infty} = 0$ for convenience.
In particular, we say that $f$ is simply called of infra-exponential type if $h \equiv 0$.
\end{df}

Define a sheaf on $E^*_\infty$ by, for an open subset $\Omega$
in $E^*_\infty$,
$$
\hinfo(\Omega) := 
\lind{W}\, \{f \in \mathscr{O}(W \cap E^*);\,
\text{$f$ is of infra-exponential type on $W$}\},
$$
where $W$ runs through open neighborhoods of $\Omega$ in $\rcDE$.
Then the family $\{\hinfo(\Omega)\}_{\Omega}$ forms the sheaf $\hinfo$ on $E^*_\infty$.
Similarly we define the sheaf $\hhinfo{h}$ on $E^*_\infty$ by, for an open subset $\Omega \subset E^*_\infty$, 
$$
\hhinfo{h}(\Omega) := 
\lind{W}\, \{f \in \mathscr{O}(W \cap E^*);\,
\text{$f$ is of infra-$h$-exponential type on $W$}\},
$$
where $W$ runs through open neighborhoods of $\Omega$ in $\rcDE$.

We also introduces the sheaf $\hexpa := \left.\hexpo\right|_{\rcM}$ of real analytic functions
of exponential type and the one $\hexpv$ of real analytic volumes
of exponential type.
The latter sheaf is defined by
$$
\hexpv = \left.\hexppo{n}\right|_{\rcM} \otimes_{\mathbb{Z}_{\rcM}} or_{\rcM},
$$
where $or_{\rcM} := (j_M)_*\,or_M$ with
$j_M: M \hookrightarrow \rcM$ being the canonical inclusion.
Note that
we can also define the orientation sheaf $or_{\rcE}$ on $\rcE$ by $(j_E)_*\,or_E$ with
the canonical inclusion $j_E: E \hookrightarrow \rcE$,
for which we have the canonical isomorphism
\begin{equation}{\label{eq:orientation_iso}}
or_{\rcM/\rcE} \otimes or_{\rcM} \simeq or_\rcE|_{\rcM}.
\end{equation}

Let $K$ be a subset in $\rcE$. Then we define the support function 
$h_K(\zeta): E^*_\infty \to \{\pm\infty\} \cup \mathbb{R}$ by
\begin{equation}
h_K(\zeta) =
\begin{cases}
+\infty \qquad &\text{if $K \cap E$ is empty}, \\
\\
\underset{z \in K \cap E}{\inf}\,\,\,\textrm{Re}\,\langle z,\,\zeta \rangle
\qquad & \text{otherwise},
\end{cases}
\end{equation}
where we identify $\zeta \in E^*_\infty$ with a unit vector in $E^*$.
Note that if $K$ is properly contained 
in a half space of $\rcE$ with direction $\zeta_0 \in E^*_\infty$ (which is equivalently saying $\zeta_0 \in \HHPC{K}$)
and if $K \cap E$ is non-empty, then the subset
$$
\overline{K} \cap \overline{\{z \in E;\, \textrm{Re}\,\langle z,\,\zeta_0 \rangle = h_K(\zeta_0)\}}
$$
is a compact set in $E$.
The following lemma easily follows from the definition.
\begin{lem}
	Let $K \subset \rcE$ with $\HHPC{K}\ne \emptyset$.
	Then $\HHPC{K}$ is a connected open subset in $E^*_\infty$. 
	Furthermore, the function
	$h_K(\zeta)$ is upper semi-continuous on $E^*_\infty$, in particular,
	it is continuous on $\HHPC{K}$ and $h_K(\zeta) > -\infty$ there. 
\end{lem}
\begin{oss}
In the above lemma, if $K \subset \rcM$, then we have
$$
\HHPC{K} = \begin{cases}
\varpi_{M^*_\infty}^{-1}(\HHPC{K}\cap M^*_\infty) 
\qquad &(\overline{K} \cap M_\infty \ne \emptyset), \\
\\
\varpi_{M^*_\infty}^{-1}(\HHPC{K}\cap M^*_\infty) \cup \sqrt{-1}M^*_\infty = E^*_\infty
\qquad &(\overline{K} \cap M_\infty = \emptyset)
\end{cases}
$$
and $h_K(\zeta)$ is continuous on $\HHPC{K} \cup \sqrt{-1}M^*_\infty$
(for the definition of $\varpi_{M^*\infty}$, see {\eqref{eq:inf_times_star}}).
\end{oss}

\subsection{Laplace transformation}

Let $K$ be a closed subset in $\rcM$ such that $\HHPC{K}$ is non-empty.
Take $\xi_0 \in \HHPC{K} \cap M^*_\infty$ and an open neighborhood $V$ of $K$ in $\rcE$.
Set $U := \rcM \cap V$ and coverings
\begin{equation}{\label{eq:def-V-K}}
\mathcal{V}_K := \{V_0 := V \setminus K,\, V_1 := V\},\qquad
\mathcal{V}'_K := \{V_0\}.
\end{equation}
In what follow, we assume that $U$ and $V$ are connected for simplicity.  
Note that we have
$$
\Gamma_K(U;\,\bexpo \otimes_{\hexpa} \hexpv) 
\simeq \HQDC{n}{n}{\mathcal{V}_K}\, \underset{\mathbb{Z}_{\rcM}(U)}{\otimes}\,
or_{\rcM/\rcE}(U)\, 
\underset{\mathbb{Z}_{\rcM}(U)}{\otimes}\, or_{\rcM}(U).
$$
Let 
$$
u = \tilde{u} \otimes a_{\rcM/\rcE} \otimes a_\rcM \in \Gamma_K(U;\,\bexpo \otimes_{\hexpa} \hexpv),
$$
where 
$a_{\rcM/\rcE} \otimes a_{\rcM} \in  or_{\rcM/\rcE}(U)\, 
\underset{\mathbb{Z}_{\rcM}(U)}{\otimes}\, or_{\rcM}(U)$ and let
$\nu = (\nu_1,\,\nu_{01}) \in \QDC{n}{n}{\mathcal{V}_K}$ be a 
representative of $\tilde{u}$, i.e., $\tilde{u}= [\nu]$. 

Here we may assume that $a_{\rcM/\rcE}$ and $a_\rcM$
are generators in each orientation sheaf. Hence,
through the canonical isomorphism {\eqref{eq:orientation_iso}}, the section
$a_{\rcM/\rcE} \otimes a_\rcM$ determines the orientation of $E$. We perform the subsequent integrations under this orientation.
\begin{oss}
If $or_{\rcM/\rcE}$ gives the orientation so that $\{dy_1,\dots,dy_n\}$ is a positive frame
and if $or_{M}$ gives the orientation so that $\{dx_1,\dots,dx_n\}$ is a positive one.
Then 
$\{dy_1,\dots,dy_n,dx_1,\dots,dx_n\}$ becomes a positive frame
under the orientation determined by $a_{\rcM/\rcE} \otimes a_\rcM$.
\end{oss}

\begin{df}{\label{def:laplace-integral}}
The Laplace transform of $u$ with a \v{C}ech Dolbeault representative
$\nu = (\nu_1,\,\nu_{01}) \in \QDC{n}{n}{\mathcal{V}_K}$ is defined by
\begin{equation}\label{eq:def-laplace}
\LL_D(u)(\zeta) := \int_{D \cap E} e^{-z \zeta} \nu_1 - 
\int_{\partial D \cap E} e^{-z \zeta}\nu_{01},
\end{equation}
where $D$ is a contractible open subset in $\rcE$ with a \textit{good} boundary 
(see the remark below)
such that $K \subset D \subset \overline{D} \subset V$ and it is properly contained
in a half space of $\rcE$ with direction $\xi_0$.
\end{df}
Note that the orientation of $D$ and $\partial D$ is taken in the usual way, that is,
the orientation of $D$ is that of $E$, and the one of $\partial D$ is determined so
that the outward pointing normal vector of $D$ and a positive frame of $\partial D$
form that of $E$.

\begin{oss}{\label{oss:good_boundary}}
Recall the definition of a piecewise $C^\infty$ submanifold with boundary in $E$
introduced in Definition 5.10 \cite{Suwa00}. We use, through the paper, a slightly weaker definition of a piecewise $C^\infty$ submanifold with boundary so that a cone 
whose boundary is smooth except for the vertex in $E$ also belongs to such a class of submanifolds.

Let $Z' \subset Z$ be subsets in $E$.
$Z$ is called a partially $C^\infty$ submanifold with
boundary $Z'$ if there exist a triangulation $(K,h)$ of $E$ and
subcomplexes $L' \subset L$ of $K$ which satisfy the conditions below:
\begin{enumerate}
\item $h(|L|) = Z$ and $h(|L'|) = Z'$
and $|L|$ is a $\mathrm{PL}$ 
submanifold of $|K|$ with boundary $|L'|$. 
\item Any simplex $\sigma$ of $L$ satisfies the following conditions, where
we set $\ell_\sigma = \dim \sigma$ and
$H_\sigma$ denotes the $\ell_\sigma$ dimensional affine space containing $\sigma$:
\begin{enumerate}
	\item $h|_{\sigma}: \sigma \to E$ is a Lipschitz mapping.
\item  $h|_{\sigma^\circ}: \sigma \setminus  |\partial \sigma| \to E$ is a $C^\infty$
mapping of rank $\ell_\sigma$, and
it extends to a $C^\infty$ mapping of rank $\ell_\sigma$ on 
an open neighborhood of $\sigma^*$ in $H_\sigma$. Here we set
$$
\sigma^* =
\sigma \setminus \bigcup_{\tau \prec \sigma,\,\dim \tau \le \ell_\sigma-2} \tau.
$$
\end{enumerate}
\end{enumerate}
Note that 
the Stokes formula holds for a partially $C^\infty$ submanifolds with boundary
in the same way as those in the case of ``a piecewise $C^\infty$ manifold with boundary'' (\cite{Suwa00}) and ``a standard chain'' introduced by Whitney (see Chapter III \cite{Wh}).

Now let us define that a open subset $D \subset \rcE$ has a good boundary:
$D$ has a good boundary if there exist $C > 0$ and $R>0$ such that the following conditions are satisfied.
\begin{itemize}
\item For any $r > R$, the sets $\overline{D \cap B_r}$ and $\partial D \cap B_r$
are compact partially $C^\infty$ submanifolds with boundaries 
$\partial(D \cap B_r)$ and $\partial D \cap \partial B_r$
in $E$, respectively. Furthermore,  
$$
|\partial (D \cap B_r)|_{2n-1} < Cr^{2n-1}, \qquad
|\partial D \cap \partial B_r|_{2n-2} < Cr^{2n-2}
$$
hold. Here $|\bullet|_{d}$ denotes the $d$ dimensional Hausdorff measure of a Borel set in $E$ and
$$
B_r = \{x + \sqrt{-1}y \in E;\, |x| \le r\}.
$$
\end{itemize}
In what follows, we always assume a chain $D$ of Laplace integral has a good boundary.

\

Before ending this remark, we explain how to make a chain $D$ with a good boundary of the integration
of Laplace transform in a convenient way. Let $K \subset \rcM$ be a non-empty regular closed cone. Now let us take a $1$-regular open cone $W \subset \rcM$ with the vertex $a \in M$ such that 
$K \subset W$ and $W \cap M$ has a smooth boundary except for the vertex
and choose a smooth function
$\rho=(\rho_1,\rho_2\,\cdots,\rho_n): W' \cap M \to \mathbb{R}^n$ 
($W' \subset \rcM$ is an open neighborhood of $\overline{W}$) satisfying
\begin{enumerate}
\item $\rho(x) \in \mathbb{R}_+^n$ ($x \in W \cap M$) and there exists $C> 0$ such that 
	$\displaystyle\sum_{1 \le i \le n}\left|\dfrac{\partial \rho}{\partial x_i}(x)\right| \le C$ on $W \cap M$.
\item there exist $\epsilon > 0$ and an open neighborhood $\widetilde{W}$ of 
$K \cap M_\infty$ in $W$ such that $|\rho(x)| \ge \epsilon|x|$ $(x \in \widetilde{W} \cap M)$ holds.
\end{enumerate}
Then we define
$$
D(W,\rho) = \widehat{\,\,\,\,}\{x + \sqrt{-1}y \in E;\, 
x \in W,\, |y_k| < \rho_k(x)\,\,(k=1,\cdots,n)\}.
$$
Clearly $D(W,\rho)$ is an open neighborhood of $K$ in $\rcE$.
Conversely,
for any open neighborhood $V$ of $K$ in $\rcE$, we have $K \subset \overline{D(W,\rho)} \subset V$ by suitable choices of $W$ and $\rho$. 
Note that $D(W,\rho)$ has a good boundary in the above sense.
As a special case of $D(W,\rho)$, if we take 
$$
W = \widehat{\,\,\,\,}\{(y_1,\cdots,y_n) \in M;\, y_k > 0\,(k=1,2,\cdots,n)\}
$$
and
$$
\rho(x) = \sigma x
$$
for some $\sigma > 0$, then $D(W, \rho) \cap E$ becomes the product of one dimensional cone
in $\mathbb{C}$.
$$
\widehat{\,\,\,\,}
(\{|y_1| < \sigma x_1\} \times \cdots \times \{|y_n| < \sigma x_n\}).
$$
\end{oss}

\begin{oss}
In what follows, we write
$\displaystyle\int_{D} e^{-z\zeta}\nu_1$ instead of
$\displaystyle\int_{D \cap E} e^{-z\zeta}\nu_1$, etc., for simplicity.
\end{oss}

\

Set $z = x + \sqrt{-1}y$ and $\zeta = \xi + \sqrt{-1}\eta$. We may assume $\xi_0 = (1,0,\cdots,0)$, and we write $x = (x_1,x')$ and $\xi = (\xi_1, \xi')$. 
Then there exist $b \in \mathbb{R}$ and  $\kappa > 0$ such that
$$
D \subset \{z = x + \sqrt{-1}y;\, |x'| + |y| \le \kappa(x_1 - b)\}.
$$
Furthermore, it follows from the
definition of $\nu$ that there exist $H > 0$ and $C \ge 0$
such that $|\nu_{01}| \le Ce^{Hx_1}$
on a neighborhood of $\partial D$ and
$|\nu_{1}| \le Ce^{Hx_1}$
on a neighborhood of $\overline{D}$.
Hence, if $z \in D$, we have
$$
|e^{-z\zeta}\nu_1| \le
Ce^{-x\xi + y\eta + Hx_1}
\le Ce^{-x_1\xi_1 + \kappa(|\xi'|+|\eta|)(x_1-b)+ Hx_1},
$$
from which the integral $\displaystyle\int_{D} e^{-z\zeta}\nu_1$ converges
if $\xi_1$ is sufficiently large. We also have the same conclusion
for $\displaystyle\int_{\partial D} e^{-z\zeta}\nu_{01}$.
\begin{lem}{\label{lem:L_well_defined}} $\LL_D(u)$ is holomorphic at points
$\zeta = R \xi_0$ if $R > 0$ is sufficiently large.
Furthermore, $\LL_D(u)$ is independent of the choices of
a representative $\nu$ of $u$ and $D$ of the integral.
Here we identify $\xi_0$ with the corresponding unit vector in $M^*$.
\end{lem}
\begin{proof}
The convergence of the integration is already shown above.
We first show $\LL_D(\bvth \tau) = 0$ for $\tau = (\tau_1, \tau_{01}) \in \QDC{n}{n-1}{\mathcal{V}_K}$.
$$
\begin{aligned}
\LL_D(\bvth \tau)
&=
\int_{D} e^{-z\zeta}\overline{\partial} \tau_1
- \int_{\partial D} e^{-z\zeta}(\tau_1 - \overline{\partial} \tau_{01})
=
\int_{D} d(e^{-z\zeta} \tau_1)
- \int_{\partial D} e^{-z\zeta}(\tau_1 - d \tau_{01}) \\
&=
\big(\int_{D} d(e^{-z\zeta} \tau_1) - \int_{\partial D} e^{-z\zeta}\tau_1\big)
+ \int_{\partial D} d(e^{-z\zeta}  \tau_{01})
= 0,
\end{aligned}
$$
where the last equality comes from the Stokes formula. Hence
the Laplace integral does not depend on the choices of representative of $u$.

Next we will show the Laplace integral is independent of the choices of $D$.
Let $\varphi \in \mathscr{Q}_{\rcE}(\rcE)$ which satisfies
\begin{enumerate}
\item $\operatorname{supp}(\varphi) \subset D$,
\item $\varphi = 1$ on $W \cap E$ for an open neighborhood $W$ of $K$ in $\rcE$,
\end{enumerate}
and define 
$$
\tilde{\nu} = (\tilde{\nu}_{1},\,\tilde{\nu}_{01}) = \left(\varphi\nu_{1} 
+ \bar{\partial}\varphi \wedge \nu_{01},\, \varphi\nu_{01}\right).
$$
Since we have 
$$
\nu - \tilde{\nu} = 
\bvth\,\left((1 - \varphi)\nu_{01},\,0\right),
$$
representatives $\nu$ and $\tilde{\nu}$ give the same cohomology class.
Hence, as the support of $\tilde{\nu}$ is contained in $D$, we have obtained
\begin{equation}{\label{eq:concise_laplace_exp}}
	\LL_D(\nu) = \LL_D(\tilde{\nu}) = 
	\int_{D} e^{-z\zeta} \tilde{\nu}_{1} =
	\int_{E} e^{-z\zeta} \tilde{\nu}_{1} =
	\int_{E} e^{-z\zeta} \left(\varphi\nu_{1}
	+ \bar{\partial}\varphi \wedge \nu_{01}\right).
\end{equation}
The last expression does not depend on $D$. This show the claim.
\end{proof}

Due to the above lemma, in what follows, we write $\LL(\bullet)$ instead of $\LL_D(\bullet)$. 
Then, thanks to the expression \eqref{eq:concise_laplace_exp} and
the integration by parts, we get:
\begin{cor}
For $u \in \Gamma_K(U;\,\bexpo \otimes_{\hexpa} \hexpv)$ and
$v \in \Gamma_K(U;\,\bexpo)$,
we have the formulas
$$
\dfrac{\partial}{\partial \zeta_k} \LL(u)
=\LL(-x_k u),\qquad
\zeta_k \LL(v\, dx\otimes a_{\rcM}) = \LL\left(\dfrac{\partial v}{\partial x_k} dx\otimes a_{\rcM}\right) \qquad (k=1,2,\cdots,n).
$$
\end{cor}

\

Note that, by the definition of $\varpi_{M^*_\infty}$ given in {\eqref{eq:inf_times_star}},
we have, for $\xi_0 \in M^*_\infty$, 
$$
\varpi_{M^*_\infty}^{-1}(\xi_0) = 
\{\xi_0 + \sqrt{-1}\eta \in E^*;\, \eta \in M^*\}/ \mathbb{R}_+
\,\,\subset E^*_\infty \setminus \sqrt{-1}M^*_\infty.
$$
\begin{prop}
Assume $K \cap M$ is non-empty.
For any $a \in K \cap \{x \in M;\, \langle x,\,\xi_0 \rangle = h_K(\xi_0)\}$,
any $\epsilon > 0$ and any compact subset $L$ in $\varpi_{M^*_\infty}^{-1}(\xi_0)$,  
there exist $C > 0$ and an open neighborhood $W \subset \rcDE$ of $L$ such that
$$
|e^{a\zeta}\LL(u)(\zeta)| \le Ce^{\epsilon |\zeta|}\qquad (\zeta \in W \cap E^*).
$$
\end{prop}
\begin{proof}
Take a point 
$\zeta_0 = (\xi_0 + \sqrt{-1}\eta_0)/|\xi_0 + \sqrt{-1}\eta_0| \in E^*_\infty$. 
In what follows, we sometimes identify a point in $E^*_\infty$ with
a unit vector in $E^*$.
Denote by $B_\delta(\zeta_0)$ an open ball 
with radius $\delta > 0$ and center at $\zeta_0$.

Since $K$ is properly contained in a half space of $\rcM$ with direction $\xi_0$,
there exist $\delta_1 > 0$, $\sigma_1 > 0$, a relatively compact open neighborhood $O \subset M$ of
$K \cap \overline{\{x \in M;\, \langle x - a,\,\xi_0 \rangle = 0\}}$ and
an $\mathbb{R}_+$-conic proper closed set $G \subset \rcM$ such that
$$
K \subset O \cup (a + \operatorname{int}(G)),
$$
$$
O \subset \{x \in M;\, |\langle x -a,\, \xi\rangle | < \epsilon/2\}
\qquad (\xi \in B_{\delta_1}(\xi_0) \cap M^*_\infty),
$$
and
$$
\langle x, \xi \rangle \ge \sigma_1 |x|\qquad
(x \in G \cap M,\,\, \xi \in B_{\delta_1}(\xi_0) \cap M^*_\infty).
$$
For $\delta_2 > 0$, define open subsets $D_O$ in $E$ and $D_G$ in $\rcE$ by
$$
D_O = \left\{z = x+\sqrt{-1}y \in E;\,
x \in O,\,\, 
|y| < \dfrac{\epsilon}{2\max\{1,\,2|\eta_0|\}}\right\},
$$
$$
D_G = \widehat{\,\,\,\,}\left\{z = x+\sqrt{-1}y \in E;\,
x \in a+\operatorname{int}(G), |y| < \delta_2\textrm{dist}(x, M \setminus (a+G))\right\}.
$$
If we take $\delta_2>0$ sufficiently small, there exists $\sigma_2 > 0$ such that
$$
\operatorname{Re}\,\langle z-a, t\zeta \rangle \ge \sigma_2 t |z-a|\qquad(t \in \mathbb{R}_+,\,\,z \in D_G \cap E,\,\, \zeta \in B_{\delta_2}(\zeta_0) \cap E^*_\infty)
$$
holds. Note that we also have
$$
|\operatorname{Re}\,\langle z -a,\, t\zeta \rangle| < \epsilon t
\qquad (t \in \mathbb{R}_+,\,\,z \in D_O,\,\,\zeta \in B_{\delta_2}(\zeta_0) \cap E^*_\infty).
$$
As in the proof of Lemma \ref{lem:L_well_defined}, we take a
$\varphi \in \mathscr{Q}_{\rcE}(\rcE)$ and set
$$
\tilde{\nu} = (\tilde{\nu}_{1},\,\tilde{\nu}_{01}) = \left(\varphi\nu_{1} 
+ \bar{\partial}\varphi \wedge \nu_{01},\, \varphi\nu_{01}\right).
$$
Then we have
$$
\LL(u)(\zeta)
= \int_{E} e^{-z \zeta} \tilde{\nu}_1.
$$
Furthermore, by taking $\varphi$ suitably, we may assume
$$
\mathrm{supp}\,{\tilde{\nu}_1} \subset D_O \cup D_G.
$$
Therefore, for $t \in \mathbb{R}_+$ and $\zeta \in B_{\delta_2}(\zeta_0) \cap E^*_\infty$, we get
$$
|e^{ta\zeta}\LL(u)(t\zeta)|
= \left|\int_{E} e^{-t(z -a) \zeta} \tilde{\nu}_1\right|
\le
\left|\int_{D_O} e^{-t(z -a) \zeta} \tilde{\nu}_1\right|
 +
 \left|\int_{D_G} e^{-t(z -a) \zeta} \tilde{\nu}_1\right|.
$$
Then, it is easy to see that there exists a positive constant $C_1 > 0$ such that, 
for any $\zeta \in B_{\delta_2}(\zeta_0) \cap E^*_\infty$ and $t \in \mathbb{R}_+$, we have
$$
\left|\int_{D_O} e^{-t(z-a) \zeta} \tilde{\nu}_{1}\right| 
\le C_1 e^{\epsilon t}.
$$
Furthermore, since there exist a constant $C_2,H > 0$ 
such that
$$
|\tilde{\nu}_{1}| \le C_{2}e^{H |z-a|} \qquad (z \in D_G \cap E),
$$
we get, for $\zeta \in B_{\delta_2}(\zeta_0) \cap E^*_\infty$ and $t \in \mathbb{R}_+$, 
$$
\left|\int_{D_G} e^{-t(z-a) \zeta} \tilde{\nu}_{1}\right| 
\le 
C_{2} \int_{D_G} e^{(-\sigma_2 t + H)|z-a|}\, d\mu,
$$
where $d\mu$ denotes the Lebesgue measure on $E$.
Hence the last integral converges if $t$ is sufficiently large,
which completes the proof. 
\end{proof}
we have the following corollary as a consequence of the proposition:
\begin{cor}{\label{cor:laplace_h_k}}
Assume $K \cap M$ is non-empty.
Then we have
$
\LL(u) \in \hhinfo{h_K}(\HHPC{K}).
$
\end{cor}
\begin{proof}
First assume that $K \cap M_\infty \ne \emptyset$. In this case,
we see that $\HHPC{K} = \varpi_{M^*_\infty}^{-1}(\HHPC{K}\cap M^*_\infty)$ and that
$h_K(\zeta)$ is upper semi-continuous. Hence the result comes from the above proposition.

Next assume that $K \cap M_\infty = \emptyset$, which implies $K$ is a compact set in $M$.
Note that $\HHPC{K} = E^*_\infty$ holds.
Then we can take a relatively compact open subset in $E$ as $D$
which is sufficiently close to $K$, and the result immediately follows.
\end{proof}

\

Let $G \ne \emptyset$ be an $\mathbb{R}_+$-conic proper closed subset in $M$ and $a \in M$.
We denote by $G^\circ \subset E^*$ the dual cone of $G$ in $E^*$, that is,
$$
G^\circ := \{\zeta \in E^*;\, \operatorname{Re}\, \langle \zeta, x\rangle \ge 0 \,\,
\text{ for any $x \in G$}\}.
$$
Assume $K = \overline{\{a\} + G} \subset \rcM$. Since
$\HHPC{K} = \widehat{\,\,\,\,}(\mathrm{int}\,G^\circ) \cap E^*_\infty$ and
$
h_K(\zeta) = \operatorname{Re}\,a\zeta
$
on $\HHPC{K}$ hold (here we write $a\zeta = \langle a,\zeta\rangle$),
the corollary immediately implies the following theorem.
\begin{teo} Under the above situation, $e^{a\zeta} \LL(u)(\zeta)$ belongs to 
	$\hinfo(\widehat{\,\,\,\,}(\mathrm{int}\,G^\circ) \cap E^*_\infty)$.
\end{teo}

\subsection{Several equivalent definitions of Laplace transform}{\label{subsec:comparison}}

We give, in this subsection, several equivalent definitions of Laplace transform
previously defined for various expressions of a Laplace hyperfunction.
The following proposition is quite important to obtain a good \v{C}ech representation of
a Laplace hyperfunction with compact support.
Recall the definition of a regular closed subset given in Definition {\ref{def:regular-set}}
and the one of an infinitesimal wedge in Definition {\ref{def:wedge}}.
Recall also that we use the word ``1-regular at $\infty$'' to indicate
the notion ``regular at $\infty$'' introduced in Definition 3.4 \cite{hu1}.

\begin{prop}{\label{prop:stein_regular_neighborhood}}
Let $K \ne \emptyset$ be a regular closed cone in $\rcM$ and let $\eta \in M^*_\infty$.
Then we can find an open
subset $S \subset \rcE \setminus K$ 
such that
\begin{enumerate}
\item $S$ is an infinitesimal wedge of type $M \widehat{\times} \sqrt{-1} \Gamma$,
where $\Gamma = \{y \in M;\, \langle y, \eta \rangle > 0\}$. 
\item $S \cap E$ is a Stein open subset and $S$ is $1$-regular at $\infty$.
\item $S$ is an open neighborhood of $\rcM \setminus K$ in $\rcE$.
\end{enumerate}
\end{prop}
\begin{proof}
The proof is the almost same as that of Theorem 4.10 \cite{KU}. 
For reader's convenience, we briefly explain how to construct
the desired $S$. We may assume that the vertex of $S$ is the origin
and $\eta=(1,0,\cdots,0)$. Let $\sigma$ be a sufficiently small
positive number and set, for $\xi \in M$,
$$
\varphi_\xi(z) = (z_1 - (\xi_1 + \sqrt{-1}\sigma|\xi|))^2
+(z_2 - \xi_2)^2 + \cdots + (z_n - \xi_n)^2 + \sigma^2|\xi|^2.
$$
Note that
$$
\textrm{Re}\,\varphi_\xi(z) > 0 \iff
(y_1 - \sigma|\xi|)^2 + y_2^2 + \cdots + y_n^2
< \sigma^2|\xi|^2 + |x - \xi|^2.
$$
Then, by the same reasoning as in the proof of Theorem 4.10 \cite{KU},  the set
$$
O = \textrm{Int}\left( \bigcap_{\xi \in K} \{z \in E;\,
\textrm{Re}\,\varphi_\xi(z) > 0\}\right) 
$$
is an $\mathbb{R}_+$-conic Stein open subset, and hence, $\widehat{O}$ is
$1$-regular at $\infty$. Define $S$ by modifying $O$ near the origin:
$$
S = \widehat{\,\,\,\,}\textrm{Int}\left( 
\left(
\bigcap_{\xi \in K,|\xi| \ge 1} \{z \in E;\,
\textrm{Re}\,\varphi_\xi(z) > 0\}
\right)
\, \bigcap\,
\left(\bigcap_{\xi \in K,|\xi| < 1} \{z \in E;\,
\textrm{Re}\,\psi_\xi(z) > 0\}
\right)\right),
$$
where
$$
\psi_\xi(z) = (z_1 - (\xi_1 + \sqrt{-1}\sigma))^2
+(z_2 - \xi_2)^2 + \cdots + (z_n - \xi_n)^2 + \sigma^2.
$$
Since $\widehat{O}$ and $S$ coincide in an open neighborhood of $E_\infty$, the $S$ is still $1$-regular at $\infty$
and $S \cap E$ is a Stein open subset. We can easily confirm that $S$ satisfies the rest of required properties in the proposition. 
\end{proof}

\subsubsection{Laplace transform for \v{C}ech representation}

We give here several examples to compute the Laplace transform 
of a \v{C}ech representative of a Laplace hyperfunction.
The following lemma is needed for subsequent examples.
\begin{lem}{\label{lem:c_infty_distance}}
Let $K \subset \rcM$ be a non-empty regular closed cone and $\Omega \subset \rcM$
an open neighborhood of $K$. Then there exist an open subset 
$U' \subset \rcM$ with smooth boundary satisfying $K \subset U' \subset \overline{U'} \subset \Omega$ and a smooth function $\varrho_{U'}: M \to \mathbb{R}$ such that
\begin{enumerate}
\item $\displaystyle\sum_{k=1}^n \left|\dfrac{\partial \varrho_{U'}}{\partial x_k}\right|$ is bounded on $M$.
\item there exists an open subset $T \subset \rcM$ with $\partial {U'} \subset T$
satisfying
$$
\varrho_{U'}(x) = \mathrm{dist}(x, M \setminus U')\qquad 
(x \in (\overline{U'} \cap T) \cap M).
$$
Furthermore, there exists $C > 1$ such that
$$
C^{-1}\mathrm{dist}(x, M \setminus U') \le \varrho_{U'}(x) \le 
C\mathrm{dist}(X, M \setminus U')\qquad (x \in \overline{U'} \cap M).
$$
\end{enumerate}
\end{lem}
\begin{proof}
By making an open cone with smooth boundary except for the vertex $a \in M$
(see Proposition 2.10 and Corollary 2.11 \cite{Ko})
and then by modifying the cone near $a$, we can find the open subset $U' \subset \rcM$
with smooth boundary such that $K \subset U' \subset \overline{U'} \subset \Omega$
which also satisfies the additional condition: there exists $R > 0$ such that
$$
\text{for any $c > 0$ and $x \in U'$ with $|x - a| > R$ and $|c(x-a)| > R$}
\Longrightarrow c(x-a) + a \in U'.
$$
Since $\partial U'$ is smooth and $U'$ satisfies the above additional condition, the function
$$
\tau(x) = \begin{cases}
\mathrm{dist}(x, M \setminus U') \qquad &(x \in \overline{U'} \cap M),\\
-\mathrm{dist}(x, U') \qquad &(x \in M \setminus U')
\end{cases}
$$
is smooth on $T \cap M$ for an open neighborhood $T \subset \rcM$ of $\partial U'$.
Taking open sets in $\rcM$
$$
\overline{U' \setminus T} \subset U_1 \subset \overline{U_1} \subset
U_2 \subset \overline{U_2} \subset U_3 \subset \overline{U_3} \subset U'.
$$
Let $\varphi_1$ is a $C^\infty$ function on $M$ with $0 \le \varphi_1 \le 1$ 
and bounded derivatives on $M$ such that
$\varphi_1(x) = 0$ in an open neighborhood of $(\overline{U' \setminus T}) \cap M$ and
$\varphi_1(x) = 1$ in an open neighborhood of $(\overline{U' \setminus U_1}) \cap M$.
In the same way, let $\varphi_2$ is a $C^\infty$ function on $M$ 
with $0 \le \varphi_2 \le 1$ and bounded derivatives on $M$ such that
$\varphi_2(x) = 0$ in an open neighborhood of $(\overline{U' \setminus U_3}) \cap M$
and $\varphi_1(x) = 1$ in an open neighborhood of $\overline{U_2} \cap M$.
Then, for $\epsilon > 0$, we define
$$
\varrho_{U'}(x) = \varphi_1(x) \tau(x) + \varphi_2(x) (\epsilon^{-n}\psi(x/\epsilon) * 
\mathrm{dist}(x, X \setminus U')),
$$
where $\psi(x)$ is a $C^\infty$ function on $M$ with $0 \le \psi \le 1$,
$\mathrm{supp}(\psi) \subset \{|x| < 1\}$ and $\displaystyle\int_{M} \varphi(x) dx = 1$.
If we take $\epsilon > 0$ sufficiently small, then $\varrho_{U'}(x)$ satisfies 
required conditions.
\end{proof}

\begin{es}{\label{es:l-c-example01}}
Let $K \ne \emptyset$ be a closed cone in $\rcM$ which is regular and proper, and
let $\eta_0,\dots, \eta_{n-1}$ be linearly independent vectors in $M^*$ so that
$\{\eta_0, \dots, \eta_{n-1}\}$ forms a positive frame of $M^*$.
Set $\eta_{n} := - (\eta_0 + \dots + \eta_{n-1}) \in M^*$.

Then, by applying Proposition \ref{prop:stein_regular_neighborhood} to the vector
$\eta_k$, we obtain $S_k$ satisfying the conditions in the proposition 
with $\eta = \eta_k$ ($k=0,\dots,n$).
Since $S_0 \cup \cdots \cup S_{n} \cup \rcM$ is an open neighborhood of $\rcM$,
it follows from Theorem 4.10 \cite{KU} that we can take an open neighborhood $S \subset \rcE$ of $\rcM$ such that
\begin{enumerate}
\item $S \cap E$ is a Stein open subset and it is $1$-regular at $\infty$.
\item $\{S_0 \cap S,\, S_1 \cap S,\, \dots,\, S_{n} \cap S\}$ 
is a covering of the set $S \setminus K$.
\end{enumerate}

For simplicity, we set $S_{n+1} := S$. Let $\Lambda = \{0,1,2,\dots,n+1\}$ and set,
for any $\alpha = (\alpha_0,\dots,\alpha_k) \in \Lambda^{k+1}$,
$$
S_\alpha := S_{\alpha_0} \cap S_{\alpha_1} \cap \dots \cap S_{\alpha_k}.
$$
We already defined the covering $(\mathcal{V}_K, \mathcal{V}'_K)$ of
$(S,\,S \setminus K)$ in \eqref{eq:def-V-K} with $V=S$.
We also define another covering of $(S,\,S \setminus K)$ by
$$
\mathcal{S} := \{S_0,S_1,\dots, S_{n+1}\},\qquad
\mathcal{S}' := \{S_0,\dots, S_{n}\}.
$$
Then, by the theories of the relative \v{C}ech and the relative \v{C}ech Dolbeault cohomologies, we have 
$$
{\textrm H}^n_{K}(S;\,\hexppo{n})
\simeq
{\textrm H}^n(C(\mathcal{S},\,\mathcal{S}')(\hexppo{n}))
\simeq
\HQDC{n}{n}{\mathcal{S}}
\simeq
\HQDC{n}{n}{\mathcal{V}_K}.
$$
Let $\Lambda^{k+1}_*$ be the subset in $\Lambda^{k+1}$ consisting of $\alpha
= (\alpha_0,\dots, \alpha_k)$ with
$$
\alpha_0 < \alpha_1 < \dots < \alpha_k = n+1.
$$
We take proper open subset $U' \subset \rcM$ with $K \subset U'$
and $\varrho_{U'}(x)$ given in Lemma \ref{lem:c_infty_distance}.
Assume $\epsilon > 0$ is sufficiently small. Then we define closed subsets in $E$ by
$$
\sigma_{n+1} :=
\bigcap_{0 \le k \le n}\overline{\{z = x+ \sqrt{-1}y \in E;\, x \in U' \cap E,\, 
\langle y,\, \eta_k \rangle < \epsilon\varrho_{U'}(x)\}} \,\,\bigcap\,\, E 
$$
and, for $0 \le k \le n$,
$$
\sigma_k :=
\overline{\{z = x+ \sqrt{-1}y \in E;\, x \in U' \cap E,\, 
\langle y,\, \eta_k \rangle > \epsilon\varrho_{U'}(x)\}} \,\, \bigcap\,\, E.
$$
We may assume that, by taking $\epsilon> 0$ sufficiently small,
\begin{equation}
\overline{\sigma_{n+1}} \cap \overline{\sigma_k} \subset S_k \qquad
(k=0,1,\cdots,n+1)
\end{equation}
holds in $\rcE$.
For any $\alpha = (\alpha_0,\dots, \alpha_k) \in \Lambda^{k+1}_*$, we also define
$$
\sigma_\alpha := \sigma_{\alpha_0} \cap \sigma_{\alpha_1} \cap \dots \cap 
\sigma_{\alpha_k}.
$$
Here we determine the orientation of $\sigma_\alpha$ in the following way:
\begin{enumerate}
	\item $\sigma_{n+1}$ has the same orientation as the one of $E$.
	\item For $k > 0$ and $\alpha \in \Lambda_*^{k+1}$, the vectors $(-\eta_{\alpha_0})$, $(-\eta_{\alpha_1})$, $\cdots$, $(-\eta_{\alpha_{k-1}})$ followed by the the positive frame of 
$\sigma_\alpha$ form a positive frame of $E$.
Note that $\alpha_k=n+1$ as $\alpha \in \Lambda_*^{k+1}$.
\end{enumerate}
\begin{oss}
The above 2. is equivalently saying that, for a point $x$  in the smooth part of $\sigma_\alpha$ and 
taking points  $x_j \in \textrm{int}(\sigma_{\alpha_j})$ ($j=0,1,\cdots,k$) sufficiently close to $x$, 
the positive frame of $\sigma_\alpha$ at $x$ is determined so that the vectors 
$\overrightarrow{x_0x_k}$, $\overrightarrow{x_1x_k}$, $\cdots$,
$\overrightarrow{x_{k-1}x_k}$ and the positive frame of $\sigma_\alpha$ at $x$ form
that of $E$ at $x$.
\end{oss}
Then, for any $\alpha \in \Lambda^{k+1}$ which contains the index $n+1$, we can define
$\sigma_\alpha$ with orientation by extending the above definition in the alternative way, that is,
$\sigma_{\alpha}=0$ if the same index appears twice in $\alpha$, and otherwise
$$
\sigma_{\alpha} = \operatorname{sgn}(\alpha, \tilde{\alpha})\, \sigma_{\tilde{\alpha}},
$$
where $\tilde{\alpha} \in \Lambda_*^{k+1}$ is obtained by a permutation of $\alpha$ and
$\operatorname{sgn}(\alpha, \tilde{\alpha})$ denotes the signature of this permutation.

Now let us consider the \v{C}ech Dolbeault complex 
$\QDC{n}{\bullet}{\mathcal{S}}$ for the covering $(\mathcal{S},\,\mathcal{S}')$.
Then, for any
$$
\omega = \{\omega_{\alpha}\}_{0\le k \le n,\, \alpha \in \Lambda_*^{k+1}} 
\in 
\bigoplus_{0 \le k \le n} C^k(\mathcal{S},\mathcal{S}';\,
\mathscr{Q}^{(n,n-k)}_{\rcE})
= \QDC{n}{n}{\mathcal{S}},
$$
we define the Laplace transform of $\omega$ by
$$
I(\omega) := \sum_{0 \le k \le n} \sum_{\alpha \in \Lambda_*^{k+1}}
\int_{\sigma_\alpha} e^{-z \zeta}\, \omega_{\alpha}.
$$
By our convention of orientation of $\sigma_{\alpha}$ and
the fact 
$$
\dim_{\mathbb{R}} \sigma_{n+1} \cap \{z = x+\sqrt{-1}y \in E;\,
x \in \partial U'\} < n, 
$$
we have, for any $\alpha \in \Lambda_*^{k+1}$,
$$
\partial \sigma_\alpha = \sum_{0 \le j \le n+1} \sigma_{[\alpha\, j]},
$$
where $[\alpha\,j]$ denotes a sequence in $\Lambda^{k+2}$ whose last element is $j$.

Hence it follows from Stokes's formula that we obtain
$$
I(\bvth \omega) = 0	\qquad (\omega \in \QDC{n}{n-1}{\mathcal{S}}).
$$
As a matter of fact, for 
$\omega_{\alpha} \in \mathscr{Q}^{n,n-k-1}_{\rcE}(S_\alpha)$ 
with $\alpha \in \Lambda_*^{k+1}$, we have
$$
e^{-z\zeta}\bvth \omega_{\alpha} = (-1)^k\overline{\partial} (e^{-z\zeta}\omega_{\alpha})
+ \delta (e^{-z\zeta}\omega_{\alpha})
= (-1)^kd (e^{-z\zeta}\omega_{\alpha}) + \delta  (e^{-z\zeta}\omega_{\alpha}),
$$
and thus, by noticing $\sigma_{[j\,\alpha]} = (-1)^{k+1}\sigma_{[\alpha\,j]}$,
$$
\begin{aligned}
I(\bvth \omega_{\alpha})
&=
(-1)^k\int_{\sigma_\alpha} d (e^{-z\zeta}\omega_{\alpha}) + 
\sum_{j=0}^{n+1} \int_{\sigma_{[j\,\alpha]}} e^{-z\zeta}\omega_{\alpha} \\
&=
(-1)^{k} \sum_{j=0}^{n+1} \int_{\sigma_{[\alpha\,j]}} e^{-z\zeta}\omega_{\alpha}
+
\sum_{j=0}^{n+1} \int_{\sigma_{[j\,\alpha]}} e^{-z\zeta}\omega_{\alpha}
=0.
\end{aligned}
$$
Summing up, if $\omega$ and $\omega'$ in $\QDC{n}{n}{\mathcal{S}})$
give the same cohomology class, we have
$$
I(\omega) = I(\omega').
$$
Now let us consider the canonical quasi-isomorphisms of complexes
$$
C(\mathcal{S},\,\mathcal{S}')(\hexppo{n})
\overset{\beta_1}{\longrightarrow}
\QQDC{n}{\mathcal{S}}
\overset{\beta_2}{\longleftarrow}
\QQDC{n}{\mathcal{V}_K},
$$
where $\beta_1$ is induced from the resolution $\hexppo{n} \to \QQDC{n}{\mathcal{S}}$
and $\beta_2$ is due to the fact that $(\mathcal{S},\mathcal{S}')$ 
is a covering finer than $(\mathcal{V}_K, \mathcal{V}'_K)$.
It is easy to see, for $\nu_2 \in \QDC{n}{n}{\mathcal{V}_K}$ with $\bvth \nu_2 = 0$,
$$
\LL([\nu_2]) = I(\beta_2(\nu_2)).
$$
Let $\nu_1 = \{ \nu_{1,\alpha}\}_{\alpha \in \Lambda^{n+1}_*}
\in C^n(\mathcal{S},\,\mathcal{S}')(\hexppo{n})$ with $\delta \nu_1 = 0$.
If $\beta_1(\nu_1)$ and $\beta_2(\nu_2)$ give the same cohomology class in
$\HQDC{n}{n}{\mathcal{S}}$, by the above reasoning, we get
$$
I(\beta_1(\nu_1)) = I(\beta_2(\nu_2)) = \LL([\nu_2]).
$$
It follows from the definition of $I(\bullet)$ that we have
$$
I(\beta_1(\nu_1)) = \sum_{\alpha \in \Lambda^{n+1}_*} 
\int_{\sigma_{\alpha}} e^{-z\zeta}\nu_{1,\alpha}.
$$
Furthermore, each integration can be rewritten to
\begin{equation}
\int_{\sigma_{\alpha}} e^{-z\zeta}\nu_{1,\alpha}
=
(-1)^n\operatorname{sgn}(\operatorname{det}(\eta_{\alpha_0},\dots,\eta_{\alpha_{n-1}}))
\int_{L_{\alpha}} e^{-z\zeta}\nu_{1,\alpha},
\end{equation}
where $L_\alpha$ is a real $n$-chain in $E$
\begin{equation}
L_\alpha = \{z = x + \sqrt{-1}y \in E;\, x \in \overline{U'} \cap M,\,\, 
y= \rho_\alpha(x)\}
\end{equation}
with a smooth function $\rho_\alpha: \overline{U'} \cap M \to M$ satisfying the conditions
\begin{enumerate}
\item $\rho_\alpha(x) = 0$ for $x \in \partial U' \cap M$,
\item $\overline{L_\alpha} \subset S_\alpha$ in $\rcE$,
\item $\displaystyle\sum_{k=1}^n \left| \dfrac{\partial \rho_\alpha}{\partial x_k}(x)\right|$ is bounded on $\overline{U'} \cap M$,
\end{enumerate}
and its orientation is the same as the one of $U'$. 

Summing up, for a \v{C}ech representation
$\{\nu_{1,\alpha}\}_{\alpha \in \Lambda_*^{n+1}}$ of a Laplace hyperfunction $u=[\tau]$, its Laplace transform is given by
\begin{equation}
\LL(u) =
(-1)^n\sum_{\alpha \in \Lambda_*^{n+1}}\operatorname{sgn}(\operatorname{det}(\eta_{\alpha_0},\dots,\eta_{\alpha_{n-1}})) \int_{L_{\alpha}} e^{-z\zeta}\nu_{1,\alpha}.
\end{equation}
\begin{oss}
In our settings, the last index of a covering
is assigned to the one for an open neighborhood $S$ of $\rcM$, i.e., $S_{n+1} = S$. 
In usual hyperfunction theory, however, the first index $0$ is assigned to it,
i.e., $S_0=S$.  This is the reason why the factor $(-1)^n$ appeared in the above expression.
\end{oss}
\end{es}

\

\begin{es}{\label{es:l-c-example02}}
Now we consider another useful example.
Set 
$$
\Gamma_{+^n} = \{y =(y_1,\cdots,y_n) \in M\,;\, y_k > 0\,\,(k=1,2,\cdots,n)\}
$$
and $K = \overline{\Gamma_{+^n}}$ in $\rcM$.
Let $S \subset \rcE$ be an open neighborhood of $\rcM$ such that
$S \cap E$ is a Stein open subset and $S$ is $1$-regular at $\infty$.
Define, for $k=0,1,\cdots,n-1$,
$$
S_k := \widehat{\,\,\,\,}\{z = (z_1,z_2,\cdots,z_n)\in S;\, z_{k+1} \in \mathbb{C} \setminus 
\mathbb{R}_{\ge 0}\}\,\, \subset \rcE.
$$
Set $S_n = S$. Then 
$$
\mathcal{S} = \{S_0,S_1, \dots, S_n\},\qquad
\mathcal{S}' = \{S_0, \dots, S_{n-1}\}
$$
are coverings of $(S,\, S \setminus K)$.

Define the $n\times n$ matrix $B: = (1 + \epsilon )I - \epsilon C$ for
sufficiently small $\epsilon > 0$,
where $I$ is the identity matrix and $C$ is the $n \times n$ matrix with entries being
all $1$. We define the $\mathbb{R}$-linear transformation $T$ on $E = M \times \sqrt{-1}M$ by
$$
x + \sqrt{-1}y \in E \longrightarrow B\,x + \sqrt{-1}\, y \in E.
$$
Let $\gamma \subset \mathbb{C}$ be the open subset defined by
$$
\gamma := \{z =x + \sqrt{-1}y \in \mathbb{C};\, |y| < \epsilon(x + \epsilon) \}.
$$
Then we introduce real $2n$-dimensional chains in $E$ by
$$
\sigma_n := \overline{T(\gamma \times \dots \times \gamma)}\,\, \bigcap \,\, E,
$$
and, for $k=0,\dots,n-1$,
$$
\sigma_k := \overline{T(\mathbb{C} \times \dots \times 
		\underset{\text{$(k+1)$-th}}{(\mathbb{C} \setminus \gamma)} \times 
		\dots \times \mathbb{C})}\,\,\bigcap \,\, E.
$$
Note that $\overline{\sigma_n}$ is a neighborhood of $K$ in $\rcE$. One should aware that,
however, $\overline{\gamma \times \gamma \times \cdots \times \gamma}$ is not.

Set $\Lambda = \{0,1,\dots,n\}$, and $\Lambda^{k+1}_*$ is the subset of $\Lambda^{k+1}$
consisting of an element $(\alpha_0, \alpha_1,\cdots,\alpha_k)$ with
$$
\alpha_0 < \alpha_1 < \cdots < \alpha_k=n.
$$
Then, for any $\alpha = (\alpha_0,\dots, \alpha_k) \in \Lambda^{k+1}_*$, 
the orientation of
$
\sigma_\alpha := \sigma_{\alpha_0} \cap \sigma_{\alpha_1} \cap \dots \cap 
\sigma_{\alpha_{k}}
$
is determined in the following way:
\begin{enumerate}
\item $\sigma_n$ has the same orientation as the one of $E$.
\item the outward-pointing normal vector of $\sigma_{\alpha_0}$, 
that of $\sigma_{\alpha_1}$, $\cdots$, that of $\sigma_{\alpha_{k-1}}$ followed by the the positive frame of $\sigma_\alpha$ form a positive frame of $E$.
\end{enumerate}

Note that, for any $\alpha \in \Lambda^{k+1}$ which contains the index $n$, we can define $\sigma_\alpha$ with orientation 
by extending the above definition in the alternative way as did
in the previous example.

For any $\alpha \in \Lambda_*^{k+1}$, we have
$$
\partial \sigma_\alpha = \sum_{0 \le j \le n} \sigma_{[\alpha\, j]},
$$
where $[\alpha\,j]$ is the sequence in $\Lambda^{k+2}$ whose last element is $j$.
Therefore the rest of argument goes in the same way as in 
Example {\ref{es:l-c-example01}}, and
we finally obtain, for $u = [\tau] \in \Gamma_K(\rcM;\,\bexpo \otimes_{\hexpa} \hexpv)$
and its \v{C}ech representative $\nu_{(012\dots n)} \in C^n(\mathcal{S},\,\mathcal{S}';\, 
\hexppo{n}) = \hexppo{n}(S_0 \cap S_1 \cap \dots \cap S_n)$, 
\begin{equation}
\LL(u) = \int_{L_{(012\dots n)}} e^{-z\zeta}\,\nu_{(012\dots n)}
\end{equation}
with the real $n$-chain 
\begin{equation}
L_{(012\dots n)} := T(\partial\gamma \times \dots \times \partial\gamma)
\subset E
\end{equation}
whose orientation is given so that each arc $\partial\gamma \subset \mathbb{C}$  has
anti-clockwise direction.
\end{es}

\

\begin{es}{\label{es:l-c-example03}}
Let us consider another kind of \v{C}ech covering: 
Let $K \ne \emptyset$ be a closed cone in $\rcM$
which is regular and proper in $\rcM$, and
let $\eta_k$'s ($k=0,\dots,n-1)$ be a family of linearly independent vectors in $M^*$, for which the sequence $\eta_0$, $\eta_1$, $\cdots$, $\eta_{n-1}$ of vectors forms a positive frame of $M^*$. Set
$$
\eta_{k,\pm} = \pm \eta_k\qquad(k=0,\dots,n-1).
$$
Then, we take open subsets $S$ and $S_{k,\pm}$ $(k=0,1,\cdots,n-1)$ 
in the same way as those in Example \ref{es:l-c-example01} by using 
Proposition \ref{prop:stein_regular_neighborhood} with $\eta = \eta_{k,\pm}$.
Set $S_n = S$ and coverings
$$
\mathcal{S} := \{S_{0,\pm},\dots, S_{n-1,\pm}, S_n\},\qquad
\mathcal{S}' := \{S_{0,\pm},\dots, S_{n-1,\pm}\}.
$$
Let $\Lambda$ be the set consisting of ``$n$'' and pairs ``$(i, \epsilon)$'' with
$i \in \{0,1,\dots,n-1\}$ and $\epsilon \in \{+,-\}$.
We define the linear order $<$ on $\Lambda$ by:
\begin{itemize}
\item[a.]\,\, $\alpha < n$ for any $\alpha \in \Lambda \setminus \{n\}$.
\item[b.]\,\, $(i,e_i) < (j,e_j)$ if $i < j$ or if $i = j$ and $e_i = +$ and $e_j=-$.
\end{itemize}
Let $\Lambda^{k+1}_{*}$ be the  subset in $\Lambda^{k+1}$ consisting of
$\alpha = (\alpha_0,\dots,\alpha_{k})$ with 
$$
\alpha_0 < \alpha_1 < \dots < \alpha_k = n.
$$
Furthermore, let $\Lambda^{k+1}_{**}$ be the subset in $\Lambda^{k+1}_*$ consisting of
$$
\alpha = ((i_0,\epsilon_0),\, \cdots,\,(i_{k-1},\epsilon_{k-1}),\, n) \in \Lambda^{k+1}_*
$$  
with $i_0 < i_1 < \dots < i_{k-1}$. For $\alpha \in \Lambda_*^{k+1}$, the subset $S_\alpha$
is defined as usual, that is,
$$
S_\alpha = S_{\alpha_0} \cap \cdots \cap S_{\alpha_k}.
$$
Note that, in this example, 
the open subset $S_\alpha$ is not necessarily empty for $\alpha \in \Lambda^{k+1}_* \setminus \Lambda^{k+1}_{**}$ with $k > n$. 

We take a proper open subset $U' \subset \rcM$ with $K \subset {U'}$
and $\varrho_{U'}(x)$ given in Lemma {\ref{lem:c_infty_distance}}.
Assume $\epsilon > 0$ is sufficiently small. Then we define closed subsets in $E$  by
$$
\sigma_n =
\bigcap_{0 \le k \le n-1} \overline{\{z = x+ \sqrt{-1}y \in E;\, x \in U' \cap E,\, 
|\langle y,\, \eta_k \rangle| < \epsilon\varrho_{U'}(x)\}}\,\,\bigcap\,\, E
$$
and, for $0 \le k \le n-1$,
$$
\sigma_{(k,\pm)} =
\overline{\{z = x+ \sqrt{-1}y \in E;\, x \in U' \cap E,\, 
\pm\langle y,\, \eta_k \rangle > \epsilon\varrho_{U'}(x)\}}\,\, \bigcap\,\, E.
$$
Note that $\overline{\sigma_n} \cap \overline{\sigma_\alpha} \subset S_\alpha$ holds for $\alpha \in \Lambda$
if $\epsilon$ is sufficiently small.
Then, in the same way as in the previous example,
we can define $\sigma_\alpha$ for $\alpha \in \Lambda_{**}^{k+1}$ and 
determine its orientation. For any
$$
\omega = \{\omega_{\alpha}\}_{0\le k \le n,\, \alpha \in \Lambda_*^{k+1}} 
\in 
\bigoplus_{0 \le k \le n} C^k(\mathcal{S},\mathcal{S}';\,
\mathscr{Q}^{(n,n-k)}_{\rcE})
= \QDC{n}{n}{\mathcal{S}},
$$
we define the Laplace transform of $\omega$ by
$$
I(\omega) := \sum_{0 \le k \le n} \sum_{\alpha \in \Lambda_{**}^{k+1}}
\int_{\sigma_\alpha} e^{-z \zeta}\, \omega_{\alpha},
$$
for which one should aware that the sum ranges through indices only in $\Lambda_{**}^{k+1} \subset \Lambda_*^{k+1}$.

We have, for any $\alpha = ((i_0,e_0),\,\cdots,\,
(i_{k-1},\,e_{k-1}),\, n) \in \Lambda_{**}^{k+1}$,
$$
\partial \sigma_\alpha = \sum_{j \notin \{i_0,\cdots,i_{k-1},n\},\, \epsilon = \pm}
	\sigma_{[\alpha\, (j,\,\epsilon)]},
$$
where $[\alpha\, (j,\,\epsilon)]$ is a sequence in $\Lambda^{k+2}$ whose last element
is $(j,\,\epsilon)$.
The important fact here is that $\partial \sigma_\alpha$ 
($\alpha \in \Lambda_{**}^{k+1}$)
does not contain any cell $\sigma_\beta$ with $\beta \in \Lambda^{k+2}_{*}\setminus\Lambda_{**}^{k+2}$.
Hence, by Stokes's formula, we still obtain
$$
I(\bvth \omega) = 0	\qquad (\omega \in \QDC{n}{n-1}{\mathcal{S}}).
$$
As a matter of fact,
if $\alpha \in \Lambda_{*}^{k+1} \setminus \Lambda_{**}^{k+1}$, then
$
I(\bvth \omega_{\alpha}) = 0
$
for $\omega_\alpha \in \mathscr{Q}^{(n,n-k-1)}_{\rcE}(S_\alpha)$
because $\overline{\partial} \omega_{\alpha}$  (resp. $\delta \omega_{\alpha}$)
does not contain a non-zero term with an index in $\Lambda_{**}^{k+1}$
(resp. $\Lambda_{**}^{k+2}$).
If 
$\alpha = ((i_0,e_0),\,\cdots,\, (i_{k-1},\,e_{k-1}),\, n) \in \Lambda_{**}^{k+1}$, 
then we have for $\omega_\alpha \in \mathscr{Q}^{(n,n-k-1)}_{\rcE}(S_\alpha)$
$$
\begin{aligned}
I(\bvth \omega_{\alpha})
&=
(-1)^k\int_{\sigma_\alpha} d(e^{-z\zeta}\omega_{\alpha}) + 
\sum_{j \notin \{i_0,\cdots,i_{k-1},n\},\, \epsilon = \pm}
\int_{\sigma_{[(j,\epsilon)\,\alpha]}} e^{-z\zeta}\omega_{\alpha} 
\\
&=
(-1)^{k} 
\sum_{j \notin \{i_0,\cdots,i_{k-1},n\},\, \epsilon = \pm}
\int_{\sigma_{[\alpha\,(j,\epsilon)]}} e^{-z\zeta}\omega_{\alpha}
+
\sum_{j \notin \{i_0,\cdots,i_{k-1},n\},\, \epsilon = \pm}
\int_{\sigma_{[(j,\epsilon)\,\alpha]}} e^{-z\zeta}\omega_{\alpha} 
=0.
\end{aligned}
$$

The rest of argument is the same as the one in the previous example:
For $u = [\tau] \in \Gamma_K(\rcM;\,\bexpo \otimes_{\hexpa} \hexpv)$ and its relative \v{C}ech representative 
$$
\nu = \underset{\alpha \in \Lambda^{n+1}_{*}}{\bigoplus} \nu_{\alpha} \in 
C^n(\mathcal{S},\,\mathcal{S}')(\hexppo{n})
\quad \text{with}\,\, \delta \nu = 0,
$$
we obtain
\begin{equation}
\LL(u) = (-1)^n\sum_{\alpha \in \Lambda^{n+1}_{**}}
\textrm{sgn}(\alpha)\int_{L_{\alpha}} e^{-z\zeta}\nu_{\alpha}.
\end{equation}
Here, for $\alpha=((0,\epsilon_0),\dots,(n-1,\epsilon_{n-1}),\, n) \in \Lambda^{n+1}_{**}$, we set $\textrm{sgn}(\alpha)=\epsilon_0\epsilon_1 \cdots \epsilon_{n-1}$ and
$L_\alpha$ is the real $n$-chain in $E$
\begin{equation}
L_\alpha = \{z = x + \sqrt{-1}y \in E;\, x \in \overline{U'} \cap M,\,\, 
y= \rho_\alpha(x)\}
\end{equation}
with a smooth function $\rho_\alpha: \overline{U'} \cap M \to M$ satisfying the conditions
\begin{enumerate}
\item $\rho_\alpha(x) = 0$ for $x \in \partial U' \cap M$,
\item $\overline{L_\alpha} \subset S_\alpha$ in $\rcE$,
\item $\displaystyle\sum_{k=1}^n \left| \dfrac{\partial \rho_\alpha}{\partial x_k}(x)\right|$ is bounded on $\overline{U'} \cap M$,
\end{enumerate}
and its orientation is the same as the one of $U'$.
\end{es}

\subsubsection{Laplace transform whose chain is of product type}{\label{subsec:laplace_prodcut}}

Let us consider the Laplace transformation of a Laplace hyperfunction $u$ whose
support is contained in $\overline{\Gamma_{+^n}} \subset \rcM$. 
Here $\Gamma_{+^n} = \{(x_1,\cdots,x_n) \in M;\, x_k > 0 \,\,(k=1,2,\cdots,n)\}$.
In this case, one can expect
the the path of the integration to be the product 
$\gamma_1 \times \cdots \times \gamma_n$ of the one dimensional paths $\gamma_k$.
However, we cannot take such a path unless the support of $u$ is
contained in a cone strictly smaller than $\overline{\Gamma_{+^n}}$.
In this subsection, we show that a chain of product type can be taken as an integral
path of the Laplace transformation if the condition ${\operatorname{supp}(u)}\setminus \{0\} \subset \widehat{\Gamma_{+^n}}$ is satisfied.

\

Let $K \ne \emptyset$ be a regular closed cone in $\rcM$ satisfying
\begin{equation}{\label{eq:proper-condition-K}}
	K \setminus \{0\} \subset \widehat{\Gamma_{+^n}}.
\end{equation}
Let $\epsilon > 0$ and Let $D_k \subset \mathbb{C}$ be an open subset 
satisfying that
$\widehat{D_k}$ has a good boundary (see Remark \ref{oss:good_boundary} for the definition of a good boundary), $\overline{\mathbb{R}_+} \subset \widehat{D_k}$ and
$$
D_k \subset \{z = x + \sqrt{-1}y \in \mathbb{C}\,;\, |y| < \epsilon(x + \epsilon)\}.
$$
Set
$$
D = D_1 \times D_2 \times \cdots \times D_n \subset E.
$$
One should aware that
$\widehat{D}$ is not an open neighborhood of $\overline{\Gamma_{+^n}}$ in $\rcE$.
However, since $\widehat{D}$ becomes an open neighborhood of $K$ in $\rcE$ because of 
\eqref{eq:proper-condition-K}, 
we can compute its Laplace transform by
$$
\LL(u)(\zeta) := \int_D e^{-z \zeta} \nu_1 - \int_{\partial D} e^{-z \zeta}\nu_{01}
$$
for a Laplace hyperfunction $u = [(\nu_1, \nu_{01})]$ 
($(\nu_1,\,\nu_{01}) \in \QDC{n}{n}{\mathcal{V}_K}$)
with support in $K$.

Let $\eta_{k,\pm}=(0,\cdots,\pm 1,\cdots,0)$ ($k=0,\dots,n-1)$ be a unit vector
whose $(k+1)$-th element is $\pm 1$.
Recall the definitions of $\Lambda_*^{k+1}$ and $\Lambda_{**}^{k+1}$
given in Example \ref{es:l-c-example03}, and
let us introduce open subsets $S$, $S_{k,\pm}$ and the pair 
$(\mathcal{S}, \mathcal{S}')$ of coverings of $(S, S \setminus K)$
in the same way as those in Example 
\ref{es:l-c-example03}.
Set $\sigma_n = \overline{D} \cap E$ and, for $k=0,\cdots,n-1$,
$$
\sigma_{k,\pm} =
\overline{\{z = (z_1,\cdots, z_n) \in E\,;\, 
	z_{k+1} \in \mathbb{C} \setminus D_{k+1},\, \pm\textrm{Im}\, z_{k+1} \ge 0 \}
} \,\,\bigcap\,\, E.
$$
Then, as we did in the example, we define the Laplace transform by
$$
I(\omega) := \sum_{0 \le k \le n} \sum_{\alpha \in \Lambda_{**}^{k+1}}
\int_{\sigma_\alpha} e^{-z \zeta}\, \omega_{\alpha}
$$
for  
$
\omega = \{\omega_{\alpha}\}_{0\le k \le n,\, \alpha \in \Lambda_*^{k+1}} 
\in 
\bigoplus_{0 \le k \le n} C^k(\mathcal{S},\mathcal{S}';\,
\mathscr{Q}^{(n,n-k)}_{\rcE})
= \QDC{n}{n}{\mathcal{S}}.
$

Note that we have, for $\alpha = ((i_0,\epsilon_0),\cdots,(i_{k-1},\epsilon_{k-1}),n) \in \Lambda_*^{k+1}$,
$$
\partial \sigma_\alpha = \sum_{j \notin \{i_0,\dots,i_{k-1},n\},\, \epsilon = \pm}
	\sigma_{[\alpha\, (j,\,\epsilon)]}
	+ \sum_{j \in \{i_0,\dots,i_{k-1},n\},\, \epsilon = \pm}
	\sigma_{[\alpha\, (j,\,\epsilon)]}.
$$
Define $\pi_j: \mathbb{C}^n \to \mathbb{C}$ to be $\pi_j(z_1,\cdots,z_n) = z_{j+1}$.
Since 
$
\pi_j(\sigma_{j,+} \cap \sigma_{j,-} \cap \sigma_n)
$
$(j=0,1,\cdots,n-1)$
consists of the one point,
for  $j \in \{i_0,\dots,i_{k-1}\}$ and $\epsilon = \pm$, 
the restriction of the holomorphic $n$-form $dz$ to $\sigma_{[\alpha\, (j,\,\epsilon)]}$ becomes $0$
and we get
$$
\int_{\sigma_{[\alpha\, (j,\,\epsilon)]}} e^{-z\zeta}\tau = 0
$$
for an $(n,n-k-1)$-form $\tau$.
Therefore we still have the same Stokes formula as the one
in Example \ref{es:l-c-example03}
$$
\int_{\sigma_{\alpha}} 
e^{-z\zeta} 
\overline{\partial}
\tau 
= \sum_{j \notin \{i_0,\dots,i_{k-1},n\},\, \epsilon = \pm}
\int_{\sigma_{[\alpha\,(j,\epsilon)]}}e^{-z\zeta} \tau,
$$
and hence, we obtain
$$
I(\bvth \omega) = 0	\qquad (\omega \in \QDC{n}{n-1}{\mathcal{S}}).
$$
Summing up, let $u = [\tau] \in \Gamma_K(\rcM;\,\bexpo \otimes_{\hexpa} \hexpv)$
and let 
$$
\nu = \underset{\alpha \in \Lambda^{n+1}_{*}}{\bigoplus} \nu_{\alpha} \in 
C^n(\mathcal{S},\,\mathcal{S}')(\hexppo{n})\quad\text{with}\,\,\, \delta \nu = 0
$$
be its \v{C}ech representative, that is, through the diagram of isomorphisms
$$
\mathrm{H}^n(C(\mathcal{S},\,\mathcal{S}')(\hexppo{n}))
\overset{\beta_1}{\longrightarrow}
\HQDC{n}{n}{\mathcal{S}}
\overset{\beta_2}{\longleftarrow}
\HQDC{n}{n}{\mathcal{V}_K},
$$
$\tau$ and $\nu$ satisfy $[\tau] = ((\beta_2^{n})^{-1} \circ \beta_1^n)([\nu])$.
We have obtained
\begin{equation}
\LL(u) = (-1)^n\sum_{\alpha \in \Lambda^{n+1}_{**}} 
\textrm{sgn}(\alpha)\int_{\gamma_{\alpha}} 
e^{-z\zeta}\nu_{\alpha},
\end{equation}
where, for $\alpha=((0,\epsilon_0),\dots,(n-1,\epsilon_{n-1}),\, n) 
\in \Lambda^{n+1}_{**}$, we set
$\textrm{sgn}(\alpha)= \epsilon_0\epsilon_1 \dots \epsilon_{n-1}$,
\begin{equation}
\begin{aligned}
\gamma_{\alpha}
&= \left(\partial D_1 \times \partial D_2 \times \cdots \times \partial D_{n}\right)\,
\bigcap\,\overline{\Gamma_{\alpha}}, \\
\Gamma_{\alpha} &=
\{z=(z_1,\cdots,z_n) \in E\,;\, \epsilon_{k}\textrm{Im}\,z_{k+1} > 0 \quad (k=0,1,\cdots,n-1)\}
\end{aligned}
\end{equation}
and the orientation of $\gamma_{\alpha}$
is chosen to be the same as the one in $M$.

\subsection{Reconstruction of a representative}

By the same arguments as in the previous examples, we have a formula
to reconstruct the corresponding \v{C}ech representative 
from a \v{C}ech Dolbeault representative of a Laplace hyperfunction.

Recall the definition of $\Lambda_*^{n+1}$ and $\Lambda_{**}^{n+1}$ given in
Example \ref{es:l-c-example03}.
Set
$$
\Gamma_\alpha := \{x \in M;\, \epsilon_k x_{k+1} > 0\,\,\,(k=0,\dots,n-1)\}
$$
for any $\alpha = ((0,\epsilon_0),(1,\epsilon_1), \cdots, (n-1,\epsilon_{n-1}),n)
\in \Lambda_{**}^{n+1}$.
In particular, we denote by $+^n$ the sequence $((0,+),(1,+), \cdots, (n-1,+),n)$.
Thus $\Gamma_{+^n}$ denotes the first orthant in $M$.

Let $K \subset \rcM$ be a regular closed cone such that $K \cap M$ is a non-empty convex set, and $V \subset \rcE$ an open cone such that
$V$ is $1$-regular at $\infty$ and $V \cap E$ is a Stein open subset.
Note that, since $V$ is an open cone, the fact that $V$ is $1$-regular at $\infty$ is equivalent to saying that $\widehat{\,\,\,\,}(V \cap E) = V$.
We also assume 
\begin{equation}
K \setminus \{0\} \,\subset\, \widehat{\Gamma_{+^n}} \,\subset\,
\overline{\Gamma_{+^n}}\,\subset\, V.
\end{equation}
Set $U = V \cap \rcM$ and we also assume $\widehat{\,\,\,\,}(U \cap M)=U$. Let $\ILHU$ denote the intuitive representation of Laplace hyperfunctions on $U \subset \rcM$. 
\begin{oss}
In this subsection, 
we assume that $\mathcal{W}({U})$ consists of
an infinitesimal wedge which satisfies the condition B1.~in Section
{\ref{sec:boundary-value}}. For such a family $\mathcal{W}({U})$
of restricted open subsets, still Theorem {\ref{thm:intuitive}} holds.
\end{oss}

Then, we define
$
b: \ILHU \to \HQDC{0}{n}{\mathcal{V}_{U}}
$
by
$$
\hexpo(W) \ni f \, \mapsto\, b_W(f) \in \HQDC{0}{n}{\mathcal{V}_U} \qquad
(W \in \mathcal{W}(\hat{U})),
$$
where $\mathcal{V}_U = \{V \setminus {U},\,V\}$, 
$\mathcal{V}'_U = \{V \setminus {U}\}$ 
and $b_W$ is the boundary value map {\eqref{eq:def_boundary_map}}.

Recall that the isomorphism $b_{\mathcal{W}}: \ILHU \to \Gamma({U};\, \bexpo)$
was given in Theorem {\ref{thm:intuitive}}, for which
we have the commutative diagram (see Theorem \ref{teo:b_omega_commute} and
Corollary {\ref{cor:commute_b_omega_xx}}):
$$
\begin{diagram}
	\node[2]{\Gamma({U};\, \bexpo)}
	\node{\Gamma_K({U};\, \bexpo)}\arrow{w,t}{\iota}   \\
\node{\ILHU} \arrow{e,t}{b} \arrow{ne,l}{b_{\mathcal{W}}}
\node{\HQDC{0}{n}{\mathcal{V}_U}}  \arrow{n,l}{\simeq}
\node{\HQDC{0}{n}{\mathcal{V}_K}} \arrow{w,t}{\iota} \arrow{n,l}{\simeq}
\end{diagram},
$$
where $\mathcal{V}_K = \{V \setminus K,\,V\}$ and $\mathcal{V}_K' = \{V \setminus K\}$,
the morphisms $\iota$ are injective and all the other morphisms are isomorphic.
Set
$$
\ILHKU := \{u \in \ILHU;\,
\operatorname{Supp}(b_{\mathcal{W}}(u)) \subset K \}.
$$
Then the morphism $b$ induces the isomorphism
$$
b_K: \ILHKU \overset{\sim}{\longrightarrow} \HQDC{0}{n}{\mathcal{V}_K}.
$$

Now we give the inverse of $b_K$ concretely.
Let $u \in \Gamma_K({U};\, \bexpo)$ and $\tau = (\tau_1,\tau_{01}) \in \QDC{0}{n}{\mathcal{V}_K}$ be its representation.
Define
$$
h_\tau(z) = \dfrac{1}{(2\pi\sqrt{-1})^n}
\left(\int_D \dfrac{\tau_1(w) e^{(z-w)a}}{w - z}dw -
\int_{\partial D} \dfrac{\tau_{01}(w) e^{(z-w)a}}{w - z} dw
\right),
$$
where 
$\dfrac{1}{w - z}$ denotes 
$\dfrac{1}{(w_1-z_1)\cdots(w_n-z_n)}$, 
the vector $a$ and the domain $D$ are as follows:
\begin{enumerate}
\item $D$ is a contractible open subset in $\rcE$ with a good boundary 
$\partial D$ 
(see Remark {\ref{oss:good_boundary}} for the good boundary)
which satisfies 
$$
K \subset D \subset \overline{D} \subset V
$$
and 
$$
(D \cap E) \subset \bigcup_{k=1}^n\,\{w \in E;\, |w_k - z_k| > \delta\}
$$
for some $\delta > 0$. Furthermore, $D$ is properly contained in an half space of $\rcE$
with direction $\dfrac{1}{\sqrt{n}}(1,1,\cdots,1)$.
\item $a = R(1,1,\dots,1)$, where $R > 0$ is sufficiently large so that the integrals converge.
\end{enumerate}
Note that the orientation of $D$ is the same as the one of $E$, and that of $\partial D$
is determined so that the outward-pointing normal vector of $\partial D$ followed by
a positive frame of $\partial D$ form a positive frame of $E$.

Then it is easy to check that $h_\tau(z)$ remains unchanged when we take another $D$ and
representative $\tau$ of $u$ if the integral converges for the same $a$.
Hence, by deforming $D$ suitably 
(here keep $D$ unchanged near $K \cap E_\infty$, and hence, we do not
need to change $a$ in this deformation), we find that $h_\tau(z)$ belongs to
$\hexpo(\Omega)$, where 
$$
\Omega := \widehat{\,\,\,\,}\{z = x+\sqrt{-1}y \in E;\, y_1 y_2 \cdots y_n \ne 0\}.
$$
For $\alpha \in \Lambda_{**}^{n+1}$, set $\Omega_\alpha := 
M \widehat{\times} \sqrt{-1}\Gamma_\alpha \subset \rcE$.
Note that we have
$$
\Omega := \underset{\alpha \in \Lambda_{**}^{n+1}}{\bigsqcup} \Omega_\alpha.
$$
Now we define the inverse $b^\dagger_K$ of $b_K$ by
\begin{equation}
u = [\tau] \longrightarrow \bigoplus_{\alpha \in \Lambda_{**}^{n+1}}\, 
\operatorname{sgn}(\alpha) h_\tau(z)\big|_{\Omega_\alpha}
\in \ILHU,
\end{equation}
where $\operatorname{sgn}(\alpha) = \epsilon_0 \epsilon_1 \cdots \epsilon_{n-1}$
for $\alpha = ((0,\epsilon_0),(1,\epsilon_1), \cdots, (n-1,\epsilon_{n-1}),n) \in \Lambda_{**}^{n+1}$.
\begin{lem}
	$b_K^{\dagger}$ is independent of the choices of $a = R(1,\dots,1)$ if $R > 0$ is
	sufficiently large.
\end{lem}
\begin{proof}
Let $a' = (R', R,\dots,R)$ with $R' > R$. It is enough to show that
$b_K^{\dagger}(u)$ gives the same result for both the $a$ and $a'$ because a general case
is obtained by the repetition of application of this result.
Clearly we have
$$
\begin{aligned}
&\left(\int_D \dfrac{\tau_1(w) e^{(z-w)a'}}{w - z} dw-
\int_{\partial D} \dfrac{\tau_{01}(w) e^{(z-w)a'}}{w - z} dw
\right) - \\
&\qquad\qquad\left(\int_D \dfrac{\tau_1(w) e^{(z-w)a} }{w - z} dw -
	\int_{\partial D} \dfrac{\tau_{01}(w) e^{(z-w)a}}{w - z} dw
\right) \\
&=(R'-R)
\left(\int_D \int_0^1\dfrac{\tau_1(w) e^{(z-w)(ta' + (1-t)a)}}{w' - z'} dtdw-
\int_{\partial D} \int_0^1\dfrac{\tau_{01}(w) e^{(z-w)(ta' + (1-t)a)}}{w' - z'} dtdw
\right),
\end{aligned}
$$
where $z' = (z_2,\dots,z_n)$ and $w'=(w_2,\dots,w_n)$.
Since the last integral denoted by $\tilde{h}(z)$ hereafter
belongs to $\hexpo(\Omega')$ with
$$
\Omega' := \widehat{\,\,\,\,}\{z = x+\sqrt{-1}y \in E;\, y_2 \cdots y_n \ne 0\},
$$
we have $\displaystyle\bigoplus_\alpha\, \operatorname{sgn}(\alpha) \tilde{h}(z)\big|_{\Omega_\alpha} = 0$
in $\ILHU$.  This shows the result.
\end{proof}
\begin{teo}{\label{teo:b_k_invserse}}
	$b_K$ and $b_K^{\dagger}$ are inverse to each other. 
\end{teo}
\begin{proof}
We use the same notations as those in Subsection {\ref{subsec:laplace_prodcut}}, 
where we take an open subset $V$ as $S$. Hence, the pair $(\mathcal{S},\mathcal{S}')$
are coverings of $(V, V \setminus K)$. Set 
$$
Q_{k,\epsilon} = \{y=(y_1,\cdots,y_n) \in M\,;\,  \epsilon y_{k+1} > L^{-1}|y|\}
\qquad (k=0,1,\cdots,n-1, \epsilon=\pm) 
$$
for sufficiently large $L>0$ and set
$$
T_{k,\epsilon} = U \widehat{\times} \sqrt{-1}Q_{k,\epsilon}.
$$
Let $T \subset \rcE$ be an open neighborhood of ${U}$ such that
$T \cap \rcM =U$,
$T$ is $1$-regular at $\infty$ and $T \cap E$ is a Stein open subset.
Furthermore, by shirking $T$ if necessary, we may assume
$T \subset S$ and
$$
T_{k,\epsilon} \cap T \subset S_{k,\epsilon} \cap S
\qquad (k=0,1,\cdots,n-1,\,\epsilon = \pm ).
$$
Set also $T_n = T$ and define the pair $(\mathcal{T}, \mathcal{T}')$ of coverings of 
$(T, T \setminus \rcM)$ by
$$
\mathcal{T} = \{T_{0,+},\, T_{0,-},\,\cdots,\, T_{n-1,+},\,T_{n-1,-},\, T_n\},\qquad
\mathcal{T}' = \{T_{0,+},\, T_{0,-},\,\cdots,\, T_{n-1,+},\, T_{n-1,-}\}.
$$
Using these coverings, we have the commutative diagram of complexes,
where the horizontal arrows are all quasi-isomorphisms:
$$
\begin{diagram}
\node{C(\mathcal{S},\,\mathcal{S}')(\hexpo)}
\arrow{e,t}{\beta_1}
\arrow{s,l}{\iota_1}
\node{\QQDC{0}{\mathcal{S}}}
\arrow{s,l}{\iota}
\node{ \QQDC{0}{\mathcal{V}_K}} 
\arrow{w,t}{\beta_2}
\arrow{s,l}{\iota_2}
\\
\node{ C(\mathcal{T},\,\mathcal{T}')(\hexpo)}
\arrow{e,t}{\alpha_1}
\node{ \QQDC{0}{\mathcal{T}}}
\node{ \QQDC{0}{\mathcal{V}_{U}}}
\arrow{w,t}{\alpha_2}
\end{diagram}.
$$
Then by taking $n$-th cohomology groups we get
$$
\begin{diagram}
	\node{{\textrm H}^n(C(\mathcal{S},\,\mathcal{S}')(\hexpo))}
\arrow{e,t}{\beta_{1}^n}
\arrow{s,l}{\iota_{1}^n}
\node{\HQDC{0}{n}{\mathcal{S}}}
\arrow{s,l}{\iota^n}
\node{\HQDC{0}{n}{\mathcal{V}_K}} 
\arrow{w,t}{\beta_2^n}
\arrow{s,l}{\iota_2^n}
\\
\node{ {\textrm H}^n(C(\mathcal{T},\,\mathcal{T}')(\hexpo))}
\arrow{e,t}{\alpha_1^n}
\node{ \HQDC{0}{n}{\mathcal{T}}}
\node{ \HQDC{0}{n}{\mathcal{V}_{U}}}
\arrow{w,t}{\alpha_2^n}
\end{diagram},
$$
where all the horizontal arrows are isomorphic and
all the vertical arrows are injective.
Before entering the proof, we first confirm several fundamental facts
which are needed in the proof: 

\begin{itemize}
\item Lemma {\ref{lem:b_w_and_I_C_commute}} still holds by the same reasoning, that is,
the canonical isomorphism
$$
\iota_{IC}: {\textrm H}^n(C(\mathcal{T},\,\mathcal{T}')(\hexpo))
\xrightarrow{\,\,\sim\,\,} 
\hat{\textrm H}^n(\hexpo(\mathcal{W}(U)))
$$
is given by
$$
\left[\underset{\alpha \in \Lambda_{**}^{n+1}}{\oplus}\, g_\alpha \right]
\mapsto (-1)^n\bigoplus_{\alpha \in \Lambda_{**}^{n+1}}
	\operatorname{sgn}(\alpha) g_\alpha|_{T_\alpha}.
$$
\item 
We have the commutative diagram below as in {\eqref{eq:b_w_and_diag_equiv}}.
$$
\begin{diagram}
\node{ {\textrm H}^n(C(\mathcal{T},\,\mathcal{T}')(\hexpo))} 
\arrow{s,l}{\iota_{IC}}\arrow{se,l}{(\alpha^n_2)^{-1} \circ \alpha_1^n} \\
\node{\hat{\textrm H}^n(\hexpo(\mathcal{W}(U)))}
\arrow{e,t}{b} \node{\HQDC{0}{n}{\mathcal{V}_{U}}}
\end{diagram}.
$$
%
\item The morphism $\iota_1$ is induced from the restriction of coverings, that is,
for $(f_{\alpha})_{\alpha \in \Lambda_{*}^{n+1}} \in 
C^n(\mathcal{S},\,\mathcal{S}')(\hexpo)$, we have
$$
\iota_1^n([(f_{\alpha})_{\alpha \in \Lambda_{*}^{n+1}}])
= [(f_\alpha|_{T_\alpha})_{\alpha \in \Lambda_{**}^{n+1}}]
\qquad\text{in ${\textrm H}^n(C(\mathcal{T},\,\mathcal{T}')(\hexpo))$}.
$$
Note also that, for $\alpha \in \Lambda_*^{n+1} \setminus \Lambda_{**}^{n+1}$,
we have always $T_\alpha = \emptyset$ but $S_\alpha$ is not necessarily empty.
\end{itemize}

Since $b_K$ is an isomorphism, it suffices to show $b_K^\dagger$ is
the inverse of $b_K$, i.e., 
$$
b_K^\dagger \circ b_K = \mathrm{id}  \qquad 
\text{in $\ILHU$}.
$$
To see the above formula,
for any cocycle $f = (f_{\alpha})_{\alpha \in \Lambda_{*}^{n+1}} \in C^n(\mathcal{S},\,\mathcal{S}')(\hexpo)$, it suffices to show the equality
$$
(b_K^\dagger \circ b_K \circ \iota_{IC}\circ \iota^n_1)([f]) = 
(\iota_{IC}\circ \iota^n_1)([f]) \qquad
\text{in $\ILHU$}.
$$
Then it follows from the above fundamental facts that we have
$$
\begin{aligned}
(b_K^\dagger \circ b_K \circ \iota_{IC}\circ \iota^n_1)([f]) &= 
(b_K^\dagger \circ (\iota^n_2)^{-1} \circ b \circ \iota_{IC}\circ \iota^n_1)([f])  \\
& = 
(b_K^\dagger \circ (\iota^n_2)^{-1} \circ (\alpha_2^n)^{-1} \circ \alpha_1^n \circ \iota^n_1)([f]) = 
(b_K^\dagger \circ (\beta_2^n)^{-1}\circ \beta_1^n )([f]).
\end{aligned}
$$
Here $(\iota^n_2)^{-1}$ denotes the inverse on $\mathrm{Im}(\iota^n_2)$.
Hence it is enough to see the equality
$$
(b_K^\dagger \circ (\beta_2^n)^{-1}\circ \beta_1^n )([f]) = (\iota_{IC}\circ \iota^n_1)([f]) \qquad \text{in $\ILHU$},
$$
whose concrete form is as follows:
\begin{equation}{\label{eq:b_b_id_final_goal}}
\bigoplus_{\alpha \in \Lambda_{**}^{n+1}}\, 
\operatorname{sgn}(\alpha) h_\tau(z)\big|_{\Omega_\alpha}
= (-1)^n\bigoplus_{\alpha \in \Lambda_{**}^{n+1}} \operatorname{sgn}(\alpha)f_{\alpha}|_{T_\alpha}
\qquad \text{in $\ILHU$},
\end{equation}
where $\tau = (\tau_1, \tau_{01}) \in \QDC{0}{n}{\mathcal{V}_K})$ is
given by $[\tau] = ((\beta^n_2)^{-1}\circ \beta_1^n)([f])$. 

\

Let us show {\eqref{eq:b_b_id_final_goal}}.
By applying the same arguments as in Subsection {\ref{subsec:laplace_prodcut}}
to the integral $h_\tau(z)$, we have
$$
h_\tau(z) = \sum_{\alpha \in \Lambda^{n+1}_{**}} \dfrac{(-1)^n \operatorname{sgn}(\alpha)}{(2\pi\sqrt{-1})^n} 
\int_{\Phi_\alpha(L(z))} \dfrac{f_\alpha(z)e^{(z-w)a}}{w-z} dw
\qquad (z \in \Omega \cap E),
$$
where,
for any $\alpha = ((0,\epsilon_0),\cdots,(n-1,\epsilon_{n-1}),n) \in \Lambda_{**}^{n+1}$,
the mapping $\Phi_\alpha: \mathbb{C}^n \to \mathbb{C}^n$ is defined by
$$
\Phi_\alpha(x_1+\sqrt{-1}y_1, \cdots, x_n + \sqrt{-1}y_n)
=(x_1 + \epsilon_0 \sqrt{-1}y_1,
x_2 + \epsilon_1 \sqrt{-1}y_2,
\cdots,
x_n + \epsilon_{n-1} \sqrt{-1}y_n)
$$
and 
$$
L(z) = \ell_{+}(z_1) \times \ell_{+}(z_2) \times \dots \times \ell_{+}(z_n).
$$
Here, for $z_0 = x_0+\sqrt{-1}y_0 \in \mathbb{C}$ with $y_0 \ne 0$, the path $\ell_{+}(z_0) 
\subset \mathbb{C}$ is defined as follows:
Let $\gamma \subset \rcCone$ be a domain with a good boundary such that
it contains the real half line 
$\overline{\{z = x + \sqrt{-1}y \in \mathbb{C}; x \ge \min\{0, 2x_0\},\,\, y=0\}} 
\subset \rcCone$ and 
two points $x_0  \pm \sqrt{-1}y_0 \in \mathbb{C}$ are outside $\overline{\gamma}$.
Then we set 
$$
\ell_+(z_0) = \partial \gamma \cap \{z \in \mathbb{C};\,\textrm{Im}\, z \ge 0\}.
$$
Furthermore, the orientation of $\ell_{+}(z)$ is the same as that of the real axis.

In the same way, we define $\ell_-(z_0) \subset \mathbb{C}$ by
taking the domain $\gamma$ as in the case of $\ell_+(z_0)$. However, in this case, 
we take $\gamma$ so that the two points $x_0 \pm \sqrt{-1}y_0$ are also contained in $\gamma$.
For any $\beta = ((0,\epsilon_0),(1,\epsilon_1),\cdots,(n-1,\epsilon_{n-1}),n) 
\in \Lambda_{**}^{n+1}$, we set
$$
L_\beta(z) := \ell_{\epsilon_0}(z_1)\times \ell_{\epsilon_1}(z_2) \times \cdots
\times \ell_{\epsilon_{n-1}}(z_n)
$$
and
$$
g_{\alpha,\beta}(z) =
\dfrac{1}{(2\pi\sqrt{-1})^n} 
\displaystyle\int_{\Phi_\alpha(L_\beta(z))} \dfrac{f_\alpha(z)e^{(z-w)a}}{w-z} dw.
$$
It follows from the Cauchy integral formula that
\begin{equation}{\label{eq:cauchy-1-1}}
\sum_{\alpha \in \Lambda_{**}^{n+1}} \operatorname{sgn}(\alpha) 
g_{\beta,\alpha}(z)
= \operatorname{sgn}(\beta) f_\beta(z)  \quad (z \in T_\beta,\,\beta \in \Lambda_{**}^{n+1}).
\end{equation}
For $\alpha=((0,\epsilon_0),(1,\epsilon_1),\cdots,(n-1,\epsilon_{n-1}),n)$
and
$\beta=((0,\eta_0),(1,\eta_1),\cdots,(n-1,\eta_{n-1}),n)$ in $\Lambda_{**}^{n+1}$,
we define
$$
\alpha\cdot\beta = ((0,\epsilon_0\eta_0),(1,\epsilon_1\eta_1),
\cdots,(n-1,\epsilon_{n-1}\eta_{n-1}),n) \in \Lambda_{**}^{n+1}.
$$
Remember that $+^n$ denotes $((0,+),(1,+),\cdots,(n-1,+),n)$.
If $\beta\in \Lambda_{**}^{n+1}$ is different from $+^n$, then
$g_{\alpha,\beta}|_{T_\alpha}$ and $g_{\alpha,+^n}|_{T_{\alpha\cdot\beta}}$
can analytically extend to some common infinitesimal wedge in $\rcE$ and
they coincide there. Hence we have, for any $\alpha,\beta \in \Lambda_{**}^{n+1}$,
$$
g_{\alpha,\beta}|_{T_\alpha} = g_{\alpha,+^n}|_{T_{\alpha \cdot \beta}} \quad \text{in $\ILHU$},
$$
from which we have obtained in $\ILHU$
$$
\begin{aligned}
(-1)^n\bigoplus_{\alpha \in \Lambda_{**}^{n+1}}\, 
\operatorname{sgn}(\alpha) h_\tau(z)\big|_{\Omega_\alpha}
&=
\bigoplus_{\alpha \in \Lambda_{**}^{n+1}} 
\sum_{\beta \in \Lambda_{**}^{n+1}}\textrm{sgn}(\alpha)\textrm{sgn}(\beta)g_{\beta,+^n}|_{T_{\alpha}} \\
&=
\bigoplus_{\beta \in \Lambda_{**}^{n+1}}
\sum_{\alpha \in \Lambda_{**}^{n+1}} 
\textrm{sgn}(\alpha \cdot \beta)
g_{\beta,\,\alpha\cdot \beta}|_{T_{\beta}} \\
&=
\bigoplus_{\beta \in \Lambda_{**}^{n+1}}\textrm{sgn}(\beta)
f_\beta|_{T_{\beta}}.
\end{aligned}
$$
This completes the proof.
\end{proof}

\section{Laplace inverse transformation $\IL$}

Let $S$ be a connected open subset in $M^*_\infty$ and $a \in M$.
Note that a connected subset is, in particular, non-empty.
Recall the definition of the map $\varpi_{M^*_\infty}$
given in \eqref{eq:inf_times_star}, for which we have
$$
\varpi_{M^*_\infty}^{-1}(S)=
\{\xi + \sqrt{-1}\eta \in E^*;\, \xi \in S,\, \eta \in M^*\}/\mathbb{R}_+
\,\,\subset E^*_\infty \setminus \sqrt{-1}M^*_\infty.
$$
Here we identify a point in $M^*_\infty$ with a unit vector in $M^*$.


Let $h:M^*_\infty \to \{\pm\infty\} \cup \mathbb{R}$ be an upper semi-continuous function 
such that $h(\xi)$ is continuous on $S$ and $h(\xi)> -\infty$ there.
Now we extend $h$ to the one on $E^*_\infty$ in the following canonical way:
Define $\hat{h}(\zeta): E^*_\infty \to \{\pm \infty\} \cup \mathbb{R}$ by,
for $\zeta = \xi + \sqrt{-1}\eta \in E^*_\infty$ $((\xi,\eta) \in S^{2n-1})$,
$$
\hat{h}(\zeta)
=
\begin{cases}
0 \qquad & (\zeta \in \sqrt{-1}M^*_\infty), \\
\\
|\xi|h(\varpi_{M^*_\infty}(\zeta)) \quad & (\zeta \in E^*_\infty \setminus \sqrt{-1}M^*_\infty),
\end{cases}
$$
where we set $\pm \infty \times c = \pm \infty$ for $c > 0$.
Note that $\hat{h}$ is also upper semi-continuous on $E^*_\infty$ and
continuous on $\varpi_{M^*_\infty}^{-1}(S) \cup \sqrt{-1}M^*_\infty$.

Let $f \in \hhinfo{\hat{h}}(\varpi_{M^*_\infty}^{-1}(S))$.
It follows from the definition of $\hhinfo{\hat{h}}$ that we can find
continuous functions $\psi: S \times [0,\infty) \to \mathbb{R}_{\ge 0}$ and 
$\varphi:[0,\infty) \to \mathbb{R}_{\ge 0}$ satisfying the following conditions:
\begin{enumerate}
\item For any compact subset $L \subset S$, 
		the function $\underset{\xi \in L}{\sup}\, \psi(\xi,\,\lambda)$ 
is an infra-linear function of the variable $\lambda$ and
$f$ is holomorphic on an open subset $W_\psi \cap E^*$, where 
\begin{equation}{\label{eq:domain_W}}
	W_\psi := \widehat{\,\,\,\,}\left\{\zeta = \lambda\xi + \sqrt{-1}\eta \in E^*;\,
\eta \in M^*,\,\, \xi \in S, \,\,
\lambda > \psi(\xi,\, |\eta|)\right\}.
\end{equation}
Note that we identify a point in $M^*_\infty$ with a unit vector in $M^*$ here.
\item $\varphi(t)$ is a continuous infra-linear function on $[0,\infty)$ such that
\begin{equation}{\label{eq:growth_f}}
	|f(\zeta)| \,\le\, e^{- |\zeta|\hat{h}(\pi_{E^*_\infty}(\zeta)) + \varphi(|\zeta|)} 
	= e^{- |\xi|h(\pi_{M^*_\infty}(\xi)) + \varphi(|\zeta|)} 
	\qquad (\zeta = \xi + \sqrt{-1}\eta \in W_\psi \cap E^*),
\end{equation}
where $\pi_{E^*_\infty}: E^*\setminus\{0\} \to (E^*\setminus \{0\})/\mathbb{R}_+ = E^*_\infty$ 
(resp. $\pi_{M^*_\infty}: M^*\setminus\{0\} \to M^*_\infty$) 
is the canonical projection and we also set $e^{+\infty} = +\infty$, $e^{-\infty} = 0$.
\end{enumerate}
We also define an $n$-dimensional real chain in $E^*$ by
\begin{equation}
\gamma^* := \left\{\zeta = \xi + \sqrt{-1}\eta \in E^*;\,
\eta \in M^* \setminus \{0\},\,  \xi = \psi_{\xi_0}(|\eta|)\, \xi_0 \right\}.
\end{equation}
Here $\xi_0 \in S$ and
$\psi_{\xi_0}(\lambda)$ is a smooth function on $[0,\infty)$ which is monotonically increasing and has bounded derivatives on $[0,\infty)$. Further it is infra-linear with
$\psi_{\xi_0}(\lambda) > \psi(\xi_0, \lambda)$ ($\lambda \in [0,\infty)$) and
$\psi_{\xi_0}(\lambda)/(\psi(\xi_0,\lambda)+1) \to \infty$ ($\lambda \to \infty$).
Note that the orientation of $\gamma^*$ is chosen to be the same as that of $\sqrt{-1}M^*$.
\begin{es}
The following situation is the most important one considered in the paper:
Let $K$ be a regular closed subset in $\rcM$ such that 
$\HHPC{K} \cap M^*_\infty$ is connected (in particular, non-empty).
Then we set $S=\HHPC{K} \cap M^*_\infty$ and we also set
$$
h(\xi) = h_K(\xi) = \inf_{x \in K \cap M} \langle x, \xi \rangle \qquad 
(\xi \in M^*_\infty)
$$
if $K \cap M \ne \emptyset$ and set $h(\xi) = +\infty$ if $K \cap M = \emptyset$.
In this case, we have
$$
\HHPC{K} = 
\begin{cases}
	\varpi_{M^*_\infty}^{-1}(S) \qquad &(K \cap M_\infty \ne \emptyset), \\ 
\\
E^*_\infty = \varpi_{M^*_\infty}^{-1}(S) \cup \sqrt{-1}M^*_\infty \qquad 
&( K \cap M_\infty = \emptyset), \\ 
\end{cases}
$$
and
$$
\hat{h}(\zeta) = 
\inf_{x \in K \cap M} 
\operatorname{Re}\,
\langle x, \zeta \rangle \qquad (\zeta \in E^*_\infty)
$$
if $K \cap M \ne \emptyset$ and $\hat{h}(\zeta) = +\infty$ otherwise.
Furthermore, $\hat{h}(\zeta)$ is upper semi-continuous on $E^*_\infty$ 
and continuous on $\HHPC{K} \cup \sqrt{-1}M^*_\infty$.
\end{es}

\

Now we consider the de-Rham theorem with a parameter 
in Section {\ref{sec:dolbeault-complex}}, for which 
we take $T=S^{n-1} = \{\eta \in M^*;\,|\eta|=1\}$ and $Y = S^{n-1} \times \rcE$.  
Define coverings
$$
\mathcal{W} = \{W_0 = Y \setminus p_{\rcE}^{-1}(\rcM),\, W_1 = Y\},\qquad
\mathcal{W}' = \{W_0\}
$$
with $W_{01} = W_0 \cap W_1$.  Recall the isomorphisms
given in Proposition \ref{prop:const-with-param}
$$
\Gamma(T;\, \mathscr{L}^\infty_{loc,T})
=
\Gamma(Y;\, \tilde{p}_T^{-1} \mathscr{L}^\infty_{loc,T})
\overset{\sim}{\longrightarrow}
{\textrm H}^n_{p_{\rcE}^{-1}(\rcM)}(Y;\,p_T^{-1} \mathscr{L}^\infty_{loc,T})
=
{\textrm H}^n(C(\mathcal{W},\mathcal{W}')(\LQ_Y^{(\bullet)})),
$$
and set 
$$
\Omega :=
\widehat{\,\,\,\,}
\left\{(\theta,z) \in S^{n-1} \times E;\, \langle \theta,\,\operatorname{Im} z \rangle > 0
\right\}
\subset Y.
$$
Let $j: \Omega \hookrightarrow Y$ be the canonical open inclusion.
Then we can take a specific $\omega = (\omega_1, \omega_{01}) \in
C^n(\mathcal{W},\mathcal{W}')(\LQ_Y^{(\bullet)})$ satisfying the following
conditions:
\begin{enumerate}
\item[D1.]$\,$ $D_{\rcE} \omega = 0$ and $[\omega]$ is the image of a constant function 
$1 \in \Gamma(T;\, \mathscr{L}^\infty_{loc,T})$ through the above isomorphisms.
\item[D2.]$\,$ We have
$
\operatorname{supp}_{W_1}(\omega_1) \subset \Omega
$
and
$
\operatorname{supp}_{W_{01}}(\omega_{01}) \subset \Omega
$.
\end{enumerate}
The existence of the above $\omega$ comes from the following lemma:
\begin{lem}{\label{lem:fundamental_support_LQ}}
The canonical morphisms 
$$
\begin{aligned}
&C(\mathcal{W},\mathcal{W}')(j_!j^{-1}\LQ_Y^{(\bullet)})
\longrightarrow
C(\mathcal{W},\mathcal{W}')(\LQ_Y^{(\bullet)}), \\
&C(\mathcal{W},\mathcal{W}')
(j_!j^{-1}\EQ_Y^{(\bullet)})
\longrightarrow
C(\mathcal{W},\mathcal{W}')(\EQ_Y^{(\bullet)})
\end{aligned}
$$
are quasi-isomorphic.
\end{lem}
\begin{proof}
Let $\mathscr{F}$ be a $\mathscr{L}^\infty_{loc,T}$ or $\mathscr{E}_T$, and
let $i: Y \setminus \Omega \to Y$ denote the closed embedding.
Then the above isomorphism is equivalent to the following isomorphism:
$$
{\mathbf R}\Gamma_{p_{\rcE}^{-1}(\rcM)}(Y;\,j_!j^{-1}p_T^{-1} \mathscr{F})
\longrightarrow
{\mathbf R}\Gamma_{p_{\rcE}^{-1}(\rcM)}(Y;\,p_T^{-1} \mathscr{F}),
$$
which comes from the fact
\begin{equation}{\label{eq:fact_xxx_1111}}
{\mathbf R}\Gamma_{p_{\rcE}^{-1}(\rcM)}(Y;\,i_*i^{-1}p_T^{-1} \mathscr{F})
\simeq 0.
\end{equation}
The fact itself can be shown by the following argument: Let us consider
the distinguished triangle
$$
\begin{aligned}
{\mathbf R}\Gamma_{p_{\rcE}^{-1}(\rcM)}(Y;\,i_*i^{-1}p_T^{-1} \mathscr{F})
\to
&{\mathbf R}\Gamma(Y \setminus \Omega;\,p_T^{-1} \mathscr{F})
\overset{\beta}{\to} \\
&\qquad{\mathbf R}\Gamma((Y \setminus \Omega)\setminus p_{\rcE}^{-1}(\rcM);\,p_T^{-1} \mathscr{F})
\overset{+1}{\to}
\end{aligned}
$$
Under the commutative diagram below,
$$
\begin{diagram}
	\node{(Y \setminus \Omega) \setminus p_{\rcE}^{-1}(\rcM)}
	\arrow{e,t}{\iota}\arrow{se,l}{p_T} 
	\node{Y \setminus \Omega} 
	\arrow{s,l}{p_T} 
	\\
	\node[2]{T}
\end{diagram},
$$
the morphism $\iota$ gives a homotopical equivalence over $T$, and hence,
it follows from Corollary 2.7.7 (i) [KS] that the morphism $\beta$ is isomorphic.
This implies {\eqref{eq:fact_xxx_1111}}. The proof has been completed.
\end{proof}

Note that we will give a concrete construction of such an $\omega$ later. 
Recall the standard coverings
$$
\mathcal{V}_\rcM = \{V_0 = \rcE \setminus \rcM,\, V_1 = \rcE\}, \qquad
\mathcal{V}_\rcM' = \{V_0\},
$$
and the morphism $\rho=\{\rho_k\}: \LQ_Y^{(\bullet)} \to \LQ_Y^{(0,\bullet)}$ of complexes
which is the projection to the space of anti-holomorphic forms, that is, each
$
\rho_k: \LQ_Y^{(k)} \to \LQ_Y^{(0,k)}
$
is defined by
$$
\sum_{|I|=i,\,|J|=j,\, i+j=k} f_{I,J}(\theta,z)dz^I \wedge d\bar{z}^J
\qquad \mapsto \qquad
\sum_{|J|=k} f_{\emptyset, J}(\theta,z) d\bar{z}^J.
$$
Note that the following diagram commutes
$$
\begin{diagram}
\node{{\mathbf R}\Gamma_{p^{-1}_{\rcE}(\rcM)}(Y;\,p_T^{-1}\mathscr{L}^\infty_{loc,T})}
\arrow{e}\arrow{s}
\node{{\mathbf R}\Gamma_{p^{-1}_{\rcE}(\rcM)}(Y;\, \LO_Y)} 
\arrow{s}
\\
\node{C(\mathcal{W},\mathcal{W}')(\LQ_Y^{(\bullet)})}
\arrow{e,t}{\rho}
\node{C(\mathcal{W},\mathcal{W}')(\LQ_Y^{(0,\bullet)})}
\end{diagram},
$$
where vertical arrows are quasi-isomorphic.

Let us take an $\omega = (\omega_1, \omega_{01}) \in
C^n(\mathcal{W},\mathcal{W}')(\LQ_Y^{(\bullet)})$
which satisfies the conditions
D1.~and D2.
\begin{df}{\label{def:IL}}
The Laplace inverse transform $\IL$ is given by
$$
\IL(f) = \left([\IL_\omega(fd\zeta)] \otimes a_{\rcM/\rcE}\right) \otimes \nu_{\rcM}
$$
with
$$
\begin{aligned}
\IL_\omega(fd\zeta) 
&:=\ilc
\int_{\gamma^*} \rho(\omega)(\dfrac{\eta}{|\eta|},z)\, e^{\zeta z} f(\zeta) d\zeta \\
&=\ilc
\left(
	\int_{\gamma^*} \rho_n(\omega_1)(\dfrac{\eta}{|\eta|}, z)\, e^{\zeta z}f(\zeta)d\zeta,\,\,
	\int_{\gamma^*} \rho_{n-1}(\omega_{01})(\dfrac{\eta}{|\eta|}, z)\, e^{\zeta z}f(\zeta)d\zeta
\right). 
\end{aligned}
$$
Here $\zeta = \xi + \sqrt{-1}\eta$ are the dual variables of $z=x+\sqrt{-1}y$,
$a_{\rcM} \in or_\rcM(\rcM)$, $a_{\rcM/\rcE} \in or_{\rcM/\rcE}(\rcM)$ so that
$a_{\rcM/\rcE} \otimes a_{\rcM}$ has the same orientation as that of $E$
through the isomorphism $or_{\rcM/\rcE} \otimes or_{\rcM} \simeq or_{\rcE}|_{\rcM}$,
and the volume $\nu_{\rcM}$ is defined by $dz \otimes a_\rcM$ 
with $dz=dz_1\wedge \cdots \wedge dz_n$ and 
$d\zeta = d\zeta_1 \wedge \cdots \wedge d\zeta_n$.
\end{df}
\begin{lem}
We have
\begin{enumerate}
\item The integration $\IL_\omega(f d\zeta)$ converges and it belongs to
$\QDC{0}{n}{\mathcal{V}_\rcM}$. Furthermore, $\bvth(\IL_\omega(f d\zeta)) = 0$ holds.
\item $\IL_\omega(f d\zeta)$ does not depend on the choices of $\omega$.  
\end{enumerate}
\end{lem}
\begin{proof}
Since the support of $\omega_{01}$ (resp. $\omega_1$) 
is a closed subset 
in $W_{01}$ (resp. $W_1$)
and $Y_\infty = S^{n-1} \times E_\infty$ is compact, 
we have the following facts:
\begin{enumerate}
\item There exist an open neighborhood 
$O \subset \rcE$ of $\rcM$ and $\delta > 0$ such that
$$
\begin{aligned}
\operatorname{supp}_{W_1}(\omega_1) 
\,\,&\subset\,\,
(S^{n-1} \times (\rcE \setminus O))
\bigcap\, 
\widehat{\,\,\,\,}
\left\{(\eta,z) \in S^{n-1} \times E;\, \langle \eta,\,
\operatorname{Im} z \rangle > \delta|\operatorname{Im} z|\right\}.
\end{aligned}
$$
\item For any open neighborhood $O \subset \rcE$ of  $\rcM$, there exists
$\delta > 0$ such that
$$
\begin{aligned}
&\operatorname{supp}_{W_{01}}(\omega_{01}) \cap (S^{n-1} \times (\rcE \setminus O))
\,\,
\subset\,\,
\widehat{\,\,\,\,}
\left\{(\eta,z) \in S^{n-1} \times E;\, 
\langle \eta,\,\operatorname{Im} z \rangle > 
\delta|\operatorname{Im} z |\right\}.
\end{aligned}
$$
\end{enumerate}
The fact $\IL_\omega(fd\zeta) \in \QDC{0}{n}{\mathcal{V}_\rcM}$ and
the claim 1.~easily follows from these facts.
As a matter of fact, for example, the integral
$$
\int_{\gamma^*} \rho_{n-1}(\omega_{01})(\dfrac{\eta}{|\eta|}, z)\, e^{\zeta z}f(\zeta)d\zeta
$$
is shown to be a form on $V_0$ with the desired growth condition in the following way:
Let $O \subset \rcE$ be an open neighborhood of $\rcM$. Then,
by the above fact 2.~there exists $\delta > 0$ such that we have on $E \setminus O$
$$
\int_{\gamma^*} \rho_{n-1}(\omega_{01})(\dfrac{\eta}{|\eta|}, z)\, e^{\zeta z}f(\zeta)d\zeta
=
\int_{\widehat{\gamma}^*(z)}
\rho_{n-1}(\omega_{01})(\dfrac{\eta}{|\eta|}, z)\, e^{\zeta z}f(\zeta)d\zeta,
$$
where
$$
\widehat{\gamma}^*(z)  =
\gamma^* \cap 
\{\xi + \sqrt{-1}\eta \in \mathbb{C}^n_\zeta;\,
\langle \eta, \mathrm{Im}(z)/|\mathrm{Im}(z)| \rangle > \delta|\eta|\}.
$$
For any $z = x + \sqrt{-1}y \in E \setminus O$,
where we may assume $|y| > \delta'\max\,\{|x|,1\}$ holds for some $\delta' > 0$,
and for any $\zeta=\xi +\sqrt{-1}\eta \in \widehat{\gamma}^*(z)$, we have
\begin{equation}{\label{eq:est_lap_exp_kernel}}
\begin{aligned}
\log |e^{\zeta z} f(\zeta)| &\le
\psi_{\xi_0}(|\eta|)\langle\xi_0, x\rangle
- \langle \eta, y\rangle - \psi_{\xi_0}(|\eta|)h(\xi_0) + \varphi(|\zeta|) \\
& \le 
\psi_{\xi_0}(|\eta|)\big(\langle\xi_0, x\rangle - h(\xi_0)\big)
- \delta |\eta||y|
+ \varphi(|\zeta|).
\end{aligned}
\end{equation}
Set $\delta'' = \max\,\{\delta'^{-1},\, |h(\xi_0)|\}$. Then
the last term is estimated by:
$$
\begin{aligned}
\log |e^{\zeta z} f(\zeta)| 
	& \le \psi_{\xi_0}(|\eta|)\big(\delta'^{-1}|y| +  |h(\xi_0)|\big) 
- \delta|\eta||y| + \varphi(|\zeta|) \\
& \le \delta''\psi_{\xi_0}(|\eta|)(|y| + 1) - \delta |\eta||y| + \varphi(|\zeta|).
\end{aligned}
$$
Since $\psi_{\xi_0}(t)$ and $\varphi(t)$ are of infra-linear, for any $\epsilon > 0$,
there exists $M >0$ such that
$$
\psi_{\xi_0}(|\eta|) \le M + \epsilon|\eta|,\qquad
\varphi(|\zeta|) \le M + \epsilon |\eta|.
$$
Hence we have
$$
\log |e^{\zeta z} f(\zeta)| \le
\delta''M|y| +((\delta'' \epsilon - \delta)|y| + (\delta''+1)\epsilon )|\eta| + 
(\delta'' + 1)M,
$$
which implies that, for a sufficiently small $\epsilon > 0$, the integral converges on $E \setminus O$ with the desired growth condition. Since $O$ is an arbitrary open neighborhood of $\rcM$, we can get the conclusion.

\

Now let us show the claim 2. 
Let $\omega'$ be another choice of $\omega$. Then, by the 
Lemma \ref{lem:fundamental_support_LQ}, we can find 
$\omega^{n-1} \in C^{n-1}(\mathcal{W},\mathcal{W}')(j_!j^{-1}\LQ_Y^{(\bullet)})$
such that
$$
\rho(\omega) - \rho(\omega') = \rho(D_{\rcE} \omega^{n-1}) = \bvth_{\rcE} \rho(\omega^{n-1}).
$$
Since $\omega^{n-1}$ satisfies the same support conditions as those for $\omega$,
the integration $\IL_{\omega^{n-1}}(f d\zeta)$ which is defined by
replacing $\omega$ with $\omega^{n-1}$ in the definition of $\IL_\omega(f d\zeta)$ also converges.
Hence we have
$$
\IL_\omega(f d\zeta) - \IL_{\omega'}(f d\zeta) = \bvth \IL_{\omega^{n-1}}(f d\zeta).
$$
This completes the proof.
\end{proof}

\begin{lem}{\label{lem:IL-path-independent}}
The $\IL(f)$ is independent of the choice of $\xi_0$ and $\psi_{\xi_0}$ 
which appear in the definition of $\gamma^*$.  As a consequence, we have
\begin{equation}{\label{eq:lem-estimate-support}}
\operatorname{supp}(\IL(f)) \subset 
\bigcap_{\xi_0 \in S} \overline{\{x \in M;\, \langle x, \xi_0 \rangle \ge h(\xi_0)\}}.
\end{equation}
\end{lem}
\begin{proof}
We first assume $n > 1$.  Let us consider the commutative diagram below:
$$
\begin{array}{llll}
\Gamma(T;\, \mathscr{L}^\infty_{loc,T})
&=
\Gamma(Y;\, \tilde{p}_T^{-1} \mathscr{L}^\infty_{loc,T})
&\overset{\sim}{\longrightarrow}
{\textrm H}^n_{p_{\rcE}^{-1}(\rcM)}(Y;\,p_T^{-1} \mathscr{L}^\infty_{loc,T})
&=
{\textrm H}^n(C(\mathcal{W},\mathcal{W}')(\LQ_Y^{(\bullet)})
\\
\qquad\uparrow
&
&
\qquad\qquad\uparrow
&
\qquad\qquad\uparrow\\
\Gamma(T;\, \mathscr{E}_{T})
&=
\Gamma(Y;\, \tilde{p}_T^{-1} \mathscr{E}_{T})
&\overset{\sim}{\longrightarrow}
{\textrm H}^n_{p_{\rcE}^{-1}(\rcM)}(Y;\,p_T^{-1} \mathscr{E}_{T})
&=
{\textrm H}^n(C(\mathcal{W},\mathcal{W}')(\EQ_Y^{(\bullet)})),
\end{array}
$$
where all the horizontal arrows are isomorphisms and
every vertical arrow is injective.
Furthermore, the bottom horizontal arrows are morphisms of $\mathscr{D}_T$-modules.
Hence we can take 
$\omega = (\omega_1, \omega_{01}) \in  
C^n(\mathcal{W},\mathcal{W}')(\EQ_Y^{(\bullet)})$
that is a representative of 
the image of $1 \in \Gamma(T;\, \mathscr{E}_{T})$ by the bottom horizontal arrows.
It follows from Lemma {\ref{lem:fundamental_support_LQ}} that 
the $\omega$ is assumed to satisfy the following conditions:
\begin{enumerate}
\item $\operatorname{supp}_{W_1}(\omega_1) \subset \Omega$ and 
$\operatorname{supp}_{W_{01}}(\omega_{01}) \subset \Omega$.
\item For any vector fields $\nu$ on $T$, we have
	$\nu [\omega] = 0$ since $[\omega]$ is the image of $1$ and the bottom horizontal morphisms in the commutative diagram are $\mathscr{D}_T$-linear.
\end{enumerate}

Let $(\theta_1,\cdots,\theta_n)$ be a homogeneous coordinate system of $S^{n-1}$,
and let $\pi: M^* \setminus \{0\} \to S^{n-1}$ a smooth map defined by
$$
(\eta_1,\cdots,\eta_n) \mapsto 
(\dfrac{\eta_1}{|\eta|},\, \dfrac{\eta_2}{|\eta|},\, \cdots, 
\dfrac{\eta_n}{|\eta|}),
$$
which induces the morphism of vector bundles
$$
\pi': {T}(M^* \setminus \{0\}) \to (M^* \setminus \{0\}) 
	\underset{S^{n-1}}{\times} {T}S^{n-1}.
$$
By restricting the base space of the above bundle map 
to $S^{n-1} \subset M^*\setminus \{0\}$,
we get the morphism of vector bundles
\begin{equation}
\varphi: {T}M^*|_{S^{n-1}} \to {T}S^{n-1},
\end{equation}
by which we define the vector fields $\nu_k$ on $T=S^{n-1}$ as
\begin{equation}
\nu_k = \varphi\big(\dfrac{\partial}{\partial \eta_k}\bigg|_{S^{n-1}}\big)
\qquad (k=1,2,\cdots, n).
\end{equation}
Then, since $\nu_k[\omega] = 0$ holds, 
it follows from Lemma \ref{lem:fundamental_support_LQ} that
there exists $\tilde{\omega}_k=(\tilde{\omega}_{k,1},\,\tilde{\omega}_{k,01})
\in C^{n-1}(\mathcal{W},\mathcal{W}')(\EQ_Y^{(\bullet)})$
with
$$
\operatorname{supp}_{W_1}(\tilde{\omega}_{k,1}) \subset \Omega
\quad\text{and}\quad
\operatorname{supp}_{W_{01}}(\tilde{\omega}_{k,01}) \subset \Omega,
$$
such that
$$
\nu_k \omega = D_{\rcE} \tilde{\omega}_k,
$$
from which we have ($\zeta = \xi + \sqrt{-1}\eta$)
\begin{equation}{\label{eq:diff-const-zero}}
\begin{aligned}
\dfrac{\partial}{\partial \overline{\zeta}_k}
\big(\rho(\omega)(\eta/|\eta|,z)\big)
&=\dfrac{\sqrt{-1}}{|\eta|}
\rho(\nu_k \omega)(\eta/|\eta|,z)
=\dfrac{\sqrt{-1}}{|\eta|}
\rho(D_{\rcE}\tilde{\omega}_k)(\eta/|\eta|,z) \\
&=\bvth\big(\dfrac{\sqrt{-1}}{|\eta|}
\rho(\tilde{\omega}_k)(\eta/|\eta|,z)\big).
\end{aligned}
\end{equation}

Let us consider ($\xi_0$, $\psi_{\xi_0})$ and
($\xi_1$, $\psi_{\xi_1})$, which generate the $n$-dimensional chains $\gamma^*_0$
and $\gamma^*_1$, respectively. Then, by taking a continuous path $s(\lambda)$ ($\lambda \in [0,1]$) in $S$ with $s(0)=\xi_0$ and $s(1)=\xi_1$, we define an $(n+1)$-dimensional chain $\tilde{\gamma}^*$
by
$$
\tilde{\gamma}^* := \{\xi + \sqrt{-1}\eta \in E^*;\,
\xi = ((1-\lambda)\psi_{\xi_0}(|\eta|) + \lambda\psi_{\xi_1}(|\eta|)) s(\lambda),\,\,
0 \le \lambda \le 1,\,\,\eta \in M^* \setminus \{0\}\}.
$$
Here we may assume $\tilde{\gamma}^* \subset W_\psi$. In fact, we first
consider the pair of chains generated by $(\xi_0, \psi_{\xi_0})$ and $(\xi_0, g)$
where $g$ is taken to be a sufficiently large infra-linear function. Then
consider the pair of chains generated by $(\xi_0, g)$ and $(\xi_1,g)$ and finally
that by $(\xi_1, \psi_{\xi_1})$ and $(\xi_1,g)$.

By noticing that the function $\dfrac{1}{|\eta|}$ on $M^*\setminus \{0\}$ is integrable near the origin if $n > 1$ and
that each $\tilde{\omega}_k$ satisfies the same support condition
as that for $\omega$, it follows from the Stokes formula that we obtain
$$
\begin{aligned}
&\int_{\tilde{\gamma}^*} f(\zeta)\, 
\overline{\partial}_\zeta\big(\rho(\omega)(\eta/|\eta|,z)\big)\, e^{\zeta z}\, d\zeta  \\
&\qquad\qquad=
\int_{\gamma^*_1} f(\zeta)\, \rho(\omega)(\eta/|\eta|,z)\, e^{\zeta z}\, d\zeta 
-
\int_{\gamma^*_0} f(\zeta)\, \rho(\omega)(\eta/|\eta|,z)\, e^{\zeta z}\, d\zeta. 
\end{aligned}
$$
It follows from {\eqref{eq:diff-const-zero}} that we have
$$
\begin{aligned}
\int_{\tilde{\gamma}^*} f(\zeta)\, 
\overline{\partial}_\zeta\big(\rho(\omega)(\eta/|\eta|,z)\big)\, e^{\zeta z}\, d\zeta 
&=
\int_{\tilde{\gamma}^*} f(\zeta)\, e^{\zeta z}\, 
\sum_{k=1}^n\dfrac{\partial}{\partial \overline{\zeta}_k} 
\big(\rho(\omega)(\eta/|\eta|,z)\big)\, 
d\overline{\zeta}_k \wedge d\zeta  \\
&=
\bvth
\int_{\tilde{\gamma}^*} f(\zeta)\, e^{\zeta z}\, 
\dfrac{\sqrt{-1}}{|\eta|}\sum_{k=1}^n 
\rho(\tilde{\omega}_k)(\eta/|\eta|,z) d\overline{\zeta}_k \wedge d\zeta.
\end{aligned}
$$
Hence the Laplace transform of $f$ with the chain $\gamma^*_0$
and the one with the chain $\gamma^*_1$ give the same cohomology class.

\

Let us show {\eqref{eq:lem-estimate-support}} in the lemma.
Fix $\xi_0 \in S$ and take a sufficiently large $\ell > 0$
so that $\psi_{\xi_0}(t) \le t + \ell$ holds for $t \in [0,\infty)$.
Let us consider the $n$-dimensional chain $\gamma^*_\epsilon$ for $1 \ge \epsilon > 0$
$$
\gamma_\epsilon^* := \left\{\xi + \sqrt{-1}\eta \in E^*;\,
\xi = (\epsilon^{-1}|\eta| + \ell) \xi_0, \,\eta \in M^* \setminus \{0\}\right\}
$$
and the $(n+1)$-dimensional chain
$$
\tilde{\gamma}_\epsilon^* = \left\{\xi + \sqrt{-1}\eta \in E^*;\,
\begin{aligned}
&\xi = 
\left((1-\lambda)\psi_{\xi_0}(|\eta|) + \lambda(\epsilon^{-1}|\eta| + \ell)
\right)\, \xi_0
\\
&0 \le \lambda \le 1,\,\,\eta \in M^* \setminus \{0\}
\end{aligned}
\right\}.
$$
Note that $\tilde{\gamma}_\epsilon^* \subset W_\psi$ holds for $1 \ge \epsilon > 0$.
Then, on $\{z \in E;\,\textrm{Re}\,\langle \xi_0, z \rangle < h(\xi_0)$\}, 
by taking the estimate \eqref{eq:est_lap_exp_kernel} into account,
we have
$$
\begin{aligned}
&\int_{\tilde{\gamma}_\epsilon^*} f(\zeta)\, 
\overline{\partial}_\zeta\big(\rho(\omega)(\eta/|\eta|,z)\big)\, e^{\zeta z}\, d\zeta \\
&\qquad\qquad=
\int_{\gamma^*_0} f(\zeta)\, \rho(\omega)(\eta/|\eta|,z)\, e^{\zeta z}\, d\zeta 
-
\int_{\gamma^*_\epsilon} f(\zeta)\, \rho(\omega)(\eta/|\eta|,z)\, e^{\zeta z}\, d\zeta,
\end{aligned}
$$
where all the integrals converge. Hence, by letting $\epsilon \to 0+0$, we get
$$
\int_{\tilde{\gamma}_{0+0}^*} f(\zeta)\, 
\overline{\partial}_\zeta\big(\rho(\omega)(\eta/|\eta|,z)\big)\, e^{\zeta z}\, d\zeta 
=
\int_{\gamma^*_0} f(\zeta)\, 
\rho(\omega)(\eta/|\eta|,z)\, e^{\zeta z}\, d\zeta.
$$
Here the $(n+1)$-dimensional chain $\tilde{\gamma}_{0+0}$ is
$$
\tilde{\gamma}_{0+0}^* := \{\xi + \sqrt{-1}\eta \in E^*;\,
	\xi = \lambda \xi_0,\,\,
\lambda \ge \psi_{\xi_0}(|\eta|),\,\,\eta \in M^* \setminus \{0\}\}
$$
and all the integrals still converge.
This implies that, as the left hand side of the above equation gives the zero cohomology class in
$\widehat{\,\,\,\,}\{z \in E;\,\textrm{Re}\,\langle \xi_0, z \rangle < h(\xi_0)\}$,
and thus, $\operatorname{supp}(\IL(f))$ is contained in
$\overline{\{x \in M\,;\,\langle \xi_0, x \rangle \ge h(\xi_0)\}}$.
Since we can take any vector in $S$
as $\xi_0$, we have concluded the second claim of this lemma when $n > 1$.

\

Now we consider the case $n=1$. In this case, $S^{n-1}$ consists of only two points
$\{+1,-1\}$. Hence it follows from the definition of $\omega$ that 
$\tau = \omega(1,z)$ (resp. $\tau = \omega(-1,z)$) satisfies the conditions
in Lemma {\ref{lem:support_one}} with $\Omega = \Omega^1_+$ 
(resp. $\Omega = \Omega^1_-$), where
$$
\Omega^1_{\pm} = \widehat{\,\,\,\,\,}\{z \in \mathbb{C}; \pm \textrm{Im}\,z > 0\} \subset \rcCone.
$$
Hence we have obtained
$$
\IL_\omega(fd\zeta) 
=b_{\Omega^1_+}\left(\dfrac{1}{2\pi\sqrt{-1}}\int_{\gamma^* \cap \Omega^1_+} e^{\zeta z}f(\zeta)d\zeta\right)
-
b_{\Omega^1_-}\left(\dfrac{1}{2\pi\sqrt{-1}}\int_{\gamma^* \cap \Omega^1_-} e^{\zeta z}f(\zeta)d\zeta\right)
$$
for which we can easily see the claims of the lemma.
This completes the proof.
\end{proof}

As an immediate application of the above lemma, we have the following corollary.
Recall that, for a subset $G \subset M$, we define
$$
G^\circ = \{\zeta \in E^*;\, \textrm{Re}\,\langle \zeta,\, x\rangle \ge 0
\,\,(\forall x \in G)\}.
$$
\begin{cor}
	Let $a \in M$ and $G \ne \emptyset$ be an $\mathbb{R}_+$-conic proper closed convex subset in $M$. Furthermore, we also assume that $G^\circ \cap M^*_\infty$ 
is connected when $n=1$.
Set $K = \overline{a + G} \subset \rcM$ and let $e^{a\zeta}g(\zeta) \in \hinfo(\HHPC{K}) = \hinfo(\widehat{\,\,\,\,}(\mathrm{int}\,G^\circ) \cap E^*_\infty)$.
Then we have
\begin{equation}
\operatorname{supp}(\IL(g)) \subset K.
\end{equation}
\end{cor}
In fact, the corollary follows from the lemma by taking 
$S = \HHPC{K} \cap M^*_\infty$ and $h(\xi) = a\xi$ and by noticing the facts
that $\varpi_{M^*_\infty}^{-1}(S) = \HHPC{K}$ 
(resp. $\varpi_{M^*_\infty}^{-1}(S) \cup \sqrt{-1}M^*_\infty = \HHPC{K} = E^*_\infty$)
holds if $K \cap M_\infty \ne \emptyset$ (resp. $K \cap M_\infty = \emptyset$)
and that $S$ is connected.

\subsection{Concrete construction of $\omega$}

Now we give a method to construct $\omega$ concretely.
Let $O$ be a subset in $S^{n-1} = \{\xi \in M^*;\,|\xi| = 1\}$, and let
$\theta_k: O \to S^{n-1} \subset M^*$ ($k=1,\dots,n$) be continuous 
maps on $O$.  Set, for $\xi \in O$,
$$
\kappa(\xi) := 
\bigcap_{k=1}^n\{x \in M; \langle x,\,\theta_k(\xi) \rangle > 0\}\,\,
\subset M.
$$
We assume that there exists $\delta > 0$ satisfying
\begin{enumerate}
\item[C1.]$\,$ $S^{n-1} \setminus O$ is measure zero.
\item[C2.]$\,$
$
\kappa(\xi) \subset \left\{x \in M;\, \langle x,\,\xi \rangle > \sigma |x|\right\}
$
for any $\xi \in O$.
\item[C3.]$\,$ Let $A(\xi)$ be an $n \times n$-matrix $(\theta_1(\xi),\dots,\theta_n(\xi))$.
Then $\operatorname{det}(A(\xi)) \ge \delta$ for any $\xi \in O$.
\end{enumerate}
Note that the condition C2 is equivalent to the following C2':
\begin{enumerate}
\item[C2'.]$\,$ Set $G(\xi) := \displaystyle\sum_{k=1}^n \mathbb{R}_+ \theta_k(\xi)$. Then we have
$$
\operatorname{dist}(\xi,\, \mathbb{R}^n \setminus G(\xi)) > \delta
\qquad (\xi \in O).
$$
\end{enumerate}
In fact, C2' implies
$$
\left\{\tau \in \mathbb{R}^n;\, \left|\dfrac{\tau}{|\tau|} - \xi\right| \le \dfrac{\delta}{2}\right\}
\subset G(\xi).
$$
Then, by taking the dual of the above sets and by noticing $\kappa(\xi) = \mathrm{int}\,G(\xi)^\circ$,
we can obtain C2.

\

Let $\varphi_0(z)$, $\dots$, $\varphi_{n}(z)$ 
be in $\mathscr{Q}_{\rcE}(\rcE \setminus \rcM)$ 
which are given in Example \ref{es:fundamental_example} with $U=\rcM$, $V = \rcE$,
$$
\eta_k = \begin{cases}
(0,\,\dots,\,0,\,\overset{\text{$(k+1)$-th}}{1},\,0,\,\dots,0) \, \qquad &(k=0,\dots,n-1),\\
\eta_k = - (\eta_0 + \cdots + \eta_{n-1})/|\eta_0 + \cdots + \eta_{n-1}| \qquad & (k=n)
\end{cases}
$$
and $H_k = \Gamma_k = \{y \in M;\,\langle y,\eta_k\rangle > 0\}$ ($k=0,1,\cdots, n$).
Using these $\varphi_k$'s, we define $\omega_{01}$ by
$$
\omega_{01}(\xi, z) := (-1)^n (n-1)!\, 
\chi_{S_{01\cdots n-1}}({}^tA(\xi)z) \,
\bar{\partial}_z (\varphi_0({}^tA(\xi)z)) \wedge \cdots \wedge
\bar{\partial}_z (\varphi_{n-2}({}^tA(\xi)z)),
$$
where $S_{01\cdots n-1}$ is also given in Example \ref{es:fundamental_example}.
Then, by the same reasoning as that of Example 7.14 in \cite{HIS}
and Corollary {\ref{cor:mes-zero}}, we have
\begin{lem}
Thus constructed $\omega = (0,\, \omega_{01})$ satisfies 
the conditions D1.~and D2.~described before Lemma \ref{lem:fundamental_support_LQ}.
\end{lem}

\

We give some examples of such a family $\theta_k$'s.

\begin{es}{\label{es:explicit_const_omega}}
Let $\chi$ be a triangulation of $S^{n-1}$, 
and let $\{\sigma_\lambda\}_{\lambda \in \Lambda}$ be the set
of $(n-1)$-cells of $\chi$. 
For each $\lambda \in \Lambda$, we take linearly independent $n$-vectors 
$\nu_{\lambda,1}$, $\cdots$, $\nu_{\lambda,n} \in M^*$ which satisfies
$$
\overline{\sigma_\lambda} 
\subset \displaystyle\sum_{k=1}^n \mathbb{R}_+ \nu_{\lambda,k},
$$
and $\operatorname{det} A_\lambda > 0$
for the constant matrix $A_\lambda := (\nu_{\lambda,1},\nu_{\lambda,2},\dots,\nu_{\lambda,n})$. 
Note that such a family of constant vectors always exists if
each $\sigma_\lambda$ is sufficiently small. 
Furthermore, we may assume
the frame
$\nu_{\lambda,1},\nu_{\lambda,2},\dots,\nu_{\lambda,n}$ determine the positive orientation
in $M^*$ for each $\lambda$, 
Then, we set $O := \cup_{\lambda \in \Lambda} \sigma_\lambda$ and, for $k=1,\dots,n$,
define $\theta_k(\xi)$ on $O$ by
$$
\theta_k(\xi) = \nu_{\lambda,k} 	\qquad (\xi \in \sigma_\lambda).
$$
Clearly these $O$ and $\theta_k$'s satisfy the conditions C1, C2 and C3.
\end{es}

\

\begin{es}
Assume $M^*$ has an inner product.
Let $p$ be a point in $S^{n-1}$ and set $O := S^{n-1}\setminus \{p\}$.
Then $O$ becomes contractible, and hence, there exists a continuous orthogonal frame
$\tilde{\theta}_1(\xi)$, $\dots$, $\tilde{\theta}_n(\xi) \in M^*$ on $O$.
Here we may assume $\tilde{\theta}_1(\xi) = \xi$.
Set, for some $\delta > 0$,
$$
\begin{aligned}
\theta_1(\xi) &:= \tilde{\theta}_2(\xi) + \delta \tilde{\theta}_1(\xi), \\
\theta_2(\xi) &:= \tilde{\theta}_3(\xi) + \delta \tilde{\theta}_1(\xi), \\
&\vdots \\
\theta_{n-1}(\xi) &:= \tilde{\theta}_n(\xi) + \delta \tilde{\theta}_1(\xi), \\
\theta_{n}(\xi) &:= -(\tilde{\theta}_2 + \cdots + \tilde{\theta}_n(\xi)) + \delta \tilde{\theta}_1(\xi).
\end{aligned}
$$
Then these $O$ and $\theta_k$'s satisfy the conditions C1, C2 and C3.
\end{es}

\

Let us compute $\IL$ when $\omega$ comes from Example {\ref{es:explicit_const_omega}}.
In this case, on each $\sigma_\lambda$, $\omega_{01}(\xi,z)$ does not depend on the variables $\xi$. 
Hence we obtain
\begin{equation}
\IL(f) := \left[\ilc
\left(0,\,\, \sum_{\lambda \in \Lambda} \tau_{01,\lambda}\int_{\gamma^*_\lambda}  f(\zeta) e^{\zeta z} d\zeta\right)\right] \otimes a_{\rcM/\rcE} \otimes \nu_{\rcM}.
\end{equation}
Here
$$
\gamma^*_\lambda := \left\{\zeta = \xi + \sqrt{-1}\eta \in E^*;\,
\eta \in \mathbb{R}_+\sigma_\lambda,\,  \xi = \psi_{\xi_0}(|\eta|)\, \xi_0 \right\},
$$
and
$$
\tau_{01,\lambda}(z) 
:= (-1)^n (n-1)!\, \chi_{S_{01\cdots n-1}}({}^tA_\lambda z) \,
\bar{\partial} (\varphi_1({}^tA_\lambda z)) \wedge \cdots \wedge
\bar{\partial} (\varphi_{n-1}({}^tA_\lambda z)),
$$
where the constant matrix $A_\lambda$ is given by
$(\nu_{\lambda,1},\dots,\nu_{\lambda,n})$ and
the orientation of the chain $\gamma_\lambda^*$ is induced from the one of $\sqrt{-1}M^*$
through the canonical projection $E^* = M^* \times \sqrt{-1}M^* \to \sqrt{-1}M^*$.
Then, as we see in Example \ref{es:fundamental_example},
$\tau_\lambda := (0,\tau_{01,\lambda})$ satisfies the conditions
in Lemma \ref{lem:support_one}. Hence, by the definition of the boundary value map
explained in Subsection \ref{subsec:c-d-boundary-value-map}, we have
\begin{equation}
\IL(f) = 
\sum_{\lambda \in \Lambda} b_{\Omega_\lambda}\left(
\ilc
\int_{\gamma^*_\lambda}  f(\zeta) e^{\zeta z} d\zeta\right)\,\otimes\, \nu_{\rcM}
\in \HQDC{n}{n}{\mathcal{V}_\rcM},
\end{equation}
where 
$
\Omega_\lambda :=  M \widehat{\times} \sqrt{-1} \Gamma_{\lambda}
$
with $\Gamma_{\lambda} := \bigcap_{k=1}^n \{y \in M;\, \langle y,\,\nu_{\lambda,k}\rangle > 0\}$.


\

Let $\Lambda = \{+1,\,-1\}$. For $\alpha = (\alpha_1,\cdots, \alpha_n) \in \Lambda^n$,
we define
\begin{equation}{\label{eq:def-Gamma-e}}
\begin{aligned}
\Gamma_\alpha &:= \{x =(x_1,\cdots,x_n) \in M;\, \alpha_k x_k > 0 \,\,(k=1,\cdots,n)\}, \\
\Gamma^*_\alpha &:= \{\eta =(\eta_1,\cdots,\eta_n) \in M^*;\, \alpha_k \eta_k > 0 \,\,(k=1,\cdots,n)\}.
\end{aligned}
\end{equation}
We denote by $+^n \in \Lambda^n$ (resp. $-^n \in \Lambda^n$) 
the multi-index in $\Lambda^n$ whose entries are all $+1$ (resp. $-1$).
Hence, $\Gamma_{+^n}$ (resp. $\Gamma^*_{+^n}$) designates the first orthant of 
$M$ (resp. $M^*$).

Let $G \ne \emptyset$ be an $\mathbb{R}_+$-conic proper closed convex subset  in $M$
and $a \in M$. Furthermore, we also assume that $G^\circ \cap M^*_\infty$ is connected
if $n = 1$. Set $K = \overline{a+G} \subset \rcM$ and let 
$f \in e^{-a\zeta} \hinfo(\HHPC{K}) = e^{-a\zeta}\hinfo(\widehat{\,\,\,\,}(\mathrm{int}\,G^\circ) \cap E^*_\infty)$.
Suppose that $G\setminus\{0\} \subset \Gamma_{+^n}$.
Then $f$ is holomorphic on $W_\psi \cap E^*$ given in \eqref{eq:domain_W} 
with $S=\HHPC{K} \cap M^*_\infty$ and $h(\xi) = a\xi$,
and it satisfies
\eqref{eq:growth_f} there. It follows from the assumption 
$G\setminus\{0\} \subset \Gamma_{+^n}$ that we can find
$a^* = (a^*_1,\cdots,a^*_n) \in M^*$ such that the open subset 
$W_\psi$ given in \eqref{eq:domain_W} 
satisfies
\begin{equation}
\overline{a^* + \Gamma^*_{+^n}} \subset W_\psi.
\end{equation}
Because of this fact, we can take a specific real $n$-chain ${\tilde{\gamma}^*} \subset E^*$ 
defined below which enjoys some good properties:
$$
\tilde{\gamma}^* := \left\{\zeta = \xi + \sqrt{-1}\eta \in E^*;\,
\eta \in M^* \setminus \{0\},\,  
\xi = a^* + \hat{\psi}(|\eta|)\,\left(\dfrac{|\eta_1|}{|\eta|},\, 
\dfrac{|\eta_2|}{|\eta|},\,\dots\,,\dfrac{|\eta_n|}{|\eta|}\right)  \right\},
$$
where $\hat{\psi}(t)$ is a continuous infra-linear function on $[0,\infty)$ 
which satisfies $\hat{\psi}(0) = 0$ and $\tilde{\gamma}^* \subset W_\psi$.
Note that the orientation of $\tilde{\gamma}^*$ is the same as that of $\sqrt{-1}M^*$.
For $\alpha \in \Lambda^n$, we also define
$$
\tilde{\gamma}^*_\alpha := \left\{\zeta = \xi + \sqrt{-1}\eta \in E^*;\,
\eta \in \Gamma^*_\alpha,\,  
\xi = a^* + \hat{\psi}(|\eta|)\,\left(\dfrac{|\eta_1|}{|\eta|},\, 
\dfrac{|\eta_2|}{|\eta|},\,\dots\,,\dfrac{|\eta_n|}{|\eta|}\right)  \right\}.
$$

We can replace the chain $\gamma^*$ of $\IL_\omega$ in Definition {\ref{def:IL}}
with the above chain $\tilde{\gamma}^*$, which is guaranteed by the same proof as that
in Lemma {\ref{lem:IL-path-independent}}.
Therefore, we have obtained
\begin{lem}{\label{lem:inverse_square}}
Under the above situation, 
we can take the chain $\tilde{\gamma}^*$ 
as the chain of the Laplace inverse integral of $f$. 
In particular, we have
\begin{equation}
\IL(f) = 
\sum_{\alpha \in \Lambda^n} b_{\Omega_\alpha}\left(
\ilc
\int_{\tilde{\gamma}^*_\alpha}  f(\zeta) e^{\zeta z} d\zeta\right)\,\otimes \nu_{\rcM}
\in \HQDC{n}{n}{\mathcal{V}_\rcM},
\end{equation}
where $\Omega_\alpha := M \widehat{\times} \sqrt{-1}\Gamma_\alpha \subset \rcE$.
\end{lem}
Note that each integral
\begin{equation}{\label{eq:def_h_alpha}}
h_\alpha(z) := 
\ilc
\int_{\tilde{\gamma}^*_\alpha}  f(\zeta) e^{\zeta z} d\zeta
\end{equation}
belongs to $\hexpo(\Omega_\alpha)$.
We will now explain an advantage of this expression:
Set 
$$
\Omega := 
\widehat{\,\,\,\,}((\mathbb{C} \setminus \mathbb{R}_{\ge 0})
\times (\mathbb{C} \setminus \mathbb{R}_{\ge 0}) \times \cdots \times
(\mathbb{C} \setminus \mathbb{R}_{\ge 0})) \quad \subset \rcE.
$$
\begin{prop}{\label{prop:exp_global}}
For any $\alpha \in \Lambda^n$, the $\textrm{sgn}(\alpha)h_\alpha(z) \in \hexpo(\Omega_\alpha)$ 
analytically extends to the same holomorphic function in $\hexpo(\Omega)$.
Here we set $\textrm{sgn}(\alpha) = \alpha_1\alpha_2\cdots\alpha_n$.
\end{prop}
\begin{proof}
Let $\beta$ be the subset in $\{1,\dots,n\}$, and set
$$
\begin{aligned}
\Omega_{\alpha, \beta} &:= 
\Omega_\alpha\, \bigcap\,
\widehat{\,\,\,\,}\{z \in E;\, \textrm{Re}\,z_k < 0\,\, (k \in \beta)\}\\
&= \widehat{\,\,\,\,}\{z = x + \sqrt{-1}y \in E;\, x_k < 0\,\, (k \in \beta),\,\,\,
\alpha_j y_j > 0\,\,(j=1,2,\dots,n)\}
\end{aligned}
$$
and
$$
\widetilde{\Omega_{\alpha, \beta}} := 
\widehat{\,\,\,\,}\{z = x + \sqrt{-1}y \in E;\, x_k < 0\,\, (k \in \beta),\,\,\,
\alpha_j y_j > 0\,\, (j \notin \beta)\}.
$$
Clearly we have
$$
\Omega_{\alpha,\beta} \subset \widetilde{\Omega_{\alpha,\beta}}, \qquad
\Omega = \bigcup_{\alpha \in \Lambda^n,\, \beta \subset\{1,2,\dots,n\}} \widetilde{\Omega_{\alpha,\beta}}.
$$
Let us define the continuous function 
$\tilde{\gamma}^*_{\alpha,\beta}: [0,1] \times \Gamma^*_\alpha \to E^*$ by
$$
\tilde{\gamma}^*_{\alpha,\beta}(s, \eta)
:= \xi + \sqrt{-1}\tilde{\eta} \qquad (\eta \in \Gamma^*_\alpha,\,s \in [0,1]).
$$
Here
$$
\xi = a^* + 
\left(
	((1-\delta_{\beta,1}(s)) \hat{\psi}(\eta) + 
\delta_{\beta,1}(s)|\eta|)\dfrac{|\eta_1|}{|\eta|},\, 
\,\dots\,,
((1-\delta_{\beta,n}(s)) \hat{\psi}(\eta) + 
\delta_{\beta,n}(s)|\eta|)\dfrac{|\eta_n|}{|\eta|}
\right)
$$
and
$$
\tilde{\eta} = \big((1-\delta_{\beta,1}(s)) \eta_1, \dots,
(1-\delta_{\beta,n}(s)) \eta_n\big),
$$
where $\delta_{\beta,k}(s) = s$ if $k \in \beta$ and 
$\delta_{\beta,k}(s) = 0$ otherwise.
Since $\tilde{\gamma}^*_{\alpha,\beta}(0,\, \Gamma^*_\alpha) = \tilde{\gamma}^*_\alpha$ holds,
we have
$$
\partial \tilde{\gamma}^*_{\alpha,\beta}([0,1],\, \Gamma^*_\alpha)  =
-\tilde{\gamma}^*_\alpha +
\tilde{\gamma}^*_{\alpha,\beta}(1,\, \Gamma^*_\alpha) 
-
\tilde{\gamma}^*_{\alpha,\beta}([0,1],\, \partial \Gamma^*_\alpha).
$$
Let $z$ be a point in $\Omega_{\alpha,\beta}$.
Then, as $f$ is holomorphic, we have
$$
0 = \int_{\tilde{\gamma}^*_{\alpha,\beta}([0,1],\,\Gamma^*_\alpha)}  
	d (f(\zeta) e^{\zeta z} d\zeta)
	= \int_{\partial \tilde{\gamma}^*_{\alpha,\beta}([0,1],\,\Gamma^*_\alpha)}  
	f(\zeta) e^{\zeta z} d\zeta,
$$
which implies
$$
\int_{\tilde{\gamma}^*_{\alpha,\beta}(1,\,\Gamma^*_\alpha)}  
	f(\zeta) e^{\zeta z} d\zeta
-
\int_{\tilde{\gamma}^*_\alpha} f(\zeta) e^{\zeta z} d\zeta
=
\int_{\tilde{\gamma}^*_{\alpha,\beta}([0,1],\,\partial \Gamma^*_\alpha)}  
	f(\zeta) e^{\zeta z} d\zeta.
$$
Note that
$$
\tilde{\gamma}^*_{\alpha,\beta}([0,1],\,\partial \Gamma^*_\alpha)
=
\bigcup_{k=1}^n
\big(\tilde{\gamma}^*_{\alpha,\beta}([0,1],\,\partial \Gamma^*_\alpha) 
\cap \{\zeta_k = a^*_k\}\big)
$$
holds. By noticing $d\zeta_k = 0$ 
on each real $n$-chain
$\tilde{\gamma}^*_{\alpha,\beta}([0,1],\partial \Gamma^*_\alpha) \cap \{\zeta_k = a^*_k\}$,
we get
$$
\int_{\tilde{\gamma}^*_{\alpha,\beta}([0,1],\,\partial \Gamma^*_\alpha)}  
f(\zeta) e^{\zeta z} d\zeta = 0,
$$
from which
$$
\int_{\tilde{\gamma}^*_\alpha} f(\zeta) e^{\zeta z} d\zeta
=
\int_{\tilde{\gamma}^*_{\alpha,\beta}(1,\,\Gamma^*_\alpha)}  
	f(\zeta) e^{\zeta z} d\zeta
$$
follows. It is easy to see that the last integral belongs to
$\hexpo(\widetilde{\Omega_{\alpha,\beta}})$.
Hence, by taking arbitrary $\beta \subset \{1,\dots,n\}$, 
we see that $\textrm{sgn}(\alpha)h_\alpha(z)$ analytically extends to 
$\displaystyle\bigcup_{\beta \subset \{1,\dots,n\}} \widetilde{\Omega_{\alpha,\beta}}$.
In particular, on $\widetilde{\Omega_{\alpha,\beta}}$ with $\beta = \{1,\dots,n\}$, i.e.,
,
$$
\widetilde{\Omega_{\alpha,\beta}} = 
\widehat{\,\,\,\,}\{z = x + \sqrt{-1}y \in E;\, x_k < 0\,\,(k=1,\dots,n)\},
$$
$\textrm{sgn}(\alpha)h_\alpha(z)$ coincides with the integration 
on the real domain 
$$
\ilc
\int_{a^* + \Gamma_{+^n}^*}  f(\xi) e^{\xi z} d\xi,
$$
which does not depend on the index $\alpha \in \Lambda^n$. Therefore, all the analytic
extensions of $\textrm{sgn}(\alpha)h_\alpha$ 
coincide on this domain, and thus, they form the
holomorphic function of exponential type on the domain
$$
\bigcup_{\alpha,\beta} \widetilde{\Omega_{\alpha,\beta}} = \Omega.
$$
This completes the proof.

\end{proof}
\section{Laplace inversion formula}

This section is devoted to proof for the Laplace inversion formula,
that is, $\LL$ and $\IL$ are mutually inverse. 
\begin{teo}{\label{laplace-formula}}
Let $G \ne \emptyset$ be an $\mathbb{R}_+$-conic proper closed convex  
subset in $M$ and $a \in M$. Set $K = \overline{a + G} \subset \rcM$.
Furthermore, we also assume that $\HHPC{K} \cap M^*_\infty$ is connected when $n=1$.
Then the Laplace transformation
$$
\LL: \Gamma_{K}(\rcM;\, \bexpo \otimes_{\hexpa} \hexpv)) \to e^{-a\zeta}\hinfo(\HHPC{K})
$$
and the inverse Laplace transformation
$$
\IL: e^{-a\zeta}\hinfo(\HHPC{K}) \to \Gamma_{K}(\rcM;\, \bexpo \otimes_{\hexpa} \hexpv))
$$
are inverse to each other.
\end{teo}
\begin{oss}
For $K$ and $G$ in the above theorem, 
\begin{equation}
	\HHPC{K} = \HHPC{G} = \widehat{\,\,\,\,}(\mathrm{int}\,G^\circ) \cap E^*_\infty
\end{equation}
hold, where $G^\circ$ is the dual cone of $G$ in $E^*$, that is,
$$
G^\circ = 
\{\zeta \in E^*;\, \operatorname{Re}\, \langle \zeta, x \rangle \ge 0
\quad (\forall x \in G)\}.
$$
\end{oss}

Thanks to Corollary {\ref{cor:laplace_h_k}} and Lemmas {\ref{lem:IL-path-independent}} 
and \ref{lem:convex-banach},
the following corollary immediately follows from Theorem {\ref{laplace-formula}}:
\begin{cor}{\label{cor:laplace-inverse}}
Let $K$ be a regular closed subset in $\rcM$ satisfying that
$K \cap M$ is convex and $\HHPC{K} \cap M^*_\infty$ is connected (in particular, non-empty). 
Then the Laplace transformation 
$$
\LL: \Gamma_{K}(\rcM;\, \bexpo\otimes_{\hexpa} \hexpv)) \to \hhinfo{h_K}(\HHPC{K})
$$
and the inverse Laplace transformation
$$
\IL: \hhinfo{h_K}(\HHPC{K}) \to \Gamma_{K}(\rcM;\, \bexpo\otimes_{\hexpa} \hexpv))
$$
are inverse to each other.
\end{cor}

\begin{lem}{\label{lem:convex-banach}}
Let $K$ be a regular closed subset in $\rcM$. 
Assume that $K$ is convex and that $\HHPC{K}$ is also non-empty.
Then we have
$$
K = \bigcap_{\xi \in \HHPC{K} \cap M^*_\infty} \overline{\{x \in M\,;\, 
\langle x, \,\xi \rangle \ge h_K(\xi)\}}.
$$
\end{lem}
\begin{proof}
If $K$ is an empty set, both the sides in the above equality 
become empty sets as $h_K(\xi) = +\infty$,
and hence, we may assume $K \ne \emptyset$.
It is enough to show that, for any $x_0 \in M$ with $x_0 \notin K$, 
there exists a hypersurface $L$ in
$M$ passing through $x_0$ such that $K$ and $\overline{L}$ are disjoint in $\rcM$.

Since $\HHPC{K}$ is not empty, we can take $\xi_0 \in \HHPC{K} \cap M^*_\infty$
and $r \in \mathbb{R}$ such that
$$
K \subset \widehat{\,\,\,\,}\{x \in M\,;\, \langle x,\, \xi_0\rangle > r\}.
$$
Set 
$$
L_{\xi_0} := \widehat{\,\,\,\,}\{x \in M\,;\, \langle x,\, \xi_0\rangle = r\}.
$$
We may assume $x_0 \in \{x \in M;\, \langle x,\,\xi_0\rangle > r\}$ from the
beginning.

Since $K \cap M$ is convex, we can find a hypersurface $L$ which separates
$x_0$ and $K$ in $M$. The claim follows if $\overline{L}$ also separates them in $\rcM$.
Hence we may assume that $\overline{L} \cap K \cap M_\infty$ is non-empty,
from which we conclude that the both normal vectors of $L$ are not in $\HHPC{K}$, and thus,
we have $\dim (L \cap L_{\xi_0}) = n -2$. 

We can take the hypersurface $\tilde{L}$ in $M$ which passes $x_0$ and $L \cap L_{\xi_0}$.
Then the hypersurface $\tilde{L}$ has the required properties, which completes the proof.
\end{proof}

\subsection{The proof for $\LL \circ \IL = \operatorname{id}$.}

Let $f \in e^{-a\zeta} \hinfo(\HHPC{K}) = e^{-a\zeta}\hinfo(\widehat{\,\,\,\,}(\mathrm{int}\,G^\circ)\cap E^*_\infty)$.
By a coordinate transformation, we may assume 
that $a = 0$ and $G \subset \Gamma_{+^n} \cup \{0\}$ from
the beginning (see \eqref{eq:def-Gamma-e} for the set $\Gamma_{+^n}$).
Let $\Lambda = \{+1,\,-1\}$, and let $h_\alpha(z)$ ($\alpha \in \Lambda^n$) be a holomorphic function defined in \eqref{eq:def_h_alpha}. Then, by Lemma {\ref{lem:inverse_square}}, we have
$$
\IL(f) = \sum_{\alpha \in \Lambda^n} b_{\Omega_\alpha}(h_\alpha(z)) \otimes \nu_{\rcM}.
$$
Note that $\operatorname{Supp}(\IL(f)) \subset \overline{G} 
\subset \widehat{\Gamma_{+^n}} \cup \{0\}$ hold.
It follows from Proposition {\ref{prop:exp_global}} that we can 
compute the Laplace transform of $\IL(f)$ by the formula given in Example 
\ref{es:l-c-example02}.
Hence, by noticing \eqref{eq:b_w_and_diag_equiv}, we have
$$
(\LL\circ \IL)(f)(\tilde{\zeta}) =
\dfrac{1}{(2\pi\sqrt{-1})^n}
\sum_{\alpha\in\Lambda^n} \operatorname{sgn}(\alpha)
\int_{\gamma_\alpha} dz 
\int_{\gamma^*_\alpha}  f(\zeta) e^{(\zeta - \tilde{\zeta}) z} d\zeta.
$$
Here we take $\epsilon > 0$ sufficiently small and $\gamma_\alpha \subset E$ is given by
$$
\left\{z = b + (B_\epsilon  + \sqrt{-1}\epsilon A_\alpha) x\,;\, x \in \Gamma_{+^n}
\right\},
$$
where the diagonal matrix
$$
A_\alpha
=
\left(\begin{matrix}
\alpha_1 &0         &            \\
0         &\alpha_2  &0         &  \\
          &           &\vdots    &        \\
	  &           &0 &\alpha_{n-1}         & 0 \\
          &           &  &0         & \alpha_n
\end{matrix}
\right),
$$
$b = -c(1,1,\dots,1) \in \Gamma_{-^n}$ 
with a sufficiently small $c > 0$ and $B_\epsilon$ is given
in Example \ref{es:l-c-example02}.
The $\gamma^*_\alpha \subset E^*$ is given by
$$
\left\{\zeta = a^* + \xi (\delta I + \sqrt{-1}A_\alpha)\,;\, \xi \in \Gamma^*_{+^n} 
\right\},
$$
where $I$ is the identity matrix, $\epsilon > \delta > 0$ and $a^* = a (1,1,\cdots,1)\in \Gamma^*_{+^n}$
for a sufficiently large $a > 0$.
Note that the orientation of $\gamma_\alpha$ and $\gamma_\alpha^*$ are determined by
those of the parameter spaces $\Gamma_{+^n}$ and $\Gamma_{+^n}^*$, respectively.
\begin{oss}{\label{rem:const-epsilon}}
The above integral does not depend on the choice of $\epsilon > 0$ if it
is sufficiently small, and we make $\epsilon$ tend to $0$ later.
\end{oss}
In what follows, we may assume that
$\tilde{\zeta} \in E^*$ is in a sufficiently small open
neighborhood of $a^* + \Gamma^*_{+^n}$ and that $|\tilde{\zeta}|$ is
large enough. As a matter of fact,
if we could show $(\LL \circ \IL)(f)(\tilde{\zeta}) = f(\tilde{\zeta})$
for such a $\tilde{\zeta}$, the claims follows from the unique continuation property
of $f$.

When $z \in \gamma_\alpha$ and $\zeta \in \gamma^*_\alpha$, we have
$$
\begin{aligned}
\operatorname{Re} (\zeta - \tilde{\zeta})z
&= - \operatorname{Re} \tilde{\zeta}z + \operatorname{Re} \zeta z \\
&=
\bigg(\langle a^* - \operatorname{Re} \tilde{\zeta},\,b + B_\epsilon x\rangle +
\epsilon\sum_{k=1}^n \alpha_k x_k\operatorname{Im} \tilde{\zeta}_k\bigg)
+ \delta \langle \xi,\, b\rangle
+ 
\bigg(\delta \langle \xi,\, B_\epsilon x\rangle
	-
\epsilon \langle \xi,\, x\rangle\bigg).
\end{aligned}
$$
Note that, for $x \in \Gamma_{+^n}$ and $\xi \in \Gamma^*_{+^n}$, we have
$$
\bigg(\delta \langle \xi,\, B_\epsilon x\rangle
	-
\epsilon \langle \xi,\, x\rangle\bigg) \le
 - \min\{\epsilon - \delta, \epsilon\delta\}|x||\xi| \le 0.
$$
Hence the above integration absolutely converges and, by the Fubini's theorem, we obtain
$$
(\LL \circ \IL)(f)(\tilde{\zeta}) =
\dfrac{1}{(2\pi\sqrt{-1})^n}
\sum_{\alpha\in\Lambda^n}
\operatorname{sgn}(\alpha)
\int_{\gamma^*_\alpha}  f(\zeta) d\zeta
\int_{\gamma_\alpha} 
e^{(\zeta - \tilde{\zeta}) z} dz.
$$
Then, if $\zeta$ is quite near $a^*$
and $\tilde{\zeta} \in E^*$ belongs to a sufficiently small open
neighborhood of $a^* + \Gamma^*_{+^n}$ and if $|\tilde{\zeta}|$ is
large enough, we get
$$
\int_{\gamma_\alpha} 
e^{(\zeta - \tilde{\zeta}) z} dz
=
\det({Q}_{\alpha,\epsilon})\int_{\Gamma_{+^n}}
e^{(\zeta - \tilde{\zeta})(b + Q_{\alpha,\epsilon} x)} dx
= \dfrac{\det(Q_{\alpha,\epsilon})\,e^{(\zeta-\tilde{\zeta})b}}
{(\tilde{\zeta}-\zeta)Q_{\alpha,\epsilon}},
$$
where
$$
\dfrac{1}{(\tilde{\zeta}-\zeta)Q_{\alpha,\epsilon}} :=
\dfrac{1}{\prod_{k=1}^n e_k (\tilde{\zeta}-\zeta)Q_{\alpha,\epsilon}}.
$$
Here $e_k$ is the unit row vector whose $k$-th entry is $1$ and
$$
Q_{\alpha,\epsilon} = B_\epsilon + \sqrt{-1}\epsilon A_\alpha.
$$
By uniqueness of the analytic continuation, the above formula holds at
any point $\zeta$ in a neighborhood of the chain $\gamma^*_\alpha$,
and hence, we have
$$
(\LL\circ \IL)(f)(\tilde{\zeta}) =
\dfrac{1}{(2\pi\sqrt{-1})^n}
\sum_{\alpha\in\Lambda^n}
\operatorname{sgn}(\alpha)\det({Q}_{\alpha,\epsilon})
\int_{\gamma^*_\alpha}  
\dfrac{f(\zeta)\,e^{(\zeta-\tilde{\zeta})b}}
{(\tilde{\zeta}-\zeta)Q_{\alpha,\epsilon}} d\zeta.
$$
Now if we could show that 
there exist $s > \delta$ and a complex open neighborhood $T \subset \mathbb{C}$ 
of $(0,s)$ such that the denominator of the integrand in the above integral
does not vanish when $\zeta \in \gamma_\alpha^*$ and $\epsilon \in T$
($\delta$ and other constants are fixed, 
where we do not keep the condition $\epsilon > \delta$ anymore),
then the above integral becomes an analytic function of $\epsilon$ ($\epsilon > 0$),
and thus, it turns out to be a constant function of $\epsilon$
due to Remark {\ref{rem:const-epsilon}}. Hence, by letting $\epsilon$ to $0$,
we have obtained
$$
(\LL\circ \IL)(f)(\tilde{\zeta}) =
\ilc
\sum_{\alpha\in\Lambda^n} \operatorname{sgn}(\alpha)
\int_{\gamma^*_\alpha}  
\dfrac{f(\zeta)\,e^{(\zeta-\tilde{\zeta})b}}
{\tilde{\zeta} - \zeta} d\zeta
$$
which is clearly equal to $f(\tilde{\zeta})$ by the Cauchy integral formula.

\

Let $g(\zeta,\eta)$ be the first element of the vector 
$\zeta Q_{\alpha,\epsilon} - \eta$, and
let us show $g(\zeta, \eta) \ne 0$ for any $\zeta \in \gamma^*_\alpha$ and
for any $\eta$ contained in a sufficiently small neighborhood
of the point $R(1,1,\cdots,1)Q_{\alpha,\epsilon}$ with a sufficiently large $R > 0$.
Set $\epsilon = \epsilon' + \sqrt{-1}\epsilon''$ for a sufficiently small $\epsilon' > 0$ and $\epsilon'' \in \mathbb{R}$ with $|\epsilon''| < \delta\epsilon'/2$.
The real part of $g(\zeta,\eta)$ is, for $\zeta = a^* + \xi (\delta E + \sqrt{-1}A_\alpha)$ with
$\xi \in \overline{\Gamma_{+^n}}$,
$$
(\delta-\epsilon'-\alpha_1\epsilon'')\xi_1 - 
((\delta\epsilon'-\alpha_2\epsilon'')\xi_2+ \cdots+
(\delta\epsilon'-\alpha_n\epsilon'')\xi_n) + (1 - (n-1)\epsilon'-\alpha_1\epsilon'')a
- \operatorname{Re}\,\eta_1
$$
and its imaginary part is
$$
\alpha_1 ((1+\epsilon' -\alpha_1\epsilon'') \xi_1  + (\epsilon' -(n-1)\alpha_1\epsilon'') a)
- 
((\alpha_2\epsilon' + \delta\epsilon'')\xi_2 + \cdots+ (\alpha_n\epsilon' + \delta\epsilon'')\xi_n) - \operatorname{Im}\,\eta_1.
$$
If $\operatorname{Re} g(\zeta,\eta) = 0$, then we have
$$
(\delta\epsilon'- \alpha_2\epsilon'')\xi_2+\cdots+
(\delta\epsilon'-\alpha_n\epsilon'')\xi_n = 
(\delta-\epsilon'-\alpha_1\epsilon'')\xi_1 - 
(\operatorname{Re} \eta_1 - (1-(n-1)\epsilon'-\alpha_1\epsilon'')a),
$$
which gives the estimate
$$
(\delta\epsilon'-|\epsilon''|)(\xi_2+\cdots+\xi_n) \le
(\delta-(\epsilon'-|\epsilon''|))\xi_1 - (\operatorname{Re} \eta_1 - a),
$$
that is, we have obtained
$$
(\epsilon'- \delta^{-1}|\epsilon''|)(\xi_2+\cdots+\xi_n) \le
(1- \delta^{-1}(\epsilon'-|\epsilon''|))\xi_1 - \delta^{-1}(\operatorname{Re} \eta_1 - a).
$$
Hence, when $\operatorname{Re} g(\zeta,\eta) = 0$, we get
$$
\begin{aligned}
|\operatorname{Im} g(\zeta,\eta)|
& \ge (1+\epsilon' - |\epsilon''|)\xi_1  - 
(\epsilon'+\delta|\epsilon''|)(\xi_2 + \cdots + \xi_n) 
- |\operatorname{Im} \eta_1| \\
&\ge
\ell(\epsilon',\epsilon'')\xi_1 +
\delta^{-1}
\left(\dfrac{\epsilon'+\delta|\epsilon''|}{\epsilon' - \delta^{-1}|\epsilon''|}\right)
(\operatorname{Re} \eta_1 - a) - |\operatorname{Im}\eta_1|,
\end{aligned}
$$
where
$$
\ell(\epsilon',\epsilon'') =
\left(1+\epsilon'-|\epsilon''|\right) - 
\left(\dfrac{\epsilon'+\delta|\epsilon''|}{\epsilon' - \delta^{-1}|\epsilon''|}\right)
\left(1 - \delta^{-1}(\epsilon' - |\epsilon''|)\right).
$$
Note that, for each $\epsilon' > 0$, we have $\ell(\epsilon',\epsilon'') > 0$
if $|\epsilon''|$ is sufficiently small. In what follows,
we consider the case for such an $\epsilon = \epsilon' + \sqrt{-1}\epsilon''$.
When $\eta$ is contained in a sufficiently small neighborhood
of the point $R(1,1,\cdots,1)Q_{\alpha,\epsilon}$,
we have
$$
\operatorname{Re}\,\eta_1 \sim R - ((n-1)\epsilon'+\alpha_1\epsilon'')R,\qquad
\operatorname{Im}\,\eta_1 \sim (\alpha_1\epsilon' - (n-1)\epsilon'')R.
$$
Hence, if $R$ is sufficiently large, $\operatorname{Im} g(\zeta,\eta)$ never becomes zero.
This completes the proof.

\subsection{The proof for $\IL \circ \LL = \operatorname{id}$.}

Let $G$ be an $\mathbb{R}_+$-conic proper closed convex subset
in $M$ and $a \in M$. Set $K= \overline{a + G} \subset \rcM$.
Then we take an open convex cone $V \subset \rcE$ containing $K$.
Let $u \in \Gamma_K(\rcM;\, \hexpo \otimes \hexpv)$ with
a representative 
$\nu = (\nu_1,\,\nu_{01}) \in \QDC{n}{n}{\mathcal{V}_K}$.
We will show $(\IL \circ \LL)(u) = u$. By a coordinate transformation,
we may assume $a = 0$ and
$$
\overline{G} \setminus \{0\} \,\subset\, \widehat{\Gamma_{+^n}}
\,\subset\,
\overline{\Gamma_{+^n}} \,\subset\, V
$$
from the beginning. Then, 
it follows from Lemma \ref{lem:inverse_square} that we get
$$
(\IL \circ \LL)(u) =
\ilc
\sum_{\alpha \in \Lambda^n} 
b_{\Omega_\alpha}\,\left(\int_{\tilde{\gamma}^*_\alpha} 
\LL(u)\, e^{\zeta \tilde{z}} d\zeta\right).
$$
Set
$$
g_\alpha(\tilde{z}) :=\int_{\tilde{\gamma}^*_\alpha} \LL(u)\, e^{\zeta \tilde{z}} d\zeta.
$$
It follows from Proposition {\ref{prop:exp_global}} that
$g_\alpha$ extends to a holomorphic function on $\Omega$ of exponential type.
Here
$$
\Omega = 
\widehat{\,\,\,\,}((\mathbb{C} \setminus \mathbb{R}_{\ge 0})
\times (\mathbb{C} \setminus \mathbb{R}_{\ge 0}) \times \cdots \times
(\mathbb{C} \setminus \mathbb{R}_{\ge 0})) \quad \subset \rcE.
$$
We first consider $g_\alpha(\tilde{z})$ at a point in 
$\Gamma_{-^n} \times \sqrt{-1}\Gamma_\alpha$.
Let us take $\tilde{z}$ in $\Gamma_{-^n} \times \sqrt{-1}\Gamma_\alpha$ and fix it.
Then, at this $\tilde{z}$, we can deform the $n$-chain $\tilde{\gamma}^*_\alpha$ to
$$
\left\{\zeta = \xi + \sqrt{-1}\eta' \in E^*;\, 
\begin{aligned}	
&\eta'_k = \alpha_k \eta_k\,\, (k=1,\dots,n), \\
&\xi = a^* + \epsilon' \eta, \quad \eta \in \Gamma^*_{+^n} 
\end{aligned}
\right\}
$$
with $a^* \in \Gamma^*_{+^n}$ and $\epsilon' > 0$.
Here the orientation of the modified chain $\tilde{\gamma}^*_\alpha$ is
the same as the original one and we assume $|a^*|$ to be sufficiently large.

Now, since $\zeta$ runs in $\tilde{\gamma}^*_\alpha$, the real $2n$-chain $D$ of
the integration 
$$
\LL(u)(\zeta) := \int_D e^{-z \zeta} \nu_1(z) - \int_{\partial D} e^{-z \zeta}\nu_{01}(z),
$$
can be taken so that it is an open subset in $\rcE$ with a good boundary such that 
$$
\overline{G} \subset D \subset \widehat{\,\,\,\,}\{z = x + \sqrt{-1}y;\, x \in b + \Gamma',\,
|y| < \epsilon \operatorname{dist}(x,\, M \setminus (b+ \Gamma'))\},
$$
where $b = -\epsilon\,(1,\dots,1)$ and ${\Gamma'} \subset M$
is an $\mathbb{R}_+$-conic open convex cone such that
$$
\overline{G}\setminus \{0\} \,\subset\, \widehat{\Gamma'} \,\subset\,
\overline{\Gamma'} \setminus\{0\} \,\subset\, \widehat{\Gamma_{+^n}}.
$$
Note that, if $\zeta = \xi + \sqrt{-1}\eta\in \tilde{\gamma}^*_\alpha$ 
and $z = x+ \sqrt{-1}y \in D \cap E$, we have
$$
\operatorname{Re} (\tilde{z} - z)\zeta  
= (\operatorname{Re}\tilde{z} - x) a^* - 
\sum_{k=1}^n \alpha_k(\operatorname{Im}\tilde{z}_k)\eta_k 
+ 
\bigg(\epsilon'(\operatorname{Re}\tilde{z} - x) \eta + 
\sum_{k=1}^n \alpha_k y_k \eta_k\bigg).
$$
If $\epsilon>0$ is sufficiently small and $\operatorname{Re}\tilde{z} \in b + \Gamma_{-^n}$, by noticing
$$
|y| < \epsilon (|b| + |x|)\qquad (z = x+\sqrt{-1}y \in D \cap E),
$$
we can easily see that, for any $z = x + \sqrt{-1}y \in D \cap E$ and
$\zeta = \xi + \sqrt{-1}\eta \in \tilde{\gamma}^*_\alpha$,
$$
\epsilon'(\operatorname{Re}\tilde{z} - x) \eta + 
\sum_{k=1}^n \alpha_k y_k \eta_k = 
\bigg(\epsilon'(\operatorname{Re}\tilde{z} - x) + 
(\alpha_1y_1,\dots,\alpha_n y_n)\bigg)\eta 
\le 0
$$
holds.  Hence, the double integral in $g_\alpha$ absolutely converges and
we can apply Fubini's theorem to $g_\alpha$, from which we get
$$
g_\alpha(\tilde{z}) =
\int_D \nu_1(z)
\int_{\tilde{\gamma}^*_\alpha} 
e^{(\tilde{z} -z) \zeta} d\zeta
- 
\int_{\partial D} \nu_{01}(z)
\int_{\tilde{\gamma}^*_\alpha} 
e^{(\tilde{z} -z) \zeta} d\zeta.
$$
Now let us consider the integral
$
\displaystyle\int_{\tilde{\gamma}^*_\alpha} 
e^{(\tilde{z} -z) \zeta} d\zeta.
$
If $\operatorname{Re}(\tilde{z} - z) \in \Gamma_{-^n}$
and $|\operatorname{Im}(\tilde{z} - z)|$ is sufficiently small, then
we can deform the $n$-chain to the one in $M^*$ as was done 
in the proof of Proposition {\ref{prop:exp_global}},
we have
$$
\displaystyle\int_{\tilde{\gamma}^*_\alpha} 
e^{(\tilde{z} -z) \zeta} d\zeta
= 
\textrm{sgn}(\alpha) \displaystyle\int_{a^* + \Gamma^*_{+^n}}
e^{(\tilde{z} -z) \xi} d\xi
= \textrm{sgn}(\alpha)\dfrac{e^{(\tilde{z} -z)a^*}}{z- \tilde{z}},
$$
where $\operatorname{sgn}(\alpha) = \alpha_1 \alpha_2 \cdots \alpha_n$.
Note that, by the unique continuation property, 
$$
\displaystyle\int_{\tilde{\gamma}^*_\alpha} 
e^{(\tilde{z} -z) \zeta} d\zeta
= \textrm{sgn}(\alpha) \dfrac{e^{(\tilde{z} -z)a^*}}{z- \tilde{z}}
$$
holds at a point where the integral is defined. Summing up, we have obtained
$$
g_\alpha(\tilde{z}) =
\operatorname{sgn}(\alpha)
\left(
\int_D \dfrac{e^{(\tilde{z} -z) a^*} \nu_1(z)}{z - \tilde{z}} - 
\int_{\partial D} \dfrac{e^{(\tilde{z} - z) a^*} \nu_{01}(z)}{z - \tilde{z}}
\right)
$$
if $\tilde{z} \in \Gamma_{-^n} \times \sqrt{-1}\Gamma_\alpha$.
By deforming $D$ appropriately, we see that 
the integrals in the right-hand side converge on  $M \times \sqrt{-1}\Gamma_\alpha$, 
and hence, the above equation also holds there.
It follows from Theorem {\ref{teo:b_k_invserse}}
that we have
$$
\ilc
\sum_{\alpha\in\Lambda^n}
\operatorname{sgn}(\alpha)\,
b_{\Omega_\alpha}\left(
\int_D \dfrac{e^{(\tilde{z} -z) a^*} \nu_1(z)}{z- \tilde{z}} - 
\int_{\partial D} \dfrac{e^{(\tilde{z} - z) a^*} \nu_{01}(z)}{z- \tilde{z}}
\right)
= [\nu] = u.
$$
This completes the proof.

\section{Application to PDE with constant coefficients}
Let $\mathfrak{R}$ be the polynomial ring $\mathbb{C}[\zeta_1,\cdots, \zeta_n]$ on $E^*$
and $\mathfrak{D}$ the ring $\mathbb{C}[\partial_{x_1}, \cdots, \partial_{x_n}]$ 
of linear differential operators on $M$ with constant coefficients. 
We denote by $\sigma$ the principal symbol map from $\mathfrak{D}$ to $\mathfrak{R}$, that is,
$$
\mathfrak{D} \ni P(\partial) = \sum c_\alpha \partial^\alpha
\mapsto \sigma(P)(\zeta) = \sum_{|\alpha|=\textrm{ord}(P)} c_\alpha \zeta^\alpha
\in \mathfrak{R}.
$$
For an $\mathfrak{D}$-module $\mathfrak{M} = \mathfrak{D}/\mathfrak{I}$ with
the ideal $\mathfrak{I} \subset \mathfrak{D}$, 
we define the closed subset in $E^*_\infty$
$$
\textrm{Char}_{E_\infty^*}(\mathfrak{M}) = \{\zeta \in E_\infty^*;\, \sigma(P)(\zeta) = 0\quad(\forall\,P \in \mathfrak{I})\}.
$$
Here we identify a point in $E^*_\infty$ with a unit vector in $E^*$.

Recall that $\{f_1,\cdots,f_\ell\}$ ($f_k \in \mathfrak{R}$) is said to be 
a regular sequence over $\mathfrak{R}$ if and only if the conditions below are satisfied:
\begin{enumerate}
\item $(f_1,\cdots,f_\ell) \ne \mathfrak{R}$.
\item For any $k=1,2,\cdots,\ell$, the $f_k$ is not a zero divisor on $\mathfrak{R}/(f_1, \cdots, f_{k-1})$.
\end{enumerate}
The following theorem is fundamental in the theory of operational calculus:
Let $P_1(\partial)$, $\cdots$, $P_\ell(\partial)$ be in
$\mathfrak{D}$, and
define the $\mathfrak{D}$-module
$$
\mathfrak{M} = \mathfrak{D}/(P_1(\partial),\cdots,P_\ell(\partial)).
$$
\begin{teo}
Let $K$ be a regular closed subset in $\rcM$.
Assume that $K \cap M$ is convex and $\HHPC{K} \cap M^*_\infty$ is connected (in particular, it is non-empty),
and that $P_1(\zeta)$, $\cdots$, $P_\ell(\zeta)$ form a regular sequence
over $\mathfrak{R}$. 
Then the condition 
$$
\HHPC{K} \cap \textrm{Char}_{E^*_\infty}(\mathfrak{M}) = \emptyset
$$ 
implies
$$
\textrm{Ext}^k_{\mathfrak{D}}(\mathfrak{M},\, \Gamma_K(\rcM,\bexpo)) = 0
\qquad (k=0,1).
$$
\end{teo}
\begin{proof}
	Let $\mathcal{F}$ be a sheaf of $\mathbb{Z}$-modules or a $\mathbb{Z}$-module itself and $s_i: \mathcal{F} \to \mathcal{F}$
($i=1,\cdots,\ell$) a morphism such that $s_i \circ s_j = s_j \circ s_i$ holds
for $1 \le i,j \le \ell$. Then we denote by $K(s_1,\cdots,s_\ell;\,\mathcal{F})$ 
the Koszul complex associated to $(s_1,\cdots, s_\ell)$ with coefficients in $\mathcal{F}$. That is, 
$$
{0}
\to 
\overset{\text{$0$-th degree}}{\mathcal{F} \otimes (\overset{0}{\wedge}\Lambda)} 
\xrightarrow{d} 
\mathcal{F}\otimes (\overset{1}{\wedge} \Lambda)
\xrightarrow{d} \mathcal{F}\otimes (\overset{2}{\wedge}\Lambda)
\xrightarrow{d} \cdots 
\xrightarrow{d} \mathcal{F}\otimes (\overset{\ell}{\wedge}\Lambda)
\to {0},
$$
where $\Lambda$ is a free $\mathbb{Z}$-module of rank $\ell$ with basis $e_1$, $e_2$, $\cdots$,
$e_\ell$ and
$$
d(f \otimes e_{i_1} \wedge e_{i_2} \wedge \cdots \wedge e_{i_k})
= \sum_{j=1}^\ell s_j(f) \otimes e_j \wedge e_{i_1} \wedge e_{i_2} \wedge \cdots \wedge e_{i_k}.
$$
Since $P_1(\zeta)$, $\cdots$, $P_\ell(\zeta)$ form a regular sequence, 
the complex $K(P_1(\partial),\cdots, P_\ell(\partial);\, \mathfrak{D})[\ell]$ is a free resolution
of $\mathfrak{M}$ and we get
$$
\textbf{R}\textrm{Hom}_{\mathfrak{D}}(\mathfrak{M},\,\Gamma_K(\rcM;\,\bexpo))
\simeq
K(P_1(\partial),\cdots,P_\ell(\partial);\, \Gamma_K(\rcM;\,\bexpo)).
$$
Hence it follows from Corollary {\ref{cor:laplace-inverse}} that
we have
\begin{equation}{\label{eq:resol-m}}
\textbf{R}\textrm{Hom}_{\mathfrak{D}}(\mathfrak{M},\,\Gamma_K(\rcM;\,\bexpo)) \simeq
K(P_1(\zeta),\cdots,P_\ell(\zeta);\, \Gamma(\HHPC{K};\,\hhinfo{h_K})).
\end{equation}

The lemma below is a key for the theorem:
\begin{lem}
Let $\zeta^* \notin \textrm{Char}_{E^*_\infty}(\mathfrak{M})$.
Then the Koszul complex
\begin{equation}{\label{eq:koszul_at_zeta}}
K(P_1(\zeta),\cdots,P_\ell(\zeta);\, (\hhinfo{h_K})_{\zeta^*})
\end{equation}
is exact. 
\end{lem}
\begin{proof}
By the definition of $\textrm{Char}_{E^*_\infty}(\mathfrak{M})$, 
we can find $h(\zeta)$ and $a_j(\zeta)$ ($j=1,2,\cdots,\ell)$ in $\mathfrak{R}$ such that
$$
\sigma(h)(\zeta^*) \ne 0,\qquad
h(\zeta) = \sum_{j=1}^\ell a_j(\zeta)P_j(\zeta).
$$
In particular, as $\sigma(h)(\zeta^*) \ne 0$ holds, $h$ is also invertible in the germ 
$({\hhinfo{h_K}})_{\zeta^*}$ of 
the sheaf ${\hhinfo{h_K}}$ at $\zeta^*$.
Set $\Lambda = \{1,2,\cdots,\ell\}$ and
let $s  = \{s_k\}$ be a homotopy from
$K(P_1,\cdots, P_\ell;\, (\hhinfo{h_K})_{\zeta^*})$ to itself. Here
$$
s_k: K^{k+1}(P_1,\cdots, P_\ell;\, (\hhinfo{h_K})_{\zeta^*})
\to
K^{k}(P_1,\cdots, P_\ell;\, (\hhinfo{h_K})_{\zeta^*})
$$
is given by
$$
s_k(\sum_{\alpha \in \Lambda^{k+1}} f_\alpha(\zeta) e_\alpha) = 
\sum_{\beta \in \Lambda^k} \sum_{j=1}^{\ell} a_j(\zeta)f_{j\,\beta}(\zeta)e_{\beta},
$$
where $e_\alpha = e_{\alpha_1}\wedge \cdots \wedge e_{\alpha_{k+1}}$ and
$j\,\beta $ is a sequence such that $\beta$ follows $j$.
Then, by the simple computation, we can easily get the equality
$$
s \circ d - d \circ s = h.
$$
Here $h: K(P_1,\cdots, P_\ell;\, (\hhinfo{h_K})_{\zeta^*})
\to
K(P_1,\cdots, P_\ell;\, (\hhinfo{h_K})_{\zeta^*})$
is the morphism of complexes defined by
$$
K^k(P_1,\cdots, P_\ell;\, (\hhinfo{h_K})_{\zeta^*}) \ni u
\mapsto hu \in K^k(P_1,\cdots, P_\ell;\, (\hhinfo{h_K})_{\zeta^*}),
$$
which is an isomorphism because $h(\zeta)$ is invertible on $(\hhinfo{h_K})_{\zeta^*}$.
Therefore the isomorphism $h$ is homotopic to zero, 
from which we conclude that
the complex \eqref{eq:koszul_at_zeta} is
quasi-isomorphic to zero. This completes the proof of the lemma.
\end{proof}

It follows from the lemma that the Koszul complex
\begin{equation}{\label{eq:koszul_p}}
K(P_1(\zeta),\cdots,P_\ell(\zeta);\, \hhinfo{h_K})
\end{equation}
of sheaves is exact on $\HHPC{K}$ because of the condition 
$\HHPC{K} \cap \textrm{Char}_{E^*_\infty}(\mathfrak{M}) = \emptyset$. 
Applying the left exact functor $\Gamma(\HHPC{K};\,\bullet)$ to the complex 
{\eqref{eq:koszul_p}}, we get a short exact sequence
$$
0 \to
 \Gamma(\HHPC{K};\,\hhinfo{h_K})
\to
\Gamma(\HHPC{K};\,\hhinfo{h_K}) \otimes (\overset{1}{\wedge}\Lambda)
\to
\Gamma(\HHPC{K};\,\hhinfo{h_K}) \otimes (\overset{2}{\wedge}\Lambda).
$$
Then, by noticing {\eqref{eq:resol-m}},
the claim follows from the above short exact sequence.
\end{proof}
\begin{cor}
Let $P(\partial) \in \mathfrak{D}$, and let
$K$ be a regular closed subset in $\rcM$ satisfying that
$K \cap M$ is convex and $\HHPC{K} \cap M^*_\infty$ is connected (in particular, non-empty).
Then the morphism
$$
\Gamma_K(\rcM,\bexpo) \xrightarrow{\,\,P(\partial)\, \bullet\,\,}
\Gamma_K(\rcM,\bexpo) 
$$
becomes isomorphic if $\sigma(P)(\zeta) \ne 0$ holds for any $\zeta \in \HHPC{K}$.
\end{cor}

\end{document}